\documentclass[11pt]{article}
\usepackage[utf8]{inputenc}
\usepackage{lmodern}
\usepackage{subfiles}
\usepackage{enumitem}
	\setenumerate{label={\normalfont(\roman*)}, itemsep=0em} 
        
\usepackage{amsfonts}
\usepackage{amsthm}
\usepackage{amsmath}
\usepackage{amssymb}
\usepackage{amscd}
\usepackage{thmtools, thm-restate}
\usepackage{mathrsfs}
\usepackage{mathtools}
\usepackage{bm}
\usepackage{esint}

\usepackage[margin=3cm]{geometry}
\usepackage{indentfirst}
\usepackage{graphicx}
\usepackage{graphics}
\usepackage{lscape}
\usepackage{tikz-cd}
\usepackage{tikz,pgf}
	\usetikzlibrary{arrows}
\usepackage{color}
\usepackage{pict2e}
\usepackage{epic}
\usepackage{epstopdf}
\usepackage{titlesec, titlefoot}
	\titleformat{\section}[block]{\Large\bfseries\filcenter}{\thesection}{1em}{}
\usepackage{commath}
\usepackage{float}
\usepackage{caption}
\usepackage{subcaption}
\usepackage{etoolbox}
\usepackage[affil-it]{authblk}
\usepackage{combelow}
\usepackage{bbm}

\usepackage[hidelinks]{hyperref}
\hypersetup{bookmarksopen=true} 
\usepackage{hypcap}

\graphicspath{{./Pictures/}}
\allowdisplaybreaks

\usepackage{listings}
\usepackage{fancybox}

\theoremstyle{plain}

\renewcommand*\thesection{\arabic{section}}
\numberwithin{equation}{section} 

\theoremstyle{plain}
\newtheorem{thm}{Theorem}
\numberwithin{thm}{section} 

\newtheorem{lemma}[thm]{Lemma}

\newtheorem{corollary}[thm]{Corollary}
\newtheorem{theorem}[thm]{Theorem}

\theoremstyle{definition}

\newtheorem{remark}[thm]{Remark}

\newtheorem{definition}[thm]{Definition}

\newcommand{\thistheoremname}{}
\newtheorem{genericthm}[equation]{\thistheoremname}

\newcommand{\thistheoremnames}{}
\newtheorem*{genericthms}{\thistheoremnames}
\newenvironment{para*}[1]
  {\renewcommand{\thistheoremnames}{#1}%
   \begin{genericthms}}
  {\end{genericthms}}

\expandafter\let\expandafter\oldproof\csname\string\proof\endcsname
\let\oldendproof\endproof
\renewenvironment{proof}[1][\proofname]{%
  \oldproof[\upshape \bfseries #1]%
}%
{\oldendproof}

\makeatletter
\def\@makechapterhead#1{%
  \vspace*{50\p@}%
  {\parindent \z@ \raggedright \normalfont
    \interlinepenalty\@M
    \Huge\bfseries  \thechapter.\quad #1\par\nobreak
    \vskip 40\p@
  }}
\makeatother

\newcommand{\reqnomode}{\tagsleft@false}
\def \d{\,{\rm d}}
\def\dist{\,{\rm dist}}
\def\supp{\,{\rm supp }}
\def\diam{\,{\rm diam}}

\allowdisplaybreaks
\makeatletter
\DeclareRobustCommand*{\bfseries}{%
  \not@math@alphabet\bfseries\mathbf
  \fontseries\bfdefault\selectfont
  \boldmath
}

\makeatother

\newlength{\defbaselineskip}
\newcommand{\N}{\mathbb{N}}
\newcommand\eps\varepsilon
\def\mean#1{\mathchoice%
          {\mathop{\kern 0.2em\vrule width 0.6em height 0.69678ex depth -0.58065ex
                  \kern -0.8em \intop}\nolimits_{\kern -0.4em#1}}%
          {\mathop{\kern 0.1em\vrule width 0.5em height 0.69678ex depth -0.60387ex
                  \kern -0.6em \intop}\nolimits_{#1}}%
          {\mathop{\kern 0.1em\vrule width 0.5em height 0.69678ex
              depth -0.60387ex
                  \kern -0.6em \intop}\nolimits_{#1}}%
          {\mathop{\kern 0.1em\vrule width 0.5em height 0.69678ex depth -0.60387ex
                  \kern -0.6em \intop}\nolimits_{#1}}}

\numberwithin{equation}{section}
\allowdisplaybreaks[4]

\newcommand\blfootnote[1]{%
  \begingroup
  \renewcommand\thefootnote{}\footnote{#1}%
  \addtocounter{footnote}{-1}%
  \endgroup
}

\def\loc{\operatorname{loc}}

\def\eqn#1$$#2$${\begin{equation}\label#1#2\end{equation}}
\newcommand\R{\mathbb{R}}

\newcommand{\F}{\mathscr F}

\def \tp{\textup}
\def \p{\partial}
\def \e{\varepsilon}
\def \H{\mathscr H}
\def \D{\mathrm D}

\newcommand{\bb}[1]{\mathbb{#1}}
\def \LL {\mathrm L}
\def \WW{\mathrm W}
\def \BB{\mathrm B}

\newcommand\restr[2]{{
  \left.\kern-\nulldelimiterspace 
  #1 
  \vphantom{|} 
  \right|_{#2} 
  }
}

\def\Xint#1{\,\mathchoice
  {\XXint\displaystyle\textstyle{#1}}%
  {\XXint\textstyle\scriptstyle{#1}}%
  {\XXint\scriptstyle\scriptscriptstyle{#1}}%
  {\XXint\scriptscriptstyle\scriptscriptstyle{#1}}%
  \!\int}
\def\XXint#1#2#3{\setbox0=\hbox{$#1{#2#3}{\int}$}\vcenter{\hbox{$#2#3$}}\kern-.5\wd0}

\def\dashint{\Xint-}

\let\norm\relax
\newcommand{\norm}[1]{\left\lVert#1\right\rVert}
\newcommand{\seminorm}[1]{\left[#1\right]}

\AtBeginDocument{
  \let\div\relax
  \DeclareMathOperator{\div}{div}
}

\mathtoolsset{showonlyrefs}

\usetikzlibrary{backgrounds}
\usetikzlibrary{arrows,arrows.meta}
\usetikzlibrary{shapes,shapes.geometric,shapes.misc}

\pgfdeclarelayer{edgelayer}
\pgfdeclarelayer{nodelayer}
\pgfsetlayers{background,edgelayer,nodelayer,main}
\tikzstyle{none}=[inner sep=0mm]

\title{Boundary regularity results for minimisers of convex functionals with $(p,q)$-growth}
\author[1]{Christopher Irving}
\author[2]{Lukas Koch}

\affil[1]{\small Faculty of Mathematics, TU Dortmund University, Vogelpothsweg 87, 44227 Dortmund, Germany 
\protect \\
  {\tt{christopher.irving@tu-dortmund.de}}
  \vspace{1em} \ }
  
  \affil[2]{\small MPI for Mathematics in the Sciences, Inselstrasse 22, 04177 Leipzig, Germany
\protect \\
  {\tt{kochl@mis.mpg.de}}
  \vspace{1em} \ }

\usepackage{etoolbox}
\makeatletter
\patchcmd{\@adjustvertspacing}
  {\jot\baselineskip \divide\jot 4}
  {\jot=.5\baselineskip}
  {}{}
\makeatother

\makeatother
\begin{document}
\maketitle

\begin{abstract}
We prove improved differentiability results for relaxed minimisers of vectorial convex functionals with $(p,q)$-growth, satisfying a H\"older-growth condition in $x$. We consider both Dirichlet and Neumann boundary data. In addition, we obtain a characterisation of regular boundary points for such minimisers. In particular, in case of homogeneous boundary conditions, this allows us to deduce partial boundary regularity of relaxed minimisers on smooth domains for radial integrands. We also obtain some partial boundary regularity results for non-homogeneous Neumann boundary conditions.
\end{abstract}



\section{Introduction and results}
We study minimisation problems of the form
\begin{align*}\label{eq:neumann_problem}
  &\min_{u\in \WW^{1,p}(\Omega,\R^m)} \F_N(u) \ \text{ where } \\
  &\quad \F_N(u)=\int_\Omega F(x,\D u)- f\cdot u\,\d x+\int_{\p\Omega}g_N\cdot u\,\d\H^{n-1},\tag{N}
\end{align*}
and
\begin{align*}\label{eq:dirichlet_problem}
  \min_{u\in \WW^{1,p}_{g_D}(\Omega,\R^m)} \F_D(u) \ \text{ where }
\F_D(u)=\int_\Omega F(x,\D u)-f\cdot u\,\d x.\tag{D}
\end{align*}
Here $f$, $g_N$ and $g_D$ are sufficiently regular data.
In the Neumann case \eqref{eq:neumann_problem} we will assume the compatibility condition
\begin{equation}\label{eq:compatibility_condition}
  \int_\Omega f \,\d x = \int_{\p\Omega} g_N \,\d \H^{n-1},
\end{equation} 
where the integral is to be understood as being applied componentwise.

In order to comment on our results, we state our assumptions on $F$ precisely. We remark that we explain our notation in Section \ref{sec:prelim}.

\textit{Let $0<\alpha\leq 1$, $n \geq 2,$ $m \geq 1,$ and let $\Omega \subset \bb R^n$ be a bounded Lipschitz domain. We assume that 
  \begin{equation}
    F = F(x,z) \colon \overline\Omega \times \bb R^{m \times n} \to \bb R
  \end{equation} 
  is measurable in $x$, continuously differentiable in $z$ and moreover satisfies natural growth conditions and a H\"older-continuity assumption in $x$ of the form
\begin{gather}\label{def:bounds1}
  z \mapsto F(x,z) - \nu\lvert V_{\mu,p}(z)\rvert^2 \text{ is convex},\tag{H1}\\
\label{def:bounds2}
\lvert F(x,z)\rvert\leq \Lambda (1+\lvert z\rvert^2)^\frac q 2,\tag{H2}\\
\label{def:bounds3} \lvert F(x,z)-F(y,z)\rvert\leq \Lambda \lvert x-y\rvert^\alpha\left(1+ \lvert z\rvert^2\right)^\frac q 2,\tag{H3}
\end{gather}
for some $\mu, \nu, \Lambda>0,$ $\alpha \in (0,1]$, all $z,w\in \R^{m\times n}$ and almost every $x,y\in \Omega$, where $1< p\leq q$.}
\textit{At times we additionally assume 
\begin{equation}
  \label{def:bounds31} \lvert \p_zF(x,z)-\p_zF(y,z)\rvert\leq \Lambda \lvert x-y\rvert^\alpha\left(1+ \lvert z\rvert^2\right)^\frac {q-1} 2\tag{H4}.   
  \end{equation} 
}
In particular, $F(x,z)$ is convex.
Further properties of such integrands will be discussed in Section \ref{sec:bounds_int}.


It is well-known that a major obstruction to regularity in the setting of $(p,q)$-growth is the possible occurrence of the Lavrentiev phenomenon. This describes the possibility that
\begin{align}\label{eq:Lphenomenon}
\inf_{u\in W^{1,p}_g(\Omega)} \F(u)< \inf_{u\in W^{1,q}_g(\Omega)} \F(u).
\end{align}
This phenomenon was first described in \cite{Lavrentiev1926}. In the context of $(p,q)$-growth functionals the theory was expanded in \cite{Zhikov1987}, \cite{Zhikov1993}, \cite{Zhikov1995}. Adapting the viewpoint and terminology of \cite{Buttazo1992}, we will not deal with pointwise minimisers of \eqref{eq:neumann_problem} and \eqref{eq:dirichlet_problem}, but rather with relaxed minimisers in the following sense:
\begin{definition}\label{def:relaxedMinimiser}
We say $u\in \WW^{1,p}(\Omega)$ is a \textbf{$\WW^{1,q}$-relaxed minimiser} (usually referred to as a relaxed minimiser) of $\F_N$ if $u$ minimises the relaxed functional
\begin{align}\label{def:relaxedFunctional}
\overline \F_N(v) = \inf \left\{\,\liminf_{j\to\infty} \F(v_j): (v_j)\subset Y, v_j\rightharpoonup v \text{ weakly in } X\,\right\}
\end{align}
amongst all $v\in X= \WW^{1,p}(\Omega)$ where $Y=\WW^{1,q}(\Omega)$. 
Similarly, we say $u\in \WW^{1,p}_{g_D}(\Omega)$ is a \textbf{$\WW^{1,q}$-relaxed minimiser} of $\F_D$, if $u$ minimises the relaxed functional
\begin{align}\label{def:relaxedFunctionalDirichlet}
\overline \F_D(v) = \inf \left\{\,\liminf_{j\to\infty} \F(v_j): (v_j)\subset Y, v_j\rightharpoonup v \text{ weakly in } X\,\right\}
\end{align}
amongst all $v\in X= \WW^{1,p}_{g_D}(\Omega)$ where $Y=\WW^{1,q}(\Omega)\cap \WW^{1,p}_{g_D}(\Omega)$. 
Note that if $u\in \WW^{1,q}(\Omega)$, then the relaxed functional agrees with the pointwise definition, justifying the terminology.
\end{definition} 

There is an extensive literature on the Lavrentiev phenomenon, an overview of which can be found in \cite{Buttazo1995}, \cite{Foss2001}. The phenomenon also arises in nonlinear elasticity \cite{Foss2003}. If the Lavrentiev phenomenon can be excluded, then minimisers of the relaxed and pointwise functional agree. In general, under the assumptions of this paper, it is not known whether the Lavrentiev phenomenon can be excluded. We refer however to \cite{Koch2022a} for conditions under which the Lavrentiev gap can be excluded, see also \cite{Esposito2004,Esposito2019} for local versions of these conditions. We further remark that while our improved differentiability results for minimisers of \eqref{eq:neumann_problem} or \eqref{eq:dirichlet_problem} are stated for relaxed minimisers, a closer look reveals that they also imply an improved differentiability result for pointwise minimisers with an argument or assumption excluding Lavrentiev.
Finally, we point out that there is a different approach to the relaxed functional, which involves studying measure representations of it. We refer to \cite{Fonseca1997,Acerbi2003} for results and further references in this direction.

The study of regularity theory for minimisers in the case $p<q$ started with the seminal papers \cite{Marcellini1989,Marcellini1991}. We do not aim to give a complete overview of the theory here, and we will instead focus on results directly relevant to this paper. We refer to \cite{Mingione2006} for a good overview and further references. In general, the study has been almost completely concentrated on \eqref{eq:dirichlet_problem} and in particular, unless explicitly mentioned, all references pertain to this setting. A particular focus of research have been the special cases of the double-phase functional $F(x,z)=\lvert z\rvert^p+a(x)\lvert z\rvert^q$ and functionals with $p(x)$-growth. For an introduction and further references with regards to these special cases we refer to the introduction of \cite{Baroni2018} and \cite{Diening2011,Radulescu2015}, respectively.  Already in the scalar $m=1$ autonomous case $F(x,z)\equiv F(z)$ counterexamples show that in order to prove regularity of minimisers $p$ and $q$ may not be too far apart \cite{Giaquinta1987,Marcellini1989,Hong1992}. We stress that the counterexamples only apply to pointwise minimisers. We list the to our knowledge best available $W^{1,q}_{\tp{loc}}$-regularity results for general autonomous convex functionals with $(p,q)$-growth (when $n\geq 2$): Under natural growth conditions it suffices to assume $q<\frac{np}{n-1}$ \cite{Carozza2013} in order to obtain $W^{1,q}_{\tp{loc}}$-regularity of minimisers. To obtain the same conclusion under controlled growth conditions the gap may be widened to $q<p\left(1+\frac 2 {n-1}\right)$ \cite{Schaeffner2020}, and under controlled duality growth conditions it suffices to take $q<\frac{np}{n-2}$ (if $n=2$ it suffices to take $q<\infty$) \cite{DeFilippis2020}. We note that in all three cases higher integrability goes hand in hand with a higher differentiability result, for instance in \cite{Esposito2002,Carozza2013} it is proved that $V_{p,\mu}(\D u)\in \WW^{1,2}_{\tp{loc}}(\Omega)$.
 We refer to \cite{Bogelein2013} for results and references in the case of parabolic systems with $(p,q)$-growth. 

The global theory is less developed and the only general results available extend the results of \cite{Carozza2013} up to the boundary in \cite{Koch2021a}. Additionally in \cite{Bulicek2018}, Lipschitz regularity up to the boundary is obtained for minimisers of scalar autonomous functionals satisfying nonstandard growth conditions and the structure condition $F\equiv F_0(\lvert z\rvert)$. The growth conditions considered include $(p,q)$-growth.

We now turn to the case of non-autonomous functionals $F(x,z)$, convex and with $(p,q)$-growth in $z$, while satisfying a uniform $\alpha$-H\"older condition in $x$. For $n\geq 2$ counterexamples to $W^{1,q}$ regularity with $1<p<n<n+\alpha<q$ are due to \cite{Esposito2004}, see also \cite{Fonseca2004}. Recent work suggests that the condition $p<n<q$ may be removed \cite{Balci2020}. If $q<\frac{(n+\alpha)p}{n}$, it was proven in \cite{Esposito2004} for many standard examples that minimisers enjoy $W^{1,q}_{\tp{loc}}$-regularity and improved differentiability $V_{p,\mu}(\D u)\in \WW^{\frac \beta 2,2}_{\tp{loc}}(\Omega)$ for all $\beta<\alpha$. Using \cite{Esposito2019}, the result may be extended to functionals satisfying an additional condition on the $x$-dependence. $W^{1,q}_{\tp{loc}}$ regularity is in general not known if $q=\frac{(n+\alpha)p}{n}$. An exception are functionals modeled on the double-phase functional \cite{Baroni2018}, see also \cite{DeFilippis2019} and \cite{Hasto2022}.

Concerning the theory of global improved differentiability and $W^{1,q}$-regularity results for non-autonomous functionals in \cite{Koch2021a}, the results of \cite{Esposito2019} are extended up to the boundary.  For functionals satisfying a structural assumption inspired by the double-phase functional Cald\'{e}ron-Zygmund estimates valid up to the boundary are obtained in \cite{Byun2017}. H\"older-regularity up to the boundary for double-phase functionals is studied in \cite{Tachikawa2020}.

If additional structure assumptions such as $F\equiv F_0(x,\lvert z\rvert)$ are imposed or if it is assumed that minimisers are bounded it is possible to improve on the results listed so far. Without going into further detail we refer to \cite{Breit2012}, \cite{Carozza2011} for results and further references in these directions. Local boundedness of minimisers for convex non-autonomous $(p,q)$-growth functionals under natural growth conditions and the additional assumption $F(x,2z)\lesssim 1+F(x,z)$ is studied in \cite{Hirsch2020}.

A secondary problem is that of partial $C^{1,\alpha}$ regularity for solutions, which is the best we can hope for in light of classical examples \cite{DeGiorgi1968,Mazya1968}.
In the setting of $(p,q)$-growth, the first partial regularity results were obtained in \cite{Acerbi1994} for functionals with a specific structure, and later in \cite{Passarelli1996} for $F \equiv F(z)$ subject to a controlled growth condition for exponents satisfying $q \geq p \geq 2$ and
\begin{equation}\label{eq:partial_reg_range}
  q < \min\left\{p+1, \frac{np}{n-1}\right\}.
\end{equation} 
The latter authors used the smoothing operator introduced in \cite{Fonseca1997}, and these ideas were developed further in \cite{Schmidt2008} to obtain a Caccioppoli-type inequality in the quasiconvex setting.
Building upon these techniques, the case of convex integrands $F(x,z)$ was obtained in \cite{DeMaria2010} assuming the exponents satisfy \eqref{eq:partial_reg_range}.
While the above results only considered pointwise minimisers, the case of $\WW^{1,q}$-relaxed minimisers was treated in \cite{Schmidt2009} through a delicate analysis of the relaxed functional. 
Recently, under substantially weaker assumptions that incorporate $(p,q)$-growth, this was improved in \cite{Gmeineder2022} to the range $1 \leq p < q < \min\left\{p+1,\frac{np}{n-1}\right\}.$
Integrands satisfying a $p$-Laplace type degeneration at the origin were treated in \cite{Schmidt2008a}, and recently in this setting partial gradient continuity in the presence of a forcing term $f \in \LL(n,1)$ was obtained in \cite{DeFilippis2022b,DeFilippis2022a}, under more restrictive conditions on $(p,q)$ than \eqref{eq:partial_reg_range}.

The aforementioned results only consider interior regularity however, which raises the question of the existence of regular boundary points.
For convex problems with regular growth this was settled in \cite{Kristensen2010}, where the integrand can additionally depend on $u.$
This relied on a characterisation of regular boundary points obtained in \cite{Kronz2005} for quasiconvex integrands, however without an improved differentiability result it is unclear whether this characterisation is satisfied at any $x_0 \in \p\Omega.$
To our knowledge this has not been considered in the case of $(p,q)$-growth, nor in the setting of Neumann boundary.

Concerning maximal improved differentiability subject to a forcing term, regarding the $p$-Laplace system, there is a wealth of information available, c.f. \cite{Mingione2007,Breit2017,Cianchi2019} and the references therein. For the purposes of this paper, we remark that a classical result in \cite{Simon1977, Simon1981} establishes interior $B^{s,p}_{\infty}$ regularity with $s = 1+\min\left\{p-1,\frac1{p-1}\right\}$ for the $p$-Laplace equation
\begin{equation}
  - \div \left( \lvert \D u\rvert^{p-2} \D u \right) = f,
\end{equation} 
with $f \in L^{p^\prime},$ and that this is sharp for $p\geq 2$. Further, up to the boundary, \cite{Simon1981} establishes that it is possible to take $s=\min\left\{(p-1)^2,\frac 1 {(p-1)^2}\right\}$. Our results indicate that this likely is not sharp.
If $f \in W^{1,p^\prime},$ the obtained regularity improves to $V_{p}(\D u) \in W^{1,2};$ this was noted to hold globally in \cite{EbmeyerLiuSteinhauer2005}. Finally, we remark that maximal improved Besov-regularity for the $p$-Laplace equation was studied in \cite{Dahlke2016}.

\subsection{Overview of results}

Our first main result extends the higher integrability results of \cite{Koch2020} to the Neumann case:
\begin{restatable}{theorem}{neumannBasic}\label{thm:regularityRelaxed}
Let $\alpha\in(0,1]$ and $1<p\leq q<\frac{(n+\alpha)p} n$.
Let $\Omega$ be a Lipschitz domain, $f\in \BB_{\infty}^{\alpha-1,q^\prime}(\Omega)$ and $g_N\in \WW^{\alpha-\frac 1 {q^\prime},q^\prime}(\p\Omega)$ be such that \eqref{eq:compatibility_condition} is satisfied.
Suppose $F$ satisfies \eqref{def:bounds1}-\eqref{def:bounds3} with $1< p\leq q< \frac{(n+\alpha)p}{n}$, and that $u$ is a relaxed minimiser of $\F_N$ in the class $\WW^{1,p}(\Omega)$. 
Then $u\in \WW^{1,q}(\Omega),$ and for any $\beta<\alpha$ the estimate
\begin{align}\label{eq:regRelaxed1}
\|u\|_{\WW^{1,\frac{np}{n-\beta}}(\Omega)} \lesssim \left(1+\overline\F_N(u)+\|f\|_{\BB_{\infty}^{\alpha-1,q^\prime}(\Omega)}^{q^\prime}+\|g_N\|_{\WW^{\alpha-\frac 1 {q^\prime},q^\prime}(\p\Omega)}^{q^\prime}\right)^\gamma
\end{align}
holds for some $\gamma>0$, where the implicit constant depends on $\Omega$, $\lambda$, $\Lambda$, $n$, $\beta$, and $p$. In fact, we have
\begin{align}\label{eq:regRelaxed2}
\|V_{p,\mu}(\D u)\|_{\BB^{1+\frac \alpha 2,2}_\infty(\Omega)}\lesssim \left(1+\overline\F_N(u)+\|f\|_{\BB_{\infty}^{\alpha-1,q^\prime}(\Omega)}^{q^\prime}+\|g_N\|_{\WW^{\alpha-\frac 1 {q^\prime},q^\prime}(\p\Omega)}^{q^\prime}\right)^\gamma.
\end{align}
\end{restatable}

If we additionally assume \eqref{def:bounds31}, then we can obtain a further improvement in differentiability.
Roughly speaking, once solutions lie in $\WW^{1,q}$ they achieve the same improvement in differentiability as solutions to integrands with $(p,p)$-growth, up to possibly the endpoint scale.
Differentiability agreeing with that enjoyed by solutions to integrands with $(p,p)$-growth was previously known only for autonomous integrands \cite{Esposito2002,Carozza2013}. The argument there relies on differentiating the Euler-Lagrange equation and hence on the existence of second derivatives of the solution. This is not available in the non-autonomous setting. Instead we use second order difference quotients and a delicate iteration argument. Moreover, we encompass a forcing term with sharp regularity assumptions, see the discussion below. We state the result here in the interior case where the result is already new.

\begin{restatable}{theorem}{interiorImproved}\label{thm:interiorImproved}
  Suppose $1<p\leq q < \infty$, $\alpha \in (0,1]$, and suppose $F$ satisfies \eqref{def:bounds1}--\eqref{def:bounds31}.
  Let $q < \frac{n+\alpha}n p$ and $f \in \BB^{\alpha-1,p^\prime}_{\infty}(\Omega).$
  Suppose $u$ is a relaxed minimiser for $\F.$ 
  Then we have
  \begin{equation}
    V_{p,\mu}(\D u) \in \BB^{\delta,2}_{\infty,\loc}(\Omega) \quad \text{ with } \delta = \frac{\alpha}2 \min\{2,p^{\prime}\}.
  \end{equation} 
  Moreover if $p \geq 2$ and $f \in \BB^{\beta-1,p^{\prime}}_{1}(\Omega)$ for $\beta \in [\alpha,2]$, then we can take $\delta = \min\left\{\frac{p^{\prime}\beta}{2},\alpha\right\}$.

  Additionally, if a-priori $u \in W^{1,q}(\Omega),$ it suffices to assume $q < \frac{np}{n-\alpha}$ and \eqref{def:bounds3} can be omitted.
\end{restatable}

\begin{remark}\label{rem:w12_diff}
In particular we see that if $\alpha=1$, we have
\begin{equation}
  f \in \BB^{\max\{1-\frac2p,0\},p^{\prime}}_{1}(\Omega) \implies V_{p,\mu}(\D u) \in \WW^{1,2}_{\loc}(\Omega).
\end{equation} 
Moreover if $p \leq 2$, we are simply requiring that $f \in \LL^{p^\prime}(\Omega)$.
In the case $p>2$, the fractional regularity $s = 1-\frac2p$ is shown to be sharp for the $p$-Laplace in \cite{Brasco2018}, however it is unclear whether one can relax the regularity of $f$ to $\BB^{1-\frac2p,p^{\prime}}_q(\Omega)$ for some $q>1$.
\end{remark}

Restricting to radial integrands $F\equiv F_0(x,\lvert z\rvert)$ and homogeneous boundary conditions, adapting techniques developed in \cite{EbmeyerLiuSteinhauer2005,Ebmeyer2005} we can extend this to hold globally.

\begin{restatable}{theorem}{neumannImproved}\label{thm:relaxedImproved}
  Suppose $1<p\leq q < \infty,$ $\alpha \in (0,1]$.  Let $\Omega$ be a bounded $C^{1,1}$-domain and suppose $F\equiv F_0(x,\lvert z\rvert)$ satisfies \eqref{def:bounds1}--\eqref{def:bounds31}.
  Let $q < \frac{n+\alpha}n p$ and $f \in \BB^{\alpha-1,p^\prime}_{\infty}(\Omega).$ Assume $g_N = g_D = 0$.
  Then, if $u$ is a relaxed minimiser for $\F_D$ or $\F_N$,  for all $\delta \leq \frac{\alpha}2 \min\{2,p^\prime\}$ we have $V_{p,\mu}(\D u) \in \BB^{\delta,2}_{\infty}(\Omega)$ with the associated estimate
  \begin{align*}
    \|V_{p,\mu}(\D u)\|_{\BB_{\infty}^{\delta,2}(\Omega)} \lesssim \left(1+\overline\F_N(u)+\|f\|_{\BB_{\infty}^{\alpha-1,p^\prime}(\Omega)}^{p^\prime}\right)^\gamma.
  \end{align*}
  Further if $p \geq 2$ and $\beta\in [\alpha,2]$ such that $f \in \BB^{\beta-1,p^\prime}_{\infty}(\Omega),$ then $V_{p,\mu}(\D u) \in \BB^{\delta,2}_{\infty}(\Omega)$ for all $\delta \leq \min\left\{\frac{p^{\prime}\beta}{2},\alpha\right\}.$
  Additionally, if a-priori $u \in W^{1,q}(\Omega),$ it suffices to assume $q < \frac{np}{n-\alpha}$ and \eqref{def:bounds3} can be omitted.
\end{restatable}

In the Neumann case we can also infer some higher differentiability when $g_N \not\equiv 0;$ see Theorem \ref{thm:radial_neumann}.
Here the differentiability of order $\delta>\frac12$ will be crucial to establish partial boundary regularity in the inhomogenous case; see Theorem \ref{thm:nonHomogeneousNeumann}.
To our knowledge a result of this type is new in the non-uniformly elliptic setting, and the proof crucially depends on the iteration technique we will develop.

\begin{restatable}{theorem}{radialNeumann}\label{thm:radial_neumann}
  Suppose $1 < p \leq q < \infty,$ $\alpha \in (0,1],$ $q < \frac{n+\alpha}n p,$ $\Omega$ is a bounded $C^{1,1}$-domain and $F \equiv F_0(x,\lvert z\rvert)$ satisfies \eqref{def:bounds1}--\eqref{def:bounds31}.
  Then there is $\delta_0>\frac 1 2$, such that the following holds for all $\delta \leq \min\left\{\frac \alpha 2\min(2,p^\prime), \delta_0\right\}$:
  If $f \in \BB^{\alpha-1,p^\prime}_{\infty}(\Omega)$, ${g_N \in \WW^{\alpha-\frac1{p^\prime},p^\prime}(\p\Omega)}$, and $u$ is a relaxed minimiser of $\overline{\F}_N,$ then $V_{p,\mu}(\D u) \in \BB^{\delta,2}_{\infty}(\Omega)$ and we have the corresponding estimate
  \begin{equation}
    \seminorm{V_{p,\mu}(\D u)}_{\BB^{\delta,2}_{\infty}(\Omega)} \leq C \left( \overline{\F}_N(u) + \norm{f}_{\BB^{\alpha,p^\prime}_{\infty}(\Omega)}^{p^\prime} + \norm{g_N}_{\WW^{\alpha-\frac1{p^\prime},p^\prime}(\p\Omega)}^{p^\prime}\right)^{\gamma}.
  \end{equation} 
  The same conclusion holds for $\delta \leq \min\left\{\alpha,\frac{p^\prime \beta} 2,\delta_0\right\}$ (with a modified estimate) if $p\geq 2,$ $f \in \BB_{1}^{\beta-1,p^\prime}(\Omega)$, $g \in \BB^{\beta-\frac1{p^{\prime}},p^{\prime}}_{1}(\partial\Omega)$ with $\beta \in [\alpha,2]$.
\end{restatable}

\begin{remark}
  The main result in \cite{Ebmeyer2005} is stated to hold for non-radial integrands, however we believe that the proof is not correct. When restricting to radial integrands ${F\equiv F_0(x,\lvert z\rvert)}$, the proof can be corrected along the lines of the argument we give in Section \ref{sec:improved} and the following statement can be shown to be true:

  \textit{Let $2<p<\infty$. Suppose $\Omega$ is a $C^{1,1}$-domain, $f\in \BB_1^{1-\frac 2 p,p^\prime}(\Omega)$, and $\int_\Omega f\,\d x = 0$. Assume $F\equiv F_0(x,\lvert z\rvert)$ satisfies \eqref{def:bounds1}-\eqref{def:bounds31} with $\alpha=1$ and $q=p$. Suppose $u$ minimises $\F$. Then $u\in \BB^{1+\frac 2 p,p}_\infty(\Omega)$ and $V_{p,\mu}(\D u)\in \WW^{1,2}(\Omega)$.}

The issue with the proof given in \cite{Ebmeyer2005} lies in the proof of Lemma 3 of the paper. In the notation used there, it is not the case that $\p_i D_n^{-h}D_n^h v(x) = D_n^{h}D_n^{-h} \p_i u$. Since $v$ is the even extension of $u$, the normal derivative $\p_n v$ is odd, whereas $\p_n u$ was defined to be an even extension of $\p_n u$. This also causes issues with the claim
that $J_{03}=-J_{04}$. It is precisely these issues that force us to restrict to radial integrands in the above theorem.
Moreover one also needs to show that $\p_n v = 0$ on $\partial\Omega$ to show this extension preserves fractional differentiability of $\D u$, which we prove in Lemma \ref{lem:zeroNeumann}, however this additional argument is missing in \cite{Ebmeyer2005}.
In connection to Remark \ref{rem:w12_diff} we also mention that \cite{EbmeyerLiuSteinhauer2005,Ebmeyer2005} assumes the weaker condition $f \in \WW^{1-\frac2p,p^{\prime}}(\Omega)$, however we were unable to verify this claim.
\end{remark}

Our results can also be adapted to piecewise $C^{1,1}$ domains, and to the case of mixed Dirichlet-Neumann boundary conditions, which we investigate in Section \ref{sec:mixedboundary}.
We also remark that in the autonomous setting, it is possible to improve the range of $q$ by employing arguments from \cite{Koch2021a} that are based on \cite{Schaeffner2020}, extending the results of \cite{Carozza2013} from the interior case.
\begin{restatable}{corollary}{autonomousImproved}\label{thm:autonomousImproved}
  If $F$ is independent of $x$, Theorems \ref{thm:regularityRelaxed}, \ref{thm:interiorImproved} and \ref{thm:relaxedImproved} hold under the assumption $2\leq p\leq q<\max\left\{\frac{np}{n-1},p+1\right\}$.
\end{restatable}

We remark that for the $p$-Laplace operator with forcing term $f\in \WW^{-1+s,p^\prime}$, $0\leq s\leq 1$, in \cite{Simon1977}, it is shown that solutions may in general fail to lie in $\WW^{1+\frac s{\max\{2,p\}-1}+\e,p}_{\tp{loc}}(\Omega)$ for any $\e>0$. For $p\geq 2$, this example has been adapted to Besov spaces in \cite{Weimar2021} showing in particular that, if $f\in \BB^{-1+s,p^\prime}_\infty(\Omega)$, solutions in general do not lie in $\BB^{1+\frac 2 {\max\{2,p\}-1}}_{\infty,\tp{loc}}(\Omega)$. An example in \cite{Brasco2018} shows that if $p\geq 2$ and $f\in \WW^{s,p^\prime}$, then solutions may in general fail to satisfy $V_{p,\mu}(\D u)\in\WW^{\frac p 2 \frac{s+1}{p-1}+\e,2}_{\tp{loc}}(\Omega)$. Thus our results are optimal.

As a consequence of our improved differentiability results, we can establish the existence of regular boundary points.
For this we will need to characterise the singular set of relaxed minimisers via a suitable $\eps$-regularity result, which extends the corresponding interior result from \cite{DeMaria2010}.
\begin{restatable}{theorem}{epsRegularity}\label{thm:eps_regularity}
Assume $F \colon \Omega \times \R^{m\times n} \to \R$ is $C^2$ in $z,$ satisfies \eqref{def:bounds1}-\eqref{def:bounds31} with $\mu=1,$ and the exponents satisfy
\begin{equation}\label{eq:pq_partialregularity}
  1 < p \leq q < \min\left\{\frac{np}{n-1},p+1\right\}.
\end{equation} 
We will assume further that $\Omega \subset \bb R^n$ is a bounded $C^{1,\alpha}$ domain and $f \in \LL^{\frac{n}{1-\alpha}}(\Omega,\R^m).$ For the boundary function we will assume $g_D \in C^{1,\alpha}(\Omega,\R^m)$ in the case of Dirichlet boundary, and $g_N \in C^{0,\alpha}(\Omega,\R^m)$ in the case of Neumann boundary, in which case we also assume the compatibility condition \eqref{eq:compatibility_condition}.
Suppose $u$ is a relaxed minimiser of \eqref{eq:neumann_problem} or \eqref{eq:dirichlet_problem}. Then for each $M>0$ and $0<\beta<\alpha,$ there is $\eps>0$ and $R_0>0$ such that if we have $\lvert (\D u)_{\Omega_R(x_0)}\rvert \leq M$ and
  \begin{equation}
    \left( \dashint_{\Omega_R(x_0)} \left\lvert V_p(\D u - (\D u)_{\Omega_R(x_0)})\right\rvert^2 \,\d x \right)^{\frac12} < \eps,
  \end{equation} 
  for some $x_0 \in \overline\Omega$ and $0<R<R_0,$ then $u$ is $C^{1,\beta}$ in $\overline\Omega_{\frac R2}(x_0),$
  where ${\Omega_R(x_0) = \Omega \cap B_R(x_0)}.$
\end{restatable}

As a direct consequence of Theorem \ref{thm:eps_regularity}, we obtain the following characterisation of regular points.
Here the singular set is understood to be the set of points $x_0$ where $u$ fails to be $C^1$ in every neighbourhood of $x_0.$
\begin{restatable}[Characterisation of regular points]{corollary}{singularPoints}\label{cor:singularPoints}
  Under the assumptions of Theorem \ref{thm:eps_regularity}, the singular set of $u$ in $\overline\Omega$ is given by
  \begin{equation}\label{eq:u_singularset}
    \Sigma = \left\{ x \in \overline\Omega : \limsup_{r \to 0} \lvert (\D u)_{\Omega_r(x)} \rvert = \infty \text{ or } \liminf_{r \to 0} \dashint_{\Omega_r(x)} \left\lvert V_p(\D u - (\D u)_{\Omega_r(x)}) \right\rvert^2 > 0 \right\}.
  \end{equation} 
  That is, $u \in C^{1,\beta}(\overline\Omega \setminus \Sigma)$ for each $\beta < \alpha.$
\end{restatable}

\begin{corollary}\label{eq:dimension_estimates}
  Let $\alpha\in(0,1]$, $\eps>0$ and $1<p\leq q\leq \min\left\{\frac{(n+\alpha)p}{n},p+1\right\}$. Suppose $\Omega$ is a $C^{1,1}$-domain and $f\in \LL^{\frac n {1-\eps}}(\Omega,\R^m)\cap \BB_{\infty}^{\alpha-1,p^\prime}(\Omega,\R^m)$.
  Under the assumptions of Theorem \ref{thm:eps_regularity}, if $F\equiv F_0(x,\lvert z\rvert)$, $u$ is a relaxed minimiser of \eqref{eq:neumann_problem} with $g_N=0$ (resp.\,\eqref{eq:dirichlet_problem} with $g_D=0$) and $\Sigma$ is the singular set for $u$ as given in \eqref{eq:u_singularset}, we have
  \begin{equation}
    \dim_{\H} \Sigma \leq n - \alpha \min\{2,p^\prime\}.
  \end{equation} 
  If $p\geq 2$, $F$ satisfies in addition \eqref{def:bounds31} and $f\in \LL^{\frac{n}{1-\eps}}(\Omega,\R^m)\cap \BB^{\beta-1,p^\prime}(\Omega,\R^m)$ with $\beta \in [\alpha,2]$, then
\begin{equation}
  \dim_{\H} \Sigma \leq n - \min\left\{2\alpha,p^\prime\beta\right\}.
\end{equation}
\end{corollary}

The following is now immediate.

\begin{theorem}[Existence of regular boundary points]\label{thm:regularBoundaryPoints}
  Suppose the assumptions of Corollary \ref{eq:dimension_estimates} hold. If $f\in \LL^{\frac{n}{1-\eps}}(\Omega,\R^m)\cap \BB^{\alpha-1,p^\prime}(\Omega,\R^m)$ for some $\eps>0$ and
  \begin{align*}
    \alpha > \max\left\{\frac 1 2,\frac 1 {p^\prime}\right\},
  \end{align*}
  then $\mathscr \H^{n-1}$-almost every boundary point is regular for $u$. In particular, this holds if ${F\equiv F(z)}$ is autonomous. If $p\geq 2$, $\alpha>\frac 1 2$ and additionally $f\in \LL^{\frac{n}{1-\eps}}(\Omega,\R^m)\cap \BB_{\infty}^{\beta-1,p^\prime}(\Omega,\R^m)$ with $\beta > \frac{1}{p^{\prime}}$, we have $\mathscr \H^{n-1}$-almost every boundary point is regular for $u$.
\end{theorem}

Finally, in the Neumann case, by Theorem \ref{thm:radial_neumann} we obtain the existence of regular boundary points for inhomogeneous boundary conditions.
\begin{theorem}\label{thm:nonHomogeneousNeumann}
  Let $\frac12 < \alpha \leq 1$, $\eps>0$ and $1<p\leq q<\min\left\{\frac{(n+\alpha)p}{n},p+1\right\}$. Suppose $\Omega$ is a $C^{1,1}$-domain and $f\in \LL^{\frac{n}{1-\eps}}(\Omega,\R^m)\cap \BB^{\beta-1,p^\prime}_{\infty}(\Omega,\R^m)$ with  $\beta > \max\left\{\frac12,\frac1{p^{\prime}}\right\}$.
  If $F\equiv F_0(x,\lvert z\rvert)$ satisfies \eqref{def:bounds1}-\eqref{def:bounds31} and $u$ is a relaxed minimiser of \eqref{eq:neumann_problem} with $g_N \in \BB_1^{\beta-\frac1{p^\prime},p^\prime}(\p\Omega,\R^m)\cap C^{0,\eps}(\p\Omega,\R^m)$, then $\mathscr \H^{n-1}$-almost every boundary point is regular for $u$.
\end{theorem}

It would be interesting to obtain a statement akin to Theorem \ref{thm:nonHomogeneousNeumann} in the case of Dirichlet boundary conditions.

\begin{remark}
  If a-priori $u\in \WW^{1,q}(\Omega)$, it suffices to assume $q<\min\left\{\frac{np}{n-\alpha},p+1\right\}$ in the statements of Corollary \ref{eq:dimension_estimates}, Theorem \ref{thm:regularBoundaryPoints} and Theorem \ref{thm:nonHomogeneousNeumann}. 
  In particular in the case of homogeneous integrands $F\equiv F_0(\lvert z\rvert)$, $2 \leq p \leq q<\min\left\{\frac{np}{n-1},p+1\right\}$ suffices.
\end{remark}


The outline of the paper is follows. In Section \ref{sec:prelim} we collect our notation and a number of preliminary results, concerning in particular Besov spaces and the relation of the relaxed functional to appropriate regularised versions of $F$.
In Section \ref{sec:smoothing_relaxation} we consider a boundary version of a smoothing operator introduced in \cite{Fonseca1997}, where an additivity property for the relaxed functional is proved allowing us to localise  and flatten the boundary in our analysis.

In Section \ref{sec:relaxed} we prove our differentiability results, starting with the basic result in the Neumann case in Section \ref{sec:basicNeumann}. 
We then consider the improved results, proving the interior statement in Section \ref{sec:improved} where the key iteration argument is described.
The corresponding global version is considered in Section \ref{sec:boundaryImproved} following the argument of \cite{EbmeyerLiuSteinhauer2005,Ebmeyer2005}, and in the Neumann case we must show an even extension of $u$ preserves the desired fractional differentiability; this follows from Lemmas \ref{eq:zeroextension} and \ref{lem:zeroNeumann}.
We also explain how to apply these results to mixed boundary conditions in Section \ref{sec:mixedboundary}.

Finally in Section \ref{sec:excessEstimate}, we prove the characterisation of regular boundary points, by establishing a Caccioppoli-type inequality which we combine with a blow-up argument.

\section{Preliminaries}\label{sec:prelim}
\subsection{Notation}\label{sec:notation}
In this section we introduce our notation. 
$\Omega$ will always denote a open, bounded domain in $\R^n$. Given $\omega\subset\R^n$, $\overline \omega$ will denote its closure and $\chi_\omega$ the indicator function of $\omega$. For subsets of $\R^n$ we will use both $\mathscr L^n$ and $\lvert \cdot \rvert$ to denote the Lebesgue measure, $\mathscr H^s$ for the Hausdorff measure, and $\tp{dim}_{\mathscr{H}}$ for the Hausdorff dimension. 
We also denote the half-space in $\bb R^n$ by $\bb R^n_+ = \{ x \in \bb R^n : x_n>0\}$.
For $x\in \R^n$ and $r>0$, $B_r(x)$ is the open Euclidean ball of radius $r$, centred at $x$, while $S^{n-1}$ is the unit sphere in $\R^n$. We also write $\Omega_r(x) = \Omega \cap B_r(x).$ Further we denote by $K_r(x)$ the cube of side-length $r$, centred at $x$.
We denote the cone of height $\rho$, aperture $\theta$ and axis in direction $\mathbf n$ by $C_\rho(\theta, \mathbf n)$. That is
\begin{align*}
C_\rho(\theta, \mathbf n)=\{\,h\in \R^n: \lvert h\rvert\leq \rho, h\cdot \mathbf n\geq \lvert h\rvert\cos(\theta)\,\}.
\end{align*}
Here and elsewhere $\lvert\cdot \rvert$ denotes the Euclidean norm of a vector in $\R^n$ and likewise the Euclidean norm of a matrix $A\in \R^{n\times n}$. $\tp{Id}$ is the identity matrix in $\R^{n\times n}$.
Given an open set $\Omega$, for $\lambda>0$ set $\Omega^\lambda=\{\,x\in \Omega: d(x,\partial\Omega)>\lambda\,\}$ and $\lambda \Omega = \{\,\lambda x: x\in\Omega\,\}$. For $h\in \R^n$, $\Omega + h = \{x\in \R^n\colon x-h\in \Omega\}$. Here $d(x,\partial\Omega)=\inf_{y\in \p\Omega} \lvert x-y\rvert$ is the distance of $x$ from the boundary of $\Omega$. Abusing notation, we write $\Omega^h = \Omega^{\lvert h\rvert}$ for $h\in \R^n$.

If $p\in[1,\infty]$, denote by $p^\prime=\frac{p}{p-1}$ the H\"older conjugate.
The symbols $a \sim b$ and $a\lesssim b$ mean that there is some constant $C>0$, depending only on $n,\, m,\, p,\,q,\, \Omega,\, \mu,\, \nu,\, \Lambda$ and $\alpha,$ and independent of $a$ and $b$ such that $C^{-1} a \leq b \leq C a$ and $a\leq C b$, respectively. 

We will often find it useful to write for a function $v$ defined on $\R^n$ and a vector $h\in \R^n$, $v_h(x)=v(x+h)$.
We also pick a family $\{\,\rho_\e\,\}$ of radially symmetric, non-negative mollifiers of unitary mass. We denote convolution of a locally integrable function $u$ with $\rho_\e$ as
\begin{align*}
u\star\rho_\e(x)=\int_{\R^n} u(y)\rho_\e(x-y)\d y.
\end{align*}
For bounded and open sets $\Omega \subset \bb R^n$, we will also use the notation $\dashint_{\Omega} f \,\d x = \frac1{\lvert\Omega\rvert}\int_{\Omega} f \,\d x$ for averaged integrals.
Function spaces will be denoted $X(\Omega) = X(\Omega,\bb R^m)$ where the target space may be dropped if it is clear from context, and for $f \in X,$ $g \in X^{\prime}$ we will write $\int_{\Omega} f \cdot g \,\d x$ to be understood as the duality pairing.

At times, we need to consider versions of $\F_N$, $\F_D$ and $\overline \F_N$, $\overline \F_D$ localised to $\omega\subset \R^n$. These are obtained by restricting the domain of integration and definition of the functions involved to $\omega$ and denoted $\F_N(\cdot,\omega)$, $\F_D(\cdot,\omega)$, $\overline \F_N(\cdot,\omega)$ and $\overline \F_D(\cdot,\omega)$, respectively.
For the case of mixed boundary, and when localising in the Neumann case, we will need the relaxed functional defined for $\omega \subset \bb R^n,$ $\gamma \subset \p\omega$ by
\begin{align}\label{eq:mixed_relaxed}
  \overline \F_M(v,\omega,\gamma)=\inf \left\{\,\liminf_{j\to\infty} \F(v_j): (v_j)\subset v+Y, v_j\rightharpoonup v \text{ weakly in } X\,\right\}
\end{align}
where $X= \{u\in \WW^{1,p}(\omega)\colon \tp{Tr}\, u=v \text{ on } \gamma\}$ where $Y=\WW^{1,q}(\omega)\cap X$. 
We will write $\overline \F_M(u) = \overline \F_M(u,\Omega,\Gamma_D)$ for mixed problems.

\subsection{WB coverings, Lipschitz domains and extensions of \texorpdfstring{$F$}{F}}\label{sec:WBLipschitz}
We define Whitney-Besicovitch coverings, which combine properties of Whitney and Besicovitch coverings and were introduced in \cite{Kislyakov2005}. A nice presentation of the theory is given in \cite{Kislyakov2013}. To be precise:
\begin{definition}
A family of dyadic cubes $\{\,Q_i\,\}_{i\in I}$ with mutually disjoint interiors is called a Whitney-covering of $\Omega$ if
\begin{align*}
\bigcup_{i\in I} Q_i = \bigcup_{i\in I} 2Q_i = \Omega\\
5Q_i \cap (\R^n\setminus \Omega) \neq \emptyset.
\end{align*}
 A family of cubes $\{\,K_i\,\}_{i\in I}$ is called a Whitney-Besicovitch-covering (WB-covering) of $\Omega$ if there is a triple $(\delta, M, \varepsilon)$ of positive numbers such that
 \begin{align}\label{def:WBcoverExtension}
  & \bigcup_{i\in I} \frac{1}{1+\delta}K_i = \bigcup_{i\in I} K_i = \Omega\\\label{def:WBcoverMultiplicity}
  & \sum_{i\in I} \chi_{K_i}\leq M \\\
  & K_i \cap K_j \neq \emptyset \Rightarrow \lvert K_i \cap K_j\rvert \geq \varepsilon \max(\lvert K_i\rvert,\lvert K_j\rvert).  
 \end{align}

\end{definition}
The existence of a Whitney-covering for $\Omega$ is classical. The refinement to a WB-covering can be found in \cite{Kislyakov2013}:
\begin{theorem}[cf. Theorem 3.15, \cite{Kislyakov2013}]\label{thm:covering}
Let $\Omega$ be an open subset of $\R^n$ with non-empty complement. Let $\{\,Q_i\,\}$ be a family of cubes which are a Whitney covering of $\Omega$. Then the cubes $K_i = \left(1+\frac{1}{6}\right)Q_i$ are a WB-covering of $\Omega$ with $\delta = \frac{1}{6}$, $\e = \frac{1}{14^n}$ and $M\leq 6^n-4^n+1$. Moreover for this covering $\frac{2}{\left(1+\frac{1}{6}\right)^\frac{1}{n}}\lvert K_i\rvert^\frac{1}{n}\leq\dist(K_i,\partial\Omega)$.
\end{theorem}
It will be of crucial importance to us that there exists a partition of unity associated to a WB-covering.
\begin{theorem}[cf. Theorem 3.19, \cite{Kislyakov2013}]\label{thm:WBunity}
 Suppose the cubes $\{\,K_i\,\}_{i\in I}$ form a WB-covering of $\Omega$ with constants $(\delta,M,\varepsilon)$. Then there is a family $\{\,\psi_i\,\}_{i\in I}$ of infinitely differentiable functions that form a partition of unity on $\Omega$ with the following properties:
 \begin{gather*}
  \supp (\psi_i) \subset \frac{1+\frac{\delta}{2}}{1+\delta} K_i\\
  \psi_i(x) \geq \frac{1}{M} \text{ for } x\in \frac{1}{1+\delta} K_i\\
  \lvert D\psi_i\rvert \leq c \frac{1}{\lvert K_i\rvert^{\frac{1}{n}}}.
 \end{gather*}
 for all $i\in I$.
\end{theorem}

It is well-known that Lipschitz domains satisfy a uniform interior and exterior cone condition. To be precise: Let $\Omega$ be a Lipschitz domain. Then (see e.g. \cite[Section 1.2.2]{Grisvard1992}): there are $\rho_0,\theta_0>0$ and a map $ \mathbf n\colon \R^n\to S^{n-1}$ such that for every $x\in \R^n$
\begin{align}\label{eq:uniformCone}
C_{\rho_0}(\theta_0, \mathbf n(x))\subset O_{\rho_0}(x)=\left\{\,h\in\R^n: \lvert h\rvert\leq \rho_0, \left(B_{3\rho_0}(x) \setminus \Omega\right)+h\subset\R^n\setminus\Omega\,\right\}\\
C_{\rho_0}(\theta_0,- \mathbf n(x))\subset I_{\rho_0}(x)=\left\{\,h\in\R^n: \lvert h\rvert\leq \rho_0, \left(\Omega\cap B_{3\rho_0}(x)\right)+h\subset\Omega\,\right\}.
\end{align}

Finally, we extend $F$ outside $\Omega$. Let $\Omega\Subset B(0,R)$. We extend $F(x,z)$ to an integrand on $B(0,R)\times \R^{m\times n}$ still denoted $F$ by setting
\begin{align*}
F(x,z) = \inf_{y\in\Omega} F(y,z)+(1+\lvert z\rvert^2)^\frac q 2\lvert x-y\rvert.
\end{align*}
It is straightforward to check that if $F$ satisfies any of the properties \eqref{def:bounds1}-\eqref{def:bounds3} on $\Omega$, $F$ satisfies them on $B(0,R)$ (after potentially increasing $\Lambda$).
In Section \ref{sec:boundaryImproved} we will also need our extension to preserve \eqref{def:bounds31}, however there we will flatten the boundary and take an even reflection in $x$ instead.

\subsection{Function spaces}
\label{sec:besov}
We recall some basic properties of Sobolev and Besov spaces following the exposition in \cite{Savare1998}, alternatively the theory can be found in \cite{Triebel1978}.

For $0\leq \alpha\leq 1$ and $k\in \N$, $C^k(\Omega)$ and $C^{k,\alpha}(\Omega)$ denote the spaces of functions $k$-times continuously differentiable in $\Omega$ and $k$-times $\alpha$-H\"older differentiable in $\Omega$, respectively.

For $1\leq p\leq\infty, k\in \N$, $\LL^p(\Omega)=\LL^p(\Omega,\R^m)$ and $\WW^{k,p}(\Omega)=\WW^{k,p}(\Omega,\R^m)$ denote the usual Lebesgue and Sobolev spaces respectively. We write $\WW^{k,p}_0(\Omega)$ for the closure of $C_0^\infty(\Omega)$-functions with respect to the $\WW^{k,p}$-norm. For $g\in \WW^{1,1}(\Omega)$, we write ${\WW^{k,p}_g(\Omega)= g + \WW^{k,p}_0(\Omega)}$.  We freely identify $\WW^{k,p}$-functions with their precise representatives.

 Denote by $[\cdot,\cdot]_{s,q}$ the real interpolation functor. Let $s\in(0,1)$ and $p,q\in [1,\infty]$. We define
\begin{gather*}
\BB^{s,p}_q(\Omega)=\BB^{s,p}_q(\Omega,\R^m)=[\WW^{1,p}(\Omega,\R^m),\LL^p(\Omega,\R^m)]_{s,q}\\
\BB^{1+s,p}_q(\Omega)=[\WW^{2,p}(\Omega),\WW^{1,p}(\Omega)]_{s,q}=\{\,v\in \WW^{1,p}(\Omega): \D v\in \BB^{s,p}_q(\Omega)\,\}
\end{gather*}
Further we recall that $\WW^{1+s,p}(\Omega)=\BB^{1+s,p}_p(\Omega)$ and that for $1\leq q<\infty$, $\BB^{s,p}_q(\Omega)$ embeds continuously in $\BB^{s,p}_\infty(\Omega)$. For $\alpha<0$, we define $\WW^{\alpha,p}(\Omega)$ as the dual of $\WW^{-\alpha,p^\prime}_0(\Omega),$ and for $s \in (0,1)$ we can also define
\begin{equation}\label{eq:besov_negative}
  \BB^{-s,p}_q(\Omega) = [\LL^{p}(\Omega),\WW^{-1,p}(\Omega)]_{s,q}.
\end{equation} 
Note we also have $\WW^{-s,p}(\Omega) = \BB^{-s,p}_p(\Omega)$ by the duality theorem for real interpolation, and also that $\BB^{-s,p^{\prime}}_{q^{\prime}}(\Omega)$ identifies with the dual of $\BB^{s,p}_{q,0}(\Omega) = [\LL^p(\Omega),\WW_0^{-1,p}(\Omega)]_{s,q}$ for $1 \leq p,q < \infty$.

We will use a characterisation of Besov spaces $\BB^{s,p}_{\infty}$ in terms of difference quotients.  Let $D$ be a set whose linear span is $\R^n$, star-shaped with respect to $0$. For $s\in(0,1)$, $p\in[1,\infty]$, consider
\begin{align*}
[v]_{s,p,\Omega}^p = \sup_{h\in D\setminus\{\,0\,\}}\int_{\Omega^h}\left\lvert\frac{v_h(x)-v(x)}{\lvert h\rvert^s}\right\rvert^p\d x
\end{align*}
This characterises $\BB^{s,p}_\infty(\Omega)$ in the sense that
\begin{align*}
v\in \BB^{s,p}_\infty(\Omega)\Leftrightarrow v\in \LL^p(\Omega) \text{ and } [v]_{s,p,\Omega}^p<\infty.
\end{align*}
Moreover there are positive constants $C_1,C_2>0$ depending only on $s,\, p,\, D,\, \Omega$ such that
\begin{align}\label{eq:besovcharacterisation}
C_1 \|v\|_{\BB^{s,p}_\infty(\Omega)}\leq \|v\|_{\LL^p(\Omega)}+[v]_{s,p,\Omega}\leq C_2\|v\|_{\BB^{s,p}_\infty(\Omega)}.
\end{align}
If $\Omega=B_r(x_0)$, then $C_1, C_2$ are unchanged by replacing $D$ with $QD$, where $Q$ is an orthonormal matrix. In particular, when $D=C_\rho(\theta, \mathbf n)$ is a cone, they are independent of the choice of $n$. 


Recall also that $\BB^{s,p}_q(\Omega)$ may be localised for $s \in(0,2)$ and $1 \leq p,q \leq \infty$: If $\{\,U_i\,\}_{i\leq M}$ is a finite collection of balls covering $\Omega$, then $v\in \BB^{s,p}_q(\Omega)$ if and only if $v_{\vert \Omega \cap U_i}\in \BB^{s,p}_q(\Omega \cap U_i)$ for $i=1,...,M$. Moreover, there are constants $C_3,C_4$ depending only on $M$ such that
\begin{align}\label{eq:besovlocalisation}
C_3 \|v\|_{\BB^{s,p}_q(\Omega)}\leq \sum_{i=1}^M \|v\|_{\BB^{s,p}_q(\Omega \cap U_i)}\leq C_4 \|v\|_{\BB^{s,p}_q(\Omega)}.
\end{align}

We recall a well-known embedding theorem, see e.g. \cite{Triebel1983}:
\begin{theorem}\label{thm:embedding}
Let $1\leq p\leq p_1\leq \infty$, $1\leq q\leq q_1\leq \infty$, $0 < s_1 < s < 2.$ Suppose $\Omega $ is a Lipschitz domain.
Assume that 
\begin{equation}\label{eq:embedding_exponents}
  s-\frac n p \geq s_1-\frac n {p_1}
\end{equation} 
 and suppose $v\in \BB^{s,p}_\infty(\Omega)$. Then,
\begin{align*}
\|v\|_{\BB^{s_1,p_1}_{q_1}(\Omega)}\lesssim\|v\|_{\BB^{s,p}_q(\Omega)}.
\end{align*}
If $q > q_1,$ the embedding remains true provided the inequality \eqref{eq:embedding_exponents} is strict.
\end{theorem}

We also recall the trace theorem in the following form, see e.g. \cite{Adams2003}.
\begin{lemma}\label{lem:traceTheorem}
  Let $\Omega$ be a Lipschitz domain. Let $1<p<\infty$. Then there is a bounded linear operator $\tp{Tr}\colon \WW^{1,p}(\Omega)\to \WW^{1-\frac 1 p,p}(\p\Omega)$. In fact, it is possible to define $\tp{Tr}\, u$ to be the restriction of the $\H^{n-1}$ almost everywhere defined precise representative of $u$ on $\p\Omega$.
\end{lemma}
Finally we recall the following well-known result (see \cite{Evans1992} for the ingredients of the proof) which will justify extending $u$ by extensions of $g$.
We will later collect results concerning extensions preserving Besov-regularity in Section \ref{sec:oddeven_extension}.
\begin{lemma}\label{lem:extension}
Let $p\in[1,\infty]$. Let $V\Supset \Omega$ be an open, bounded set.
Suppose $u\in \WW^{1,p}(\Omega)$ and $v\in \WW^{1,p}(V)$ such that $\restr v\Omega \in u+\WW^{1,p}_0(\Omega)$. Then the map
\begin{align*}
w = \begin{cases}
	u \text{ in } \Omega \\
	v \text{ in } V\setminus\Omega
	\end{cases}
\end{align*}
	belongs to $\WW^{1,p}(V)$.
\end{lemma}

\subsection{\texorpdfstring{$V$}{V}-functions and inequalities}

For $p \in [1,\infty)$ and $\mu \geq 0,$ we will introduce the \emph{$V$-functions}
\begin{equation}
  V_{p,\mu} = (\mu^2 + \lvert z\rvert^2)^{\frac{p-2}4}z.
\end{equation} 
If $\mu$ is clear from context, we will write $V_{p,\mu} = V_p.$
The fundamental property these functions satisfy is that
\begin{equation}\label{eq:Vfunction_fundemental}
  \lvert V_{p,\mu}(z)\rvert^2  = (\mu^2+\lvert z\rvert^2)^{\frac{p-2}2}\lvert z\rvert^2 \sim \begin{cases} \lvert z\rvert^2 & \text{ if } \lvert z\rvert \leq \mu \\ \lvert z\rvert^p & \text{ if } \lvert z\rvert \geq \mu,\end{cases}
\end{equation} 
noting that $\lvert V_{p,0}(z)\rvert^2 = \lvert z\rvert^p$ when $\mu=0$.
This modular quantity naturally arises instead of the $p$-norm, which is the primary complication in the subquadratic ($p \leq 2$) case.

We will record several useful and well-known properties and estimates involving these functions.

\begin{lemma}[Pointwise estimates]
  For $1 \leq p < \infty$ and $\mu \geq 0,$ the following hold for all $z, z_1, z_2 \in \bb R^{Nn}$ and $\lambda>0,$ where the implicit constants depend on $p, \mu$ only.
  \begin{align}
    \min\{\lambda^2,\lambda^p\}\lvert V_{p,\mu}(\lambda z)\rvert^2 &\leq \lvert V_{p,\mu}(\lambda z)\rvert^2 \leq \max\{\lambda^2,\lambda^p\}\lvert V_{p,\mu}(\lambda z)\rvert^2 \label{eq:vfunction_homogeneity} \\
    \lvert V_{p,\mu}(z_1 + z_2)\rvert^2 &\lesssim \left( \lvert V_{p,\mu}(z_1)\rvert^2 + \lvert V_{p,\mu}(z_2)\rvert^2 \right)\label{eq:vfunction_additive} \\
    \lvert V_{p,\mu}(z_1)-V_{p,\mu}(z_2)\rvert^2  &\sim (\mu^2 + \lvert z_1\rvert^2 + \lvert z_2\rvert^2)^{\frac{p-2}2}\lvert z_1 - z_2\rvert^2 \label{eq:vfunction_difference}
  \end{align} 
  Furthermore for each $\delta \in (0,1)$ we also have Young's inequality which takes the form
  \begin{equation}\label{eq:vfunction_young}
    \lvert z \cdot w \rvert \leq \delta \lvert V_{p,\mu}(\lvert z\rvert)\rvert^2 + C(p)\delta^{1-\min\{2,p\}} \lvert V_{p^\prime,\mu^{p-1}}(\lvert w\rvert)\rvert^2.
  \end{equation} 
\end{lemma}
Properties \eqref{eq:vfunction_homogeneity}, \eqref{eq:vfunction_additive} are easy to check, and see for instance \cite{Giaquinta1986a} for \eqref{eq:vfunction_difference}.
The estimate \eqref{eq:vfunction_young} follows by Young's inequality applied to the $N$-function ${e_{p,\mu}(t) = (1+t^2)^{\frac{p-1}2}t^2}$ together with \eqref{eq:vfunction_homogeneity}, noting its conjugate satisfies $e_{p,\mu}^* \sim e_{p^\prime,\mu^{p-1}}.$

We will later need to compare $V$-functions with different exponents, for which it will be useful to record that
\begin{equation}\label{eq:Vfunction_pq}
  \lvert V_{q,\mu}(z)\rvert^2 \lesssim \lvert V_{p,\mu}(z)\rvert^2 + \lvert V_{p,\mu}(z)\rvert^{\frac {2q}p},
\end{equation} 
for $p \leq q.$ This follows by using \eqref{eq:Vfunction_fundemental} and distinguishing between the case of small and large $\lvert z\rvert$ separately.

We will also use the following Poincar\'e-Sobolev inequality adapted to $V$-functions.
To apply these near the boundary it will be necessary to ensure the associated constants do not change, which is recorded in the following result.
\begin{lemma}\label{lem:precise_poincaresobolev}
  Let $1 \leq p < n$ and suppose $D \subset \bb R^n$ is a bounded convex domain such that there exists $x_0 \in \Omega$ and $0<r<R<\infty$ such that $B_{r}(x_0) \subset D \subset B_{R}(x_0).$ Then for any $u \in \WW^{1,p}(\Omega)$ we have
  \begin{equation}\label{eq:poincaresobolev_norm}
    \norm{u - (u)_D}_{\LL^{p^*}(\Omega)} \leq C \norm{\D u}_{\LL^p(\Omega)}
  \end{equation} 
  where $p^* = \frac{np}{n-p}.$ Also for general $1 \leq p < \infty$ we have 
  \begin{equation}\label{eq:poincaresobolev_vfunction}
    \left(\dashint_D \left\lvert V_{p,\mu}\left( \frac{u-(u)_D}{R} \right) \right\rvert^{\frac{2n}{n-1}} \,\d x\right)^{\frac{n-1}n} \leq C\dashint_D \lvert V_{p,\mu}(\D u)\rvert^2 \,\d x.
  \end{equation} 
  In both cases the constant $C$ depends on $n, p,r/R$ only.
\end{lemma}
\begin{proof}
  Estimate \eqref{eq:poincaresobolev_norm} follows for instance from the corresponding result established for John domains in \cite[Theorem 5.1]{Bojarski1988}, noting that in particular convex domains are John domains as observed in \cite[Remark 2.4(c)]{Martio1978}.
  For \eqref{eq:poincaresobolev_vfunction}, applying the $p=1$ case to $e_{p,\mu}\left(\frac{u-(u)_D}R\right)$ gives
  \begin{equation}
    \begin{split}
      \left(\int_D e_{p,\mu}\left(\frac{u-(u)_D}R\right)^{\frac{n}{n-1}} \,\d x\right)^{\frac{n-1}n} 
      &\leq \frac{C}{\lvert D\rvert^{\frac1n}} \dashint_D e_{p,\mu}^\prime\left(\frac{\lvert u -(u)_{D}\rvert}R \right) \frac{\lvert \D u\rvert}R \,\d x \\
      &\leq C \dashint_D e_{p,\mu}\left( \frac{\lvert u -(u)_{D}\rvert}R \right) + e_{p,\mu}(\D u) \,\d x,
      \end{split}
  \end{equation} 
  where we have used \eqref{eq:vfunction_young} in the last line, and the fact that $\mathscr \lvert D\rvert^{\frac1n} \sim R \sim r$ where the implicit constant depends on $n$ and $r/R.$
  It remains to show the estimate
  \begin{equation}
    \dashint_D e_{p,\mu}\left( \frac{\lvert u -(u)_{D}\rvert}R \right) \,\d x \leq C \dashint_{D} e_{p,\mu}(\D u) \,\d x,
  \end{equation} 
  which is proved in \cite[Lemma 1]{Bhattacharya1991} (taking $F = e_{p,\mu}$) noting that $\diam(D) \sim  R.$
\end{proof}

\begin{lemma}\label{lem:annular_poincaresobolev}
  For $1 < p < \infty,$ for all $x_0 \in \overline{\bb R}^n_+$ and $0<r<s<\infty$ such that $(x_0)_n \not\in \left(\frac{s+r}2,r\right),$ defining the domain ${A = \bb R^n_+\cap \left(B_s(x_0) \setminus B_r(x_0)\right)}$ we have
    \begin{equation}
      \left( \int_{A} \lvert V_{p,\mu}(u)\rvert^{\frac{2n}{n-1}} \,\d x \right)^{\frac{n-1}n} \lesssim_{n,p} \int_{A} \left\lvert V_{p,\mu}\left( \frac{u}{s-r} \right) \right\rvert^2 + \lvert V_{p,\mu}(\D u)\rvert^2 \,\d x.
  \end{equation} 
  for all $u \in W^{1,p}(\Omega).$
\end{lemma}

\begin{proof}
  Setting $\theta = s-r,$ consider the cones
  \begin{equation}
    U_j = \bb R^n_+ \cap (x_0 + C_s(\theta,n_j) \setminus \overline{C_r(\theta,n_j)}),
  \end{equation} 
  and choose finitely many $\{n_j\} \subset S^{n-1}$ so that $\{U_j\}$ covers $B_s \setminus B_r.$ 
  We can also ensure that $\sum_j \chi_{U_j} \leq M$ with $M$ independent of $s,t.$
  Now we wish to show the corresponding Sobolev inequality on each $U_j,$ from which the result will follow.
  If $B_s(x_0) \subset \Omega,$ this can be done for instance by parametrising $B_1$ by spherical coordinates, and noting this parametrisation restricts to a diffeomorphism on each $U_j$ by our choice of angle $\theta,$ and applying Lemma \ref{lem:precise_poincaresobolev} there (noting this diffeomorphism maps each $U_j$ to a convex set).
  In the boundary case, our restriction on $(x_0)_n$ ensures each $U_j$ contains a cone $\tilde U_{i}$ associated to $\bb R^n_+ \cap (B_{\frac{s+r}2}(x_0) \setminus B_r(x_0))$ and becomes convex under this parametrisation. 
  Thus we can also apply Lemma \ref{lem:precise_poincaresobolev} in this case.
\end{proof}

We will frequently pass from differentiability of the $V$-functional to differentiability of the function itself. The following is elementary, but will be useful in the sequel.

\begin{lemma}\label{lem:Vfunc_diff}
  Suppose $s \in (0,1),$ $\mu>0,$ $2 \leq r < \infty$ and $1 < p < \infty.$ Given $\Omega \subset \bb R^n$ a bounded domain, let $u$ be weakly differentiable in $\Omega$ such that $V_{p,\mu}(\D u) \in \BB^{s,r}_{\infty}(\Omega).$
  Then if $p \geq 2$ we have $u \in \BB^{1+\frac{2s}p,\frac{rp}2}_{\infty}(\Omega)$ with the associated estimate
  \begin{equation}
    \seminorm{\D u}_{\BB^{\frac {2s}p,\frac{rp}2}_{\infty}(\Omega)} \leq \seminorm{V_{p,\mu}(\D u)}_{\BB^{s,r}_{\infty}(\Omega)}^{\frac 2p}.
  \end{equation} 
  If $p < 2,$ then $u \in \BB^{1+s,\frac{rp}2}_{\infty}(\Omega)$ with the associated estimate
  \begin{equation}
    \seminorm{\D u}_{\BB^{s,\frac{rp}2}_{\infty}(\Omega)} \lesssim \left( \mu^{\frac{2}{rp}}\,\lvert\Omega\rvert^{\frac{2}{rp}} + \norm{\D u}_{\LL^{\frac{rp}2}} \right)^{1-\frac p2}\seminorm{V_{p,\mu}(\D u)}_{\BB^{s,r}_\infty(\Omega)}.
  \end{equation} 
  Moreover, if $\frac{2n} r-2s>0$, $u\in \WW^{1,\frac{np}{\frac{2n} r-2s}}(\Omega)$ with the estimate
  \begin{align*}
  [\D u]_{\LL^\frac{np}{\frac{2n} r-2s}(\Omega)}\lesssim \begin{cases}\|V_{p,\mu}(\D u)\|_{B^{s,r}_\infty(\Omega)}^\frac 2 p \quad& \text{ if } p\geq 2\\
  \mu^\frac{np}{\frac{2n} r-2s}\lvert \Omega\rvert + \|V_{p,\mu}(\D u)\|_{B^{s,r}_\infty(\Omega)} & \text{ if } p\leq 2.
  \end{cases}
  \end{align*}
\end{lemma}
\begin{proof}
  The case $p \geq 2$ follows by noting that $\lvert\,\cdot\,\rvert^p = \lvert V_{p,0}(\cdot)\rvert^2 \leq \lvert V_{p,\mu}(\cdot)\rvert^2$ by definition of $V_{p,\mu}$. If $p < 2,$ for $h \in \bb R^n$ we have by H\"older and \eqref{eq:vfunction_difference} that
  \begin{equation}
    \begin{split}
      \int_{\Omega^h} \lvert \Delta_h \D u\rvert^{\frac{rp}2} \,\d x 
      &\leq \left( \int_{\Omega^h} (\mu + \lvert \D u_h\rvert + \lvert \D u\rvert)^{-\frac{(2-p)r}2} \lvert \Delta_h \D u\rvert^r \,\d x\right)^{\frac p2} \\
      &\qquad \times \left( \int_{\Omega^h}  (\mu + \lvert \D u\rvert + \lvert \D u_h\rvert)^{\frac{rp}2} \,\d x\right)^{\frac{2-p}2} \\
      &\lesssim \left( \int_{\Omega^h} \lvert \Delta_h V_{p,\mu}(\D u)\rvert^r \,\d x\right)^{\frac p2} \left( \int_{\Omega}  (\mu + \lvert \D u\rvert)^{\frac{rp}2} \,\d x\right)^{\frac{2-p}2} \\
      &\lesssim \lvert h\rvert^{\frac{prs}2} \seminorm{V_{p,\mu}(\D u)}_{\BB^{s,r}_{\infty}(\Omega)}^{\frac{pr}2} \left( \mu\, \lvert \Omega\rvert + \norm{\D u}_{\LL^{\frac{rp}2}(\Omega)}^{\frac{rp}2} \right)^{1-\frac p2},
    \end{split}
  \end{equation} 
  from which the conclusion follows. 
  
  Note that in the case $p\geq 2$, the moreover part follows directly from Sobolev embedding theorem and the fact that $u\in B^{1+\frac{2s} p,\frac{rp} 2}(\Omega)$. In case $p\leq 2$, using Sobolev embedding theorem,
  \begin{align*}
  \|V_{p,\mu}(\D u)\|_{B^{s,r}_\infty(\Omega)}\gtrsim \|V_{p,\mu}(\D u)\|_{\LL^\frac{rn}{n-sr}(\Omega)}.
  \end{align*}
  Further, for any $\tau\geq 1$,
  \begin{align*}
 \int_\Omega \lvert \D u\rvert^\frac{\tau p} 2\d x\leq \int_{\{x\in\Omega\colon \lvert \D u(x)\rvert \geq \mu\}} (\mu^2+\lvert \D u\rvert^2)^\frac{\tau(p-2)} 4\lvert \D u \rvert^{\tau}\d x + \mu^\frac{\tau p} 2\lvert \Omega\rvert.
  \end{align*}
  Combining the last two estimates, the claim follows.
\end{proof}

\subsection{Bounds on the integrand}\label{sec:bounds_int}

We collect some basic properties satisfied by the integrand $F=F(x,z),$ that we will use extensively in what follows. We first note that for $z\in \R^{m\times n}$ and almost every $x\in\Omega$,
\begin{align}
\label{eq:pnormlower}\lvert z\rvert^p-1\lesssim F(x,z).
\end{align}
By convexity of $F,$ we have \eqref{def:bounds2} implies that
\begin{equation}\label{eq:h2bound_derivative}
  \p_zF(x,z) \lesssim (1+\lvert z\rvert)^{q-1}.
\end{equation} 
We also show that \eqref{def:bounds1} implies a more familiar ellipticity condition.
\begin{lemma}\label{lem:hbound1}
  Let $1 < p < \infty,$ $\mu \geq 0,$ and suppose $f \colon \bb R^{m\times n} \to \bb R$ is $C^1$ and satisfies \eqref{def:bounds1}.
  Then we have
  \begin{equation}\label{eq:hbound1_quantiative}
    \nu(\mu^2+\lvert z\rvert^2+\lvert w\rvert^2)^\frac{p-2}{2}\lesssim \frac{F(x,z)-F(x,w)-\langle \partial_z F(x,w),z-w\rangle}{\lvert z-w\rvert^2}
  \end{equation} 
  for all $z, w \in \bb R^{m\times n}.$
\end{lemma}

\begin{proof}
  Suppose $F$ is $C^2,$ then for $z \neq 0$ we have
  \begin{equation}
    F^{\prime\prime}(z)[\zeta,\zeta] \geq \nu \, e_{\mu,p}^{\prime\prime}(\lvert z\rvert) \lvert \zeta\rvert^2 \gtrsim c_p \nu\, (\mu^2+\lvert z\rvert^2)^{\frac{p-2}2} \lvert \zeta\rvert^2
  \end{equation} 
  for all $\zeta \in \bb R^{Nn},$ where $e_{\mu,p}(t) = (\mu^2+t^2)^{\frac{p-2}2}t^2.$
  This follows from an explicit calculation.
  Now using the fundamental theorem of calculus we infer that for $z,w \in \bb R^{Nn}$ with $tz + (1-t)w \neq 0$ for all $t \in [0,1],$
  \begin{equation}
    \begin{split}
    F(z) - F(w) - F^\prime(w) \cdot (z-w) 
    &= \int_0^1 (1-t) F^{\prime\prime}(w + t(z-w)) \,\d t \\
    &\gtrsim \int_0^1 (1-t) (\mu^2 + \lvert w + t(z-w)\rvert^2)^{\frac{p-2}2} \lvert z - w\rvert^2 \,\d t \\
    &\gtrsim (\mu^2 + \lvert z\rvert^2 + \lvert w\rvert^2)^{\frac{p-2}2} \lvert z -w \rvert^2,
    \end{split}
  \end{equation} 
  and by approximation this holds for all $z,w \in \bb R^{m\times n}.$
  For general $f$ in $C^1,$ we will fix a standard mollifier $\rho_{\delta}$ on $\bb R^{m\times n}$ and apply the above to $F_{\delta} = F(x,\cdot) \ast \rho_{\delta}.$ Then we have
  \begin{equation}
    F_{\delta} - \nu \left(\lvert V_{p,\mu}(\cdot) \rvert^2 \ast \rho_\delta\right) \text{ is convex},
  \end{equation} 
  and for $\lvert z\rvert \geq 2\delta,$ we have
  \begin{equation}
    \left(\lvert V_{p,\mu}(\cdot) \rvert^2 \ast \rho_\delta\right)^{\prime\prime}(z)[\xi,\xi] \gtrsim c_p \left( \mu^2 + \lvert z\rvert^2 - \delta^2 \right)^{\frac{p-2}2}\lvert \xi\rvert^2 \gtrsim c_p \left( \mu^2 + \lvert z\rvert^2 \right)^{\frac{p-2}2}\lvert \xi\rvert^2,
  \end{equation} 
  where the implicit constant is independent of $\delta.$
  Therefore by the $C^2$ case proved above, \eqref{eq:hbound1_quantiative} holds for $F_{\delta}$ whenever $tz + (1-t)w \not\in \{z \in \bb R^{Nn} : \lvert z\rvert \geq 2\delta\}$ for all $t \in [0,1].$
  Sending $\delta \to 0$ we deduce that \eqref{eq:hbound1_quantiative} holds for $f$ and all $z,w \in \bb R^{m\times n}.$
\end{proof}

\subsection{Properties of the relaxed and the regularised functional}\label{sec:relaxedVsRegularised}
In this section, we collect some results regarding the relaxed functional $\overline \F_N$ and its relation to a regularised version of $\F$. For local versions of $\overline \F_N$ these results are well-known, c.f. Section 6 in \cite{Marcellini1989} and \cite{Esposito2004}. Moreover, for global versions incorporating Dirichlet boundary conditions, similar results were obtained in \cite{Koch2020}. The results in this section are minor modifications of those obtained in \cite{Koch2020}.
Define for $u\in \WW^{1,p}(\Omega)$ and $\e\in(0,1]$,
\begin{align}\label{def:Fe}
\F_\e(u)= \begin{cases}
			\F(u)+\e \int_\Omega \lvert \D u\rvert^q\d x &\text{ if } u\in \WW^{1,q}(\Omega)\\
			\infty &\text{ else}.
			\end{cases}
\end{align}
Write $F_\e(x,z)=F(x,z)+\e\lvert z\rvert^q$. If $F$ is such that $F(\cdot,z)$ is measurable for any $z\in \R^n$ and $F(x,\cdot)$ is continuously differentiable for almost every $x\in\Omega$, it is easy to see that the same properties hold for $F_\e$. Moreover, if $F$ satisfies \eqref{def:bounds1}-\eqref{def:bounds3}, then so does $F_\e$ with bounds that are independent of $\e$. 

Minimisers of $\F_\e(\cdot)$ and $\overline \F_N(\cdot)$ are related as follows:
\begin{lemma}[c.f. Lemma 6.4. in \cite{Marcellini1989}]\label{lem:convApproximate}  Suppose $F(x,z)$ satisfies \eqref{def:bounds1} and \eqref{def:bounds2}, and $f \in \WW^{-1,q^\prime}$, $g_N \in \WW^{-\frac1{q^\prime},q^\prime}(\p\Omega)$ satisfies \eqref{eq:compatibility_condition}.
Suppose $u$ is a relaxed minimiser of $\F(\cdot)$ in the class $\WW^{1,p}(\Omega)$ and $u_\e$ is the pointwise minimiser of $\F_\e(\cdot)$ in the class $\WW^{1,q}(\Omega)$. Then $\F_\e(u_\e)\to \overline\F_N(u)$ as $\e\to 0$. Moreover up to passing to a subsequence $u_\e\to u$ in $\WW^{1,p}(\Omega)$.
\end{lemma}
\begin{proof}
  Note that by our assumptions $\int_{\Omega} f\cdot u \,\d x$ and $\int_{\p\Omega} g_N\cdot u\,\d\H^{n-1}$ as defined for $u \in \WW^{1,q}(\Omega)$.
Hence existence and uniqueness of $u_\e$ follows from the direct method and strict convexity, respectively.
We further note that 
\begin{align*}
\overline\F_N(u)\leq \liminf_{\e\to 0}\F(u_\e) \leq \liminf_{\e\to 0}\F_\e(u_\e).
\end{align*}
 To prove the reverse implication note that for any ${v\in \WW^{1,q}(\Omega)}$,
\begin{align*}
\limsup_{\e \to 0} \F_\e(u_\e) \leq \lim_{\e \to 0} \F_\e(v)= \F(v)= \overline \F_N(v).
\end{align*}
By definition of $\overline\F_N$ the inequality above extends to all $v\in \WW^{1,p}(\Omega)$. In particular, it holds with the choice $v=u$. Thus $\F_\e(u_\e)\to \overline \F_N(u)$.

Using \eqref{def:bounds1} we may extract a (non-relabelled) subsequence of $u_\e$, so that $u_\e\rightharpoonup v$ weakly in $\WW^{1,p}(\Omega)$ for some $v\in \WW^{1,p}(\Omega)$. Note using our calculations above that $v$ is a relaxed minimiser of $\F(\cdot)$ in the class $\WW^{1,p}_g(\Omega)$. Using \eqref{def:bounds1} it is easy to see that for $w_1,w_2\in \WW^{1,q}_g(\Omega)$,
\begin{align}\label{eq:convexArg}
\overline \F_N\left(\frac{w_1+w_2}{2}\right)+\frac \nu p\|\D w_1-\D w_2\|_{\LL^p(\Omega)}^p\leq \frac 1 2\left(\overline \F_N(w_1)+\overline \F_N(w_2)\right).
\end{align}
Using the definition of $\overline \F_N$ and weak lower semicontinuity of norms, we see that this estimate extends to ${w_1,w_2\in \WW^{1,p}_g(\Omega)}$. In particular, $\overline \F_N$ is convex and so $u=v$. Moreover the choice $w_1=u,w_2=u_\e$ in the estimate shows that $u_\e\to u$ in $\WW^{1,p}(\Omega)$.
\end{proof}

A useful criterion for establishing the equality $\overline \F_N(u)=\F(u)$ is the following:
\begin{lemma}[cf. \cite{Buttazo1995}]\label{lem:lavrientiev}Let $1< p\leq q<\infty$.
For $u\in \WW^{1,p}(\Omega)$ with $\F(u)<\infty$, we have $\overline \F_N(u)=\F(u)$ if and only if there is a sequence $u_k\in \WW^{1,q}(\Omega)$ such that $u_k\rightharpoonup u$ weakly in $\WW^{1,p}(\Omega)$ and $\F(u_\e)\to \F(u)$ as $\e\to 0$. 
\end{lemma}

We close this section by showing that relaxed minimisers are very weak solutions of the Euler-Lagrange system. We first note an elementary bound.
\begin{lemma}\label{lem:improvedLowerBound}
Suppose $F$ satisfies \eqref{def:bounds1} and \eqref{def:bounds2}. Then
\begin{align*}
\p_z F(x,z)\cdot z\gtrsim \lvert z\rvert^p+\lvert\p_z F(x,z)\rvert^{q^\prime}-1.
\end{align*}
\end{lemma}
\begin{proof}
Introduce the Fenchel-conjugate of $F$,
\begin{align*}
F^\ast(x,\xi) = \sup_{\zeta} \langle \zeta,\xi\rangle - F(x,\zeta).
\end{align*}
By the equality case of the Fenchel-young inequality
\begin{align*}
\langle\p_z F(x,z),z\rangle = F(x,z)+F^\ast(x,\p_z F(x,z)).
\end{align*}
An elementary calculation shows that \eqref{def:bounds2} implies $F^\ast(x,\p_z F(x,z))\gtrsim \lvert \p_z F(x,z)\rvert^{q^\prime}-1$. This concludes the proof, recalling that we assumed also \eqref{def:bounds1}.
\end{proof}

We now adapt an argument from \cite{DeFilippis2020}.
\begin{lemma}\label{lem:relaxedEuler}
  Suppose $F$ satisfies \eqref{def:bounds1} and \eqref{def:bounds2}. Let $u$ be a relaxed minimiser of $\F_N(\cdot)$ or $\overline \F_D(\cdot)$ in the class $\WW^{1,p}(\Omega)$, with $f \in \WW^{-1,p^\prime}$ and $g_N \in \WW^{-\frac1{p^\prime},p^\prime}(\p\Omega)$, $g_D \in \WW^{1,q}(\Omega)$ in the respective cases. Then $\p_z F(x,\D u)\in \LL^{q^\prime}(\Omega)$ and
\begin{align*}
  \int_\Omega \p_z F(x,\D u)\cdot \D \psi-f \cdot \psi\d x+\int_{\p\Omega} g_N\cdot \psi \,\d\H^{n-1}=0
\end{align*}
for all $\psi \in W^{1,q}(\Omega)$ in the Neumann case, and
\begin{align*}
  \int_\Omega \p_z F(x,\D u)\cdot \D \psi-f \cdot \psi\d x=0
\end{align*}
for all $\psi \in W^{1,q}_0(\Omega)$ in the Dirichlet case.
\end{lemma}
\begin{proof}
In the Neumann case, by the direct method we obtain $u_\e \in \WW^{1,q}(\Omega)$ minimising $\F_\e(\cdot)$ in the class $\WW^{1,q}(\Omega)$ for each $\e>0.$ Without loss of generality we may assume that $\int_\Omega u_\e\d x=0$. Denote $\sigma_\e = \p_z F_\e(\cdot,\D u_\e)$ and $\mu_\e = \lvert \D u_\e\rvert^{q-2}\D u_\e$. Then $u_\e$ satisfies the Euler-Lagrange system
\begin{align*}
\int_\Omega (\sigma_\e+\e\mu_\e)\cdot \D\psi-f\cdot \psi\d x+\int_{\p\Omega} g_N\cdot  u\,\d\H^{n-1}=0 \qquad \forall \psi\in \WW^{1,q}(\Omega).
\end{align*}
Choose $\psi = u_\e$ and use Lemma \ref{lem:improvedLowerBound} applied to $F$, H\"older and Young's inequality to find that for any $\delta>0$,
\begin{equation}
  \begin{split}
  &\int_\Omega \lvert \D u_\e\rvert^p+\lvert\sigma_\e\rvert^{q^\prime}+\e \lvert \D u_\e\rvert^q\d x\\
  &\lesssim  \int_\Omega 1 + (\sigma_\e+\e \mu_\e)\cdot \D u_\e\d x\\
  &= \int_\Omega 1+f\cdot u_\e \d x+\int_{\p\Omega}g_N \cdot u_\e\,\d\H^{n-1}\\
  &\leq \delta \|u_\e\|_{\WW^{1,p}(\Omega)}+C(\delta)\left(1+\|g_N\|_{\WW^{-\frac1{p^\prime},p^\prime}(\p\Omega)}^{q^\prime}+\|f\|_{\WW^{-1,p^\prime}(\Omega)}^{p^\prime}\right).
  \end{split}
\end{equation}
Since the implicit constant is independent of $\eps>0$, by choosing $\delta>0$ sufficiently small we may rearrange to conclude
\begin{align*}
  \limsup_{\e\searrow 0}\int_\Omega \lvert\sigma_\e\rvert^{q^\prime}\d x\lesssim 1+\|g_N\|_{\WW^{-\frac1{p^\prime},p^\prime}(\p\Omega)}^{p^\prime}+\|f\|_{\WW^{-1,p^\prime}(\Omega)}^{p^\prime}.
\end{align*}
Since $\p_z F(x,\cdot)$ is Carath\'{e}odory and $\D u_\e\to \D u$ in measure by Lemma \ref{lem:convApproximate}, we also have $\sigma_\e\to F_z(\cdot,\D u)$ in measure. Finally note that $\sigma_\e + \e \mu_\e \rightharpoonup F_z(\cdot,\D u)$ in $\LL^1(\Omega)$, so the latter is row-wise divergence free and since it is also in $\LL^{q^\prime}(\Omega)$, the result is established.

The Dirichlet case is similar, the only difference is that we test the equation against $\psi = u_{\e}-g.$
\end{proof}

\section{Smoothing and relaxation}\label{sec:smoothing_relaxation}

We will establish some technical results concerning the smoothing operator introduced in \cite{Fonseca1997}.

\begin{lemma}\label{lem:regularised_extension}
  Suppose $1<p\leq q< \frac{np}{n-1}$ and let $\Omega$ be a Lipschitz domain. Let $x_0\in\Omega$, ${0<s/2<r<s}$ and let ${u\in \WW^{1,p}(\Omega)}.$ There is $R_0=R_0(\Omega)>0$ such that if ${r\leq R_0(\Omega)}$, there is a function $w\in \WW^{1,p}(\Omega)$ and $r^\prime<s^\prime$ with $s^\prime-r^\prime\sim s-r$, such that ${w=u}$ in ${\Omega\setminus B_{s^\prime}(x_0)}$ and $w=u$ in $B_{r^\prime}(x_0)\cap\Omega$. Denoting $A= \Omega_s(x_0)\setminus B_r(x_0)$ as well as ${A^\prime = \Omega_{s^\prime}(x_0)\setminus B_{r^\prime}(x_0)}$, we have the estimates
\begin{align}
  \begin{split}
  &\int_A \left\lvert V_{p,\mu}\left(\frac{w-a}{s-r}\right)\right\rvert^2 + \left\lvert V_{p,\mu}\left(\D (w-a)\right)\right\rvert^2 \d x   \\
  &\qquad\lesssim \int_A \left\lvert V_{p,\mu}\left(\frac{u-a}{s-r}\right)\right\rvert^2 + \left\lvert V_{p,\mu}\left(\D (u-a)\right)\right\rvert^2 \d x  \label{eq:testFunctionEstimates1} \\
  \end{split}\\
  \begin{split}
  &\int_{A^\prime} \left\lvert V_{p,\mu}\left(\frac{w-a}{s-r}\right)\right\rvert^\frac {2q} p + \lvert V_{p,\mu}(\D (w-a))\rvert^\frac{2q}p \d x \\
  &\qquad\lesssim (s-r)^{n\left(1-\frac q p\right)}\left(\int_A \left\lvert V_{p,\mu}\left(\frac{u-a}{s-r}\right)\right\rvert^2 + \left\lvert V_{p,\mu}\left(\D (u-a)\right)\right\rvert^2 \d x\right)^\frac q p\label{eq:testFunctionEstimates2}
  \end{split}
\end{align}
for any affine map $a \colon \bb R^n \to \bb R^n,$
where the implicit constants depend on $n$, $p$ and $\Omega$ only.

If in addition $\Omega = \bb R^n_+$ and $u\in \WW^{1,p}_0(\Omega)$, then $w$ can be chosen to also lie in $\WW^{1,p}_0(\Omega).$
\end{lemma}

\begin{proof}
  We will closely follow the proof of \cite[Lemma 2.4]{Fonseca1997}, with a slight modification to the extension operator defined in \cite[Lemma 2.2]{Fonseca1997}. In our setting we can take $\eta(x) = \lvert x\rvert,$ so for $s^\prime,t^\prime$ to be determined we define 
$$
w(x) = Tu(x) =\dashint_{\omega(x)} u(x+\theta(x)y)\d y
$$
where $\omega(x)=\{y\in B_1(0)\colon x+\theta(x)y\in \Omega\}$ and
$$
\theta(x)=\frac 1 {2}\max(0,\min\{\lvert x\rvert-r^\prime,s^\prime-\lvert x\rvert)\}
$$
Note that if $x \in A^\prime$ we have
\begin{equation}
  w(x) = \dashint_{\Omega \cap B_{\theta(x)}(x)} u(z) \,\d z,
\end{equation} 
so we see that $w=u$ on $\Omega \cap (\p B_{s^\prime} \cup \p B_{t^\prime})$ if $u \in C(\overline\Omega).$ By approximation this property extends in the sense of traces to all $u \in W^{1,p}(\Omega).$
We can also see that $T$ preserves affine maps, that is $Ta = a.$

Now by the uniform cone condition, there is $R_0>0$ such that $\lvert \Omega_r(x) \rvert \sim \lvert B_r(x)\rvert$ for all $r \leq R_0$ and $x \in \overline\Omega.$
Using this we can now argue exactly as in the interior case to obtain the corresponding estimates for $p$-norms.
For the claimed estimates \eqref{eq:testFunctionEstimates1}, \eqref{eq:testFunctionEstimates2} involving $V$-functions, we can argue as in \cite[Lemma 6.5]{Schmidt2008} when $p \leq 2,$ and if $p \geq 2$ we simply note that $V_{p,\mu}(z)\sim \lvert z|^2+\lvert z|^p$ (see also \cite[Lemma 2.3]{Passarelli1996}).

We now assume that $u\in \WW^{1,p}_0(\Omega)$. Using the first part, we obtain $r^\prime<s^\prime$ with $s^\prime-r^\prime\sim s-r$ such that $w^\prime=u$ in $\Omega\setminus B_{s^\prime}(x_0)$ and $w^\prime=u$ in $B_{r^\prime}(x_0)\cap\Omega$. Moreover, \eqref{eq:testFunctionEstimates1} and \eqref{eq:testFunctionEstimates2} are satisfied. It remains to modify $w^\prime$ so it vanishes on $\p\Omega.$

Let $\{Q_i\}$ be a WB-covering for $A^\prime$ using Theorem \ref{thm:covering}, and $\{\psi_i\}_{i\in I}$ the associated partition of unity guaranteed by Theorem \ref{thm:WBunity}. We set
$$
w = u \chi_{A\setminus A^\prime}+\sum_{i\in I} v_i\star \rho_{\delta_i}\psi_i \quad\text{ where } \delta_i = \delta\lvert Q_i\rvert^\frac 1 n,
$$
where
$$
v_i = \begin{cases}
	0 \quad&\text{ if } d(Q_i,\p\Omega)<(s^\prime-r^\prime)/1000 \,\text{ and }\, d(Q_i,\p\Omega)<d(Q_i,\p A^\prime\setminus \p\Omega)\\
	w^\prime &\text{ else }.
	\end{cases}
$$

We analyse the boundary behaviour of $w$ on $\p A^\prime$, for which we assume $\Omega = \bb R^n_+.$
By replacing $B_s \setminus B_r$ by either $B_s \setminus B_{\frac{s+r}2}$ or $B_{\frac{s+r}2} \setminus B_r$, we can assume Lemma \ref{lem:annular_poincaresobolev} holds.
By a density argument (which is allowed since $q\leq \frac{np}{n-1}$) we may assume that $v\in C(\overline{\Omega})$. Note for $x\in A^\prime$,
\begin{align*}
\lvert w(x)-w^\prime(x)\rvert\leq \sum_{i\in I: x\in (1+\delta)Q_i} \lvert v_i\star \rho_{\delta_i}(x)-v(x)\rvert\psi_i
\leq M \max_{y\in B_r(x)} \lvert w^\prime(y)-w^\prime(x)\rvert
\end{align*}
where $r=\max\{\,\delta_i\in I, x\in (1+\delta)Q_i\,\}.$ But if $x\in (1+\delta)Q_i$, then it holds that ${\delta_i \lesssim \lvert Q_i\rvert^\frac 1 n\sim d(x,\p A^\prime)}$. It follows by definition of $v_i$ that $w\in \WW^{1,p}_v(A^\prime)=\WW^{1,p}_u(A^\prime)$. In particular, extending $w$ by $u$ and using Lemma \ref{lem:extension}, we obtain $w\in \WW^{1,p}_0(\Omega)$.

Now using the additivity property \eqref{eq:vfunction_additive}, the finite overlap property \eqref{def:WBcoverMultiplicity} and \eqref{eq:testFunctionEstimates1},
\begin{equation}
  \begin{split}
    \int_A \left\lvert V_{p,\mu}\left( \frac{w-a}{s-r} \right) \right\rvert^2 \d x 
    &\leq \int_{A \setminus A^\prime} \left\lvert V_{p,\mu}\left(\frac{u-a}{s-r}\right)\right\rvert^2 \d x + \sum_{i \in I} \int_{Q_i} \left\lvert V_{p,\mu}\left(\frac{(v_i-a) \star \rho_{\delta_i}}{s-t}\right)\right\rvert^2 \d x\\
    &\leq \int_A \left\lvert V_{p,\mu}\left( \frac{u-a}{s-r} \right) \right\rvert^2 \d x + \sum_{i \in I} \int_{(1+\delta)Q_i} \left\lvert V_{p,\mu}\left(\frac{v_i-a}{s-r}\right)\right\rvert^2 \d x \\
    &\leq \int_A \left\lvert V_{p,\mu}\left(\frac{u-a}{s-r}\right) \right\rvert^2 \d x.
  \end{split}
\end{equation} 
Next, we note that
\begin{align*}
  \D (w-a) = \sum_{i\in I} \D (v_i-a)\star \rho_{\delta_i}\psi_i + \sum_{i\in I} (v_i-a)\star\rho_{\delta_i}\D\psi_i,
\end{align*}
in $A^\prime,$ which follows by noting that $a = \sum_{i \in I} a \star \rho_{\delta_i}\psi_i.$
Similarly to above we can estimate
\begin{equation}
  \begin{split}
    &\int_{A^\prime}\lvert V_{p,\mu}(\D(w-a))\rvert^2 \d x\\
    &\lesssim \sum_{i\in I} \int_{Q_i} \lvert V_{p,\mu}(\D (v_i-a)\star\rho_{\delta_i})\rvert^2 + \sum_{i \in I}\left\lvert V_{p,\mu}\left(\frac{(v_i-a)\star\rho_{\delta_i}}{\lvert Q_i\rvert^{-\frac1n}}\right)\right\rvert^2 \d x,
  \end{split}
\end{equation} 
noting that $\lvert \D\psi_i\rvert \leq C \lvert Q_i\rvert^{-\frac1n}.$ 
Applying Lemma \ref{lem:precise_poincaresobolev} on each $Q_i$ and using \eqref{eq:vfunction_additive} we have
\begin{equation}
  \begin{split}
&    \int_{Q_i} \left\lvert V_{p,\mu}\left( \frac{(v_i-a) \star \rho_{\delta_i}}{\lvert Q_i\rvert^{\frac1n}} \right) \right\rvert^2 \,\d x \\    
    &\lesssim  \int_{(1+\delta)Q_i} \lvert V_{p,\mu}(\D(v_i-a))\rvert^2 \,\d x + \left\lvert V_{p,\mu}\left( \frac{(\lvert v_i-a\rvert )_{(1+\delta)Q_i}}{\lvert Q_i\rvert^{\frac1n}} \right) \right\rvert^2 \\
    &= A_{i,1} + A_{i,2}.
  \end{split}
\end{equation} 
To estimate the latter term, we apply Lemma \ref{lem:precise_poincaresobolev} followed by Jensen's inequality to $\varphi(t) = \lvert V_{p,\mu}(t^{\frac{n-1}n})\rvert^{\frac{2n}{n-1}},$ noting that $\varphi(t) \sim e_{1,p}(t).$ This gives,
\begin{equation}
  A_{i,2} \lesssim \left\lvert V_{p,\mu}\left( \left(\dashint_{(1+\delta)Q_i} \lvert v_i-a\rvert^{\frac{n}{n-1}} \,\d x\right)^{\frac{n-1}{n}} \right)\right\rvert^ 2 \,\d x
  \lesssim \left( \dashint_{(1+\delta)Q_i} \lvert V_{p,\mu}(v-a)\rvert^{\frac{2n}{n-1}} \,\d x \right)^{\frac{n-1}n}.
\end{equation} 
Now combining the above estimates and summing we have
\begin{align*}  
&    \int_{A^\prime} \lvert V_{p,\mu}(\D (w-a))\rvert^2 \,\d x\\
    \lesssim& \sum_{i\in I}\left( \int_{Q_i} \lvert V_{p,\mu}(\D (v_i-a)\star\rho_{\delta_i})\rvert^2 \,\d x + \norm{V_{p,\mu}((v_i-a)\star\rho_{\delta_i})}_{\LL^{\frac{2n}{n-1}}(Q_i)}^2\right) \\
    \lesssim& \sum_{i\in I}\int_{(1+\delta)Q_i} \lvert V_{p,\mu}(\D (v_i-a))\rvert^2 \,\d x + \left( \sum_{i \in I} \int_{(1+\delta)Q_i} \lvert V_{p,\mu}(v_i-a)\rvert^{\frac{2n}{n-1}} \,\d x\right)^{\frac{n-1}{n}} \\
    \lesssim& \int_A \lvert V_{p,\mu}(\D (w^\prime-a)\rvert^2 \,\d x + \left( \int_A \lvert V_{p,\mu}(w^\prime-a)\rvert^{\frac{2n}{n-1}} \,\d x \right)^{\frac{n-1}n}.
\end{align*} 
where we have used H\"older's inequality for sequences and the finite overlap property in the last two lines. To complete the estimate, we apply Lemma \ref{lem:annular_poincaresobolev} which gives
\begin{equation}
  \left( \int_A \lvert V_{p,\mu}(w^\prime-a)\rvert^{\frac{2n}{n-1}} \,\d x \right)^{\frac{n-1}n} \leq \int_A \lvert V_{p,\mu}(\D (w^\prime-a))\rvert^2 + \left\lvert V_{p,\mu}\left( \frac{w^\prime-a}{s-r} \right) \right\rvert^2 \,\d x.
\end{equation} 
This gives the estimate
\begin{equation}
  \int_{A} \lvert V_{p,\mu}(\D (w-a))\rvert^2 \,\d x \lesssim \int_A \left\lvert V_{p,\mu}\left(\frac{w^\prime-a}{s-r}\right)\right\rvert^2 + \left\lvert V_{p,\mu}\left(\D (w^\prime-a)\right)\right\rvert^2 \d x,
\end{equation} 
and so \eqref{eq:testFunctionEstimates1} follows by the corresponding estimates for $w^\prime.$
For \eqref{eq:testFunctionEstimates2} a similar argument to above using $\lvert V_{p,\mu}(\cdot))\rvert^{\frac{2q}p}$ gives
\begin{equation}
  \int_{A^\prime} \lvert V_{p,\mu}(\D(w-a))\rvert^{\frac{2q}p} \,\d x \lesssim \int_{A^\prime} \left\lvert V_{q,\mu}\left(\frac{w^\prime-a}{s-r}\right)\right\rvert^\frac{2q}p + \left\lvert V_{q,\mu}\left(\D (w^\prime-a)\right)\right\rvert^\frac{2q}p \d x,
\end{equation} 
and so the result follows by the estimates for $w^\prime.$
\end{proof}

\begin{remark}\label{rem:boundaryConditionExtension}
Inspecting the proof of Lemma 2.4 in \cite{Fonseca1997}, it holds that
$$
\limsup_{\e\to 0}\,\frac1{\e}\int_{\Omega \cap B_{r+\e}(x_0)\setminus B_{r-\e}(x_0)}\lvert \D \tilde u\vert^p\d x<\infty
$$
for $r\in\{r^\prime,s^\prime\}$.
\end{remark}

We now adapt the additivity property for relaxed functions established in \cite[Lemma 7.7]{Schmidt2009} to hold near the boundary.
For the Neumann case we will use the mixed version $\F_M$ as defined in \eqref{eq:mixed_relaxed}.
\begin{lemma}\label{lem:additivityNeumann}
Let $1<p\leq q<\frac{np}{n-1}$. Suppose $\Omega \subset \bb R^n$ is a bounded domain and that $F$ satisfies \eqref{def:bounds2} and \eqref{eq:pnormlower}. Suppose $u\in \WW^{1,p}(\Omega)$ and let $\tilde u$ be a $\WW^{1,p}$-extension of $u$ to $\R^n$ such that $\|\tilde u\|_{\WW^{1,p}(\R^n)}\lesssim \|u\|_{\WW^{1,p}(\Omega)}$. If the boundary condition
\begin{equation}\label{eq:additivity_condition}
\limsup_{\e\to 0}\,\frac1{\e}\int_{B_{s+\e}(x_0)\setminus B_{s-\e}(x_0)}\lvert \D \tilde u\vert^p\d x<\infty
\end{equation}
holds, then we have
\begin{align}\label{eq:additivity_neumann}
  \overline\F_N(u,\Omega)=\overline\F_M(u,\Omega_s(x_0),\Omega \cap \p B_s(x_0))+\overline\F_M(u,\Omega\setminus \overline{B_s(x_0)},\Omega \cap \p B_s(x_0)).
\end{align}
If, in addition $u\in \WW^{1,p}_g(\Omega)$ where $g\in \WW^{1,q}(\Omega)$, then
\begin{align}\label{eq:additivity_dirichlet}
\overline\F_D(u,\Omega)=\overline\F_D(u,\Omega_s(x_0))+\overline\F_D(u,\Omega\setminus \overline{B_s(x_0)}).
\end{align}
\end{lemma}
\begin{proof}
  Let $\{u_k\}\in \WW^{1,q}(\Omega)$ be such that 
  $$\overline\F_M(u,\Omega_s(x_0),\Omega \cap \p B_s(x_0))=\lim_{k\to\infty} \F(u_k,\Omega,\Omega \cap \p B_s(x_0))$$
   with $u_k\rightharpoonup u$ in $\WW^{1,p}(\Omega_s(x_0))$. Extend each $u_k$ to a $\WW^{1,q}$ function on $\R^n$ denoted $\tilde u_k$ in such a way that
\begin{align*}
\|\tilde u_k\|_{\WW^{1,q}(\Omega)}\lesssim \|u_k\|_{\WW^{1,q}(\Omega)}\\
\tilde u_k \rightharpoonup \tilde u \,\text{ in }\, \WW^{1,p}(\R^n).
\end{align*}
 Due to \cite[Lemma 7.7]{Schmidt2009}, there is a sequence $\{w_k\}\subset \WW^{1,q}(B_s(x_0))\cap \WW^{1,p}_{\tilde u}(B_s(x_0))$ such that $w_k\rightharpoonup \tilde u$ weakly in $\WW^{1,p}(B_s(x_0))$ and
 $$
 \lvert \F(w_k,B_s(x_0))-\F(\tilde u_k,B_s(x_0))\rvert\to 0
 $$ 
 as $k\to\infty$. 
 While \cite[Lemma 7.7]{Schmidt2009} concerns autonomous integrands of the form $F=F(z),$ an inspection of the proof reveals that we only need to assume the growth bounds \eqref{def:bounds2}, \eqref{eq:pnormlower} are satisfied.
 Hence abusing notation and using $w_k$ to denote the restriction of $w_k$ to $\Omega_s(x_0)$, we deduce $w_k\rightharpoonup \tilde u$ weakly in $\WW^{1,p}(\Omega_s(x_0))$ and further
 \begin{equation}
   {\lim_{k\to\infty}\F(w_k,\Omega_s(x_0)) = \overline \F_M(u,\Omega_s(x_0),\Omega\cap \p B_s(x_0))}.
 \end{equation} 
 
 By a similar argument and using \cite[Lemma 7.8]{Schmidt2009}, we obtain $v_k\in \WW^{1,q}(\Omega\setminus \overline{B_s(x_0)})$ with boundary value $u$ on $\p B_s(x_0)\cap\Omega$ such that $v_k\rightharpoonup u$ weakly in $\WW^{1,p}(\Omega\setminus \overline{B_s(x_0)})$ and
 \begin{equation}
   \lim_{k\to\infty}\F(v_k,\Omega\setminus\overline{B_s(x_0)}) = \overline \F_M(u,\Omega\setminus\overline{B_s(x_0)},\Omega \cap \p B_s(x_0)).
 \end{equation} 

Composing the sequences $u_k$ and $v_k$, we deduce
$$
\overline\F_M(u,\Omega)\leq\overline\F_M(u,\Omega_s(x_0),\Omega \cap \p B_s(x_0))+\overline\F_M(u,\Omega\setminus \overline{B_s(x_0)},\Omega \cap \p B_s(x_0)).
$$
The other direction is immediate from the definition.

We now turn to the Dirichlet case, where we need to modify the sequences $\{w_k\}$ and $\{v_k\}$ slightly in order to preserve the boundary condition $w_k= g_D$ on $\p\Omega\cap B_s(x_0)$ and $v_k = g_D$ on $\p\Omega\setminus B_s(x_0)$ respectively. We only illustrate how to do so in the case of $\{w_k\}$. For $\{v_k\}$ the argument is similar.
Thus we assume that $u\in \WW^{1,p}_{g_D}(\Omega)$. Let $\{u_k\}\in \WW^{1,q}(\Omega)$ be such that $\overline\F_D(u,\Omega_s(x_0))=\lim_{k\to\infty} \F(u_k,\Omega_s(x_0))$ with $u_k\rightharpoonup u$ in $\WW^{1,p}(\Omega_s(x_0))$. Extending $u_k$ by $g_D$ to $B_s(x_0)$ and applying \cite[Lemma 7.7]{Schmidt2009}, we obtain $\{w_k\}\subset \WW^{1,q}(B_s(x_0))\cap \WW^{1,p}_{\tilde u}(B_s(x_0))$ such that $w_k\rightharpoonup \tilde u$ weakly in $\WW^{1,p}(B_s(x_0))$ and
 $$
 \lvert \F(w_k,\Omega_s(x_0))-\F(u_k,\Omega_s(x_0))\rvert\to 0.
 $$ 
 Let $\{Q_i\}_{i\in I}$ be a WB-partition of $B_s(x_0)$ and $\{\psi_i\}_{i\in I}$ be a partition of unity associated to it. Set
 $$
 w_{k,\e} = \sum_{i\in I} \tilde w_{k,\e}\psi_i\quad \text{where} \quad \tilde w_{k,\e}=\begin{cases}
		w_k \text{ if } d(Q_i,\p\Omega)>\e\\
		g_D \text{ else }.
 		\end{cases}
 $$
Arguing as in the proof of Lemma \ref{lem:regularised_extension}, we see that $w_{k,\e}\in \WW^{1,q}_u(\Omega_s(x_0))$ and $w_{k,\e}\to w_k$ in $\WW^{1,q}(\Omega_s(x_0))$ as $\e\to 0$. Moreover, due to continuity of $\F$ in $\WW^{1,q}$, it holds that ${\F(w_{k,\e},\Omega_s(x_0))\to \F(w_k)}$ as $\e\to 0$. By a diagonal subsequence argument, we thus obtain a sequence ${\{w_k\}\subset \WW^{1,q}_u(\Omega_s(x_0))}$ such that $w_k\rightharpoonup u$ weakly in $\WW^{1,p}(\Omega_s(x_0))$ and
 $$
 \lvert \F(w_k,\Omega_s(x_0))-\F(u_k,\Omega_s(x_0))\rvert\to 0.
 $$ 
 This concludes the proof.
 \end{proof}

\begin{remark}
  In Lemma \ref{lem:regularised_extension} we assumed our domain was flat to ensure the extension operator preserved boundary values, which will be sufficient in what follows.
  In contrast Lemma \ref{lem:additivityNeumann} holds for general domains, even in the Dirichlet case; this is because we do not use the boundary extension operator, and instead adapt \cite[Lemma 7.7]{Schmidt2009} from the interior case.
\end{remark}

\subsection{Reduction to a flat boundary}\label{sec:flattening}

For many of our results we will assume that our boundary is locally flat, and consider the case where $\Omega = U^+ =  U \cap \bb R^n_+$ for some bounded open subset $U \subset \bb R^n$ containing $B_1(0)$.
In this section we will outline how to reduce to this setting, so let $\Omega \subset \bb R^n$ be a bounded $C^{1,\alpha}$ domain, and let $F = F(x,z)$ satisfy \eqref{def:bounds1}--\eqref{def:bounds31}.

For $x_0 \in \partial\Omega$ let $R>0$ be sufficiently small so there exists a bi-$C^{1,\alpha}$ mapping $\Psi \colon B_R(x_0) \to U \subset \bb R^n$ such that $\Psi(\Omega_R(x_0)) = U^+$ and $\Psi(\p\Omega \cap B_R(x_0)) = U \cap \{x_n=0\}.$ 
Now writing $\Phi = \Psi^{-1},$ in the Neumann case a change of variables gives $\widetilde u = u \circ \Phi$ satisfies $\F(u,\Omega_R(x_0)) = \widetilde{\F}(\tilde u),$ where
\begin{equation}\label{eq:flattened_functional}
  \begin{split}
    \widetilde{\F}(v) 
    &= \int_{U^+} \left(F\left(\Phi(x), \D v \cdot \D\Phi(x)^{-1}\right) - f(\Phi(x)) \cdot v\right) \lvert \det \D \Phi\rvert \,\d x \\
    &\quad+ \int_{U \cap \{x_n=0\}} g_N(\Phi(x)) \cdot v\, \lvert \det \D_{\tau}\Phi\rvert \,\H^{n-1},
 \end{split}
\end{equation} 
over all $v \in \WW^{1,p}(U^+,\R^m)$ such that $v = \tilde u$ on $\p U^+ \cap \R^n_+;$ here $\D_{\tau}$ denotes the vector of tangential derivatives along $\p\Omega.$ 
The Dirichlet case is analogous, except that we drop the final term on the boundary.

Note that the associated integrand $\widetilde F(x,z) = F\left(\Phi(x), \D v \cdot \D\Phi(x)^{-1}\right)\lvert \det \D \Phi\rvert$ inherits the properties \eqref{def:bounds1}--\eqref{def:bounds31} (up to possibly changing the constants appearing within them), and the Besov regularity of $f$ is preserved.
Additionally if $F = F_0(x,\lvert z\rvert)$ is radial, the same will hold for $\widetilde F.$

We now claim that if $u$ is a relaxed minimiser of $\overline{\F}_N$ or $\overline{\F}_D$, by shrinking $R$ if necessary we have $\tilde u$ is a relaxed minimiser of $\widetilde{\F}$ on $U^+.$
To achieve this we choose $R>0$ so that \eqref{eq:additivity_condition} holds, using Remark \ref{rem:boundaryConditionExtension}.
In the Neumann case, using \eqref{eq:additivity_neumann} we have $u$ minimises the mixed problem $\overline\F_M(u,\Omega_s(x_0),\Omega \cap \p B_s(x_0))$, whereas in the Dirichlet case we use \eqref{eq:additivity_dirichlet} to see that $u$ minimises $\overline\F_D(u,\Omega_s(x_0))$.
  Now composing the minimisers $u_\e$ of the regularised functional \eqref{def:Fe} with a flattening map $\Phi = \Psi^{-1}$, we have $\tilde u_{\e} = u_\e \circ \Phi$ minimises
  \begin{equation}
    v \mapsto \widetilde{\F}(v, U^+) + \e \int_{U^+} \lvert \D v \cdot \D\Phi(x)^{-1}\rvert^q\, \lvert \det \D \Phi\rvert\,\d x,
  \end{equation} 
  with $\widetilde{\F}(v)$ as in \eqref{eq:flattened_functional}.
  Now sending $\eps \to 0$ and noting $v \mapsto v \circ \Phi$ is weak-to-weak continuous in $W^{1,p}$, we deduce that $\tilde u$ is a relaxed minimiser of $\widetilde{\F}_N,$ $\widetilde{\F}_D$ respectively on $U^+.$

\section{Regularity of minimisers of the relaxed functional}\label{sec:relaxed}
\subsection{Basic result: Neumann case}\label{sec:basicNeumann}
The aim of this section is to prove Theorem \ref{thm:regularityRelaxed}. For the convenience of the reader we recall the statement:
\neumannBasic*

The proof follows the argument in \cite[Section 3]{Koch2020} for the Dirichlet case closely, and we focus on the key differences.
\begin{proof}
We will write $g = g_N$ to simplify notation. The key is to prove the following a-priori estimate:
Let $v_\e$ be the minimiser of $\F_\e(\cdot)$ in the class $\WW^{1,q}(\Omega)$. Note that $v_\e$ exists by the direct method and using the strict convexity deriving from \eqref{def:bounds1}. Then for any $0\leq \beta<\alpha,$ we claim there is $\gamma>0$ such that the estimate
\begin{align*}
  \|v_\e\|_{\WW^{1,\frac{np}{n-\beta}}(\Omega)}\lesssim \left(1+\F_\e(v_\e)+\|f\|_{\LL^{q^\prime}(\Omega)}^{q^\prime}+ \|g\|_{\WW^{\alpha-\frac1{q^\prime},q^\prime}(\p\Omega)}^{q^\prime}\right)^\gamma 
\end{align*} 
holds, with the implicit constant independent of $\e$ and $\gamma$.

Let $\rho_0>0$ and $\mathbf n\colon \R^n\to S^{n-1}$ be so that the uniform cone property \eqref{eq:uniformCone} holds. Possibly reducing $\rho_0$, assume without loss of generality that $\Omega+B_{3\rho_0}(x)\subset B(0,R)$ for all $x\in \Omega$. 
 Here $B(0,R)\Supset\Omega$ is the ball defined in Section \ref{sec:WBLipschitz}.
Given ${x_0\in\Omega}$, let $0\leq {\phi=\phi_{x_0,\rho_0}}\leq 1$ be a smooth cut-off supported in $B_{2\rho_0}(x_0)$ with $\phi(x)=1$ in $B_{\rho_0}(x_0)$ and ${\lvert \D\phi(x)\rvert\leq C {\rho_0}^{-k}}$  for some $C>0$.
Given functions $v$ defined on $\R^n$ and $h\in\R^n$ introduce
\begin{align*}
  T_h v &= \phi v_{h}+(1-\phi)v.
\end{align*}

As in the Dirichlet case, for $h\in C_{\rho_0}(\theta_0,- n(x_0))$, we note using Lemma \ref{lem:hbound1} the lower bound
\begin{align}\label{eq:lowerEstimate}
\F_\e(T_h v_\e)-\F_\e( v_\e)\gtrsim\|V_{p,\mu}(\D v_{\e,h})-V_{p,\mu}(\D v_\e)\|_{\LL^2(B_{\rho_0}(x_0))}^2.
\end{align}

The proof will conclude as in the Dirichlet case once we show that for every $x_0\in \R^n$ there is a constant $C=C(n,\rho_0,\Lambda,\Omega)$ such that for all $v\in \WW^{1,q}(\Omega)$
\begin{equation}\label{clm:diffEstimate}
  \begin{split}
    &\sup_{h\in C_{\rho_0}(\theta_0, n(x_0))}\frac{\F_\e(T_h v)-\F_\e(v)}{\lvert h\rvert^\alpha}\\
    &\qquad\leq C\left(1+\|\D v\|_{\LL^q(\Omega)}^q + \|g\|_{\WW^{\alpha-\frac 1 {q^\prime},q^\prime}(\p\Omega)}^{q^\prime}+\|f\|_{\BB_{\infty}^{\alpha-1,q^\prime}(\Omega)}^{q^\prime}\right)
  \end{split}
\end{equation}
We note that if $p < 2,$ we combine \eqref{eq:lowerEstimate}, \eqref{clm:diffEstimate} and use the improved differentiability at the level of $V$-functions as is done in \cite{Esposito2004}. Moreover, we comment that \eqref{eq:regRelaxed1} is a consequence of \eqref{eq:regRelaxed2} and Lemma \ref{lem:Vfunc_diff}.

Indeed, given $v\in \WW^{1,q}(\Omega)$, let $\tilde v$ be a $\WW^{1,q}$-extension of $v$ to $\bb R^n.$ Then considering ${h\in C_{\rho_0}(\theta_0,n(x_0))},$ evidently $T_h\tilde v \in \WW^{1,q}(\Omega)$ and so we can write
\begin{align*}
\F_\e(T_h v)-\F_\e(v) = &\int_\Omega F_\e(x,T_h \D v +\D\phi \otimes(v_h-v))-F_\e(x,T_h \D v)\,\d x\\
&\qquad+\int_\Omega F_\e(x,T_h \D v)-F_\e(x,\D v)\,\d x\\
&\qquad -\int_\Omega f\cdot(T_h v-v)\,\d x+\int_{\p\Omega} g\cdot(T_h v-v)\,\d\H^{n-1}\\
=& A_1 + A_2+ A_3+A_4.
\end{align*}
The terms $A_1$ and $A_2$ are estimated exactly as in the Dirichlet case.
For $A_3$ we make a small refinement and estimate
\begin{equation}
  \lvert A_3 \rvert \lesssim \lvert h\rvert^{\alpha} \norm{f}_{\BB_{\infty}^{\alpha-1,q^\prime}} \norm{v}_{\WW^{1,q}(\Omega)},
\end{equation} 
noting this holds when $\alpha = 0, 1$ and interpolating $f \mapsto \lvert A_3\rvert$ in $[\cdot,\cdot]_{1-\alpha,\infty}$ for the general case.
Concerning $A_4$, we claim that for $\alpha \in (0,1)$, we have
\begin{equation}\label{eq:A4Interpolated}
  \lvert A_4\rvert = \left\lvert \int_{\partial\Omega} g \cdot (T_h v - v) \,\d\H^{n-1}\right\rvert \lesssim \lvert h\rvert^{\alpha} \norm{g}_{\WW^{\alpha-\frac 1 {q^\prime},q^\prime}(\p\Omega)}^{q^\prime} \norm{\D v}_{\LL^q}.
\end{equation} 
We will establish this by similarly interpolating between $\alpha  = 0,1.$ When $\alpha = 0$ we have
\begin{equation*}
    \lvert A_4 \rvert \leq \norm{g}_{\BB_{q^\prime}^{-1+\frac1{q},q^\prime}(\partial\Omega)} \norm{T_h v - v}_{\BB^{1-\frac1q,q}_q(\partial\Omega)} \lesssim \norm{g}_{\BB_{q^\prime}^{-\frac1{q^\prime},q^\prime}(\partial\Omega)} \norm{v}_{\WW^{1,q}(\Omega)},
\end{equation*} 
where we have used Lemma \ref{lem:traceTheorem}. On the other hand, if $\alpha=1$, let $\tilde g \in \BB_{\infty}^{1,q^\prime}(\Omega)$ be an extension of $g$ to $\Omega$. Denote for $\sigma\in \p \Omega\cap B_{2\rho_0}$, $h_{\sigma}(t)= \sigma - tn$ and put ${t_\mathrm{max}(\sigma) = \max\{t > 0 \colon \sigma - t n\in B_{2\rho_0}\}}$. Then, using the fundamental theorem of calculus, the coarea formula and the interior cone condition we have
\begin{align*}
  \lvert A_4\rvert =& \left\lvert\int_{\p\Omega\cap B_{2\rho_0}} \int_0^{t_{\max}(\sigma)} \p_n \left(\tilde g\cdot(T_h v - v)(h_{\sigma}(t))\right)\,\d t\d\mathscr H^{n-1}\right\rvert\\
  =& \left\lvert\int_{\Omega\cap B_{2\rho_0}} \p_n \tilde g\cdot(T_h v - v)\,\d x\right\rvert\\
\leq& \left\lvert\int_{\Omega\cap B_{2\rho_0}} \p_n\left(\tilde g \phi \right)\cdot(v_h - v)\,\d x\right\rvert + \left\lvert\int_{\Omega\cap B_{3\rho_0}} \left((\tilde g \phi)_{-h} - (\tilde g \phi) \right)\cdot \p_n v\,\d x\right\rvert \\
\lesssim& \|(\phi g)_h-\phi g\|_{\LL^{q^\prime}(\Omega)}\|\D v\|_{\LL^q(\Omega)}+\lvert h\rvert\|\tilde g\|_{\WW^{1,q^\prime}(\Omega)}\|\D v\|_{\LL^{q}(\Omega)}\\
\lesssim& \lvert h \rvert \|g\|_{\WW^{1-1/q^\prime,q^\prime}(\Omega)}\|v\|_{\WW^{1,q}(\Omega)}.
\end{align*}
Note that we may apply the previous two estimates to $v-(v)_{\Omega}$ instead, in order to replace the $W^{1,q}$ norm by $\norm{\D v}_{L^q(\Omega)}$ using the Poincar\'e inequality.

Hence, viewing $g \mapsto \lvert A_4\rvert$ as the mapping from $\BB_{q^\prime}^{\alpha-\frac1{q^\prime},q^\prime}(\p\Omega)) \to \R,$ since the estimate holds for $\alpha = 0,1$ we can interpolate in $[\cdot,\cdot]_{\alpha,q^\prime}$ to deduce \eqref{eq:A4Interpolated}.
This concludes the proof.
\end{proof}

Improving the argument by optimising the cut-off function as demonstrated in \cite[Section 4.1]{Koch2021a}, we obtain the following improved version in the autonomous case.
\begin{corollary}\label{cor:autonomous}
  Let $2 \leq p\leq q< \max\left\{\frac{np}{n-1},p+1\right\}$. Assume  $\Omega$ is a Lipschitz domain, $f\in \BB_{\infty}^{\alpha-1,q^\prime}(\Omega)$ and $g_N\in \WW^{\alpha-\frac1{q^\prime},q^\prime}(\p\Omega)$.
  Suppose $F=F(z)$ is autonomous and satisfies \eqref{def:bounds1}, \eqref{def:bounds2}. Then if $u$ is a relaxed minimiser of $\F_N$ in the class $\WW^{1,p}(\Omega)$, we have $u\in \WW^{1,q}(\Omega)$ and for any $\beta<1$ the estimate
\begin{align*}
  \|u\|_{\WW^{1,\frac{np}{n-\beta}}(\Omega)} \lesssim \left(1+\overline\F_N(u)+\|f\|_{\BB^{\alpha-1,q^\prime}_{\infty}(\Omega)}+\|g_N\|_{\WW^{\alpha-\frac1{q^\prime},q^\prime}(\Omega)}\right)^\gamma
\end{align*}
holds for some $\gamma>0$. In fact, we have
$$
\|V_p(\D u)\|_{\BB^{\frac12,2}_\infty(\Omega)}\lesssim \left(1+\overline\F_N(u)+\|f\|_{\BB^{\alpha-1,q^\prime}_{\infty}(\Omega)}^{q^\prime}+\|g_N\|_{\WW^{\alpha-\frac1{q^\prime},q^\prime}(\Omega)}^{q^\prime}\right)^\gamma.
$$
\end{corollary}

\begin{remark}[Excluding Lavrentiev]\label{rem:lavrentievNeumann}
  We point out that under the assumptions of Theorem \ref{thm:regularityRelaxed} (or the corresponding analog in the Dirichlet case, see \cite{Koch2020}), we can rule out the Lavrentiev phenomenon provided we assume the following structural condition: we assume there is $\e_0> 0$ such that for all $0<\e<\e_0$ and $x \in \Omega,$ there is $\hat y \in \overline{\Omega_{\eps}(x)}$ such that
  \begin{equation}
    F(\hat y, z) \leq F(y,z)
  \end{equation} 
  against all $y \in \overline{\Omega_{\eps}(x)}$ and $z \in \bb R^{m\times n}.$ This condition was introduced in \cite{Esposito2019} in the context of functionals with $(p,q)$-growth and is similar to Assumption 2.3 in \cite{Zhikov1995}. It has also been used in the context of lower semi-continuity in \cite{Acerbi1994a}. 
  If this holds, we have that ${\overline{\F}_N(u) = \F_N(u)}$ for all $u \in \WW^{1,p}(\Omega)$ such that $F(x,\D u) \in \LL^1(\Omega)$ by a straightforward adaption of the argument for the Dirichlet case in \cite{Koch2022a}.
\end{remark}

\subsection{Improved interior regularity}\label{sec:improved}
We now present an improved differentiability result in the interior case, which goes beyond the $\frac{\alpha}2$ differentiability obtained in \cite{Esposito2004} (and correspondingly Theorem \ref{thm:regularityRelaxed} and \cite{Koch2020} in the global case). We remark again that the techniques utilised to obtain the corresponding result in the autonomous setting \cite{Carozza2013} are not directly applicable. Moreover, an approach by regularisation does not seem sufficient to obtain in particular Theorem \ref{thm:radial_neumann}. Accordingly, the iteration technique presented in this section is new in the literature.

\interiorImproved*

\begin{remark}\label{rem:optimality_forcing}
In light of the results in \cite{Simon1977,Weimar2021} our results are sharp, however in the first part the limiting factor is the regularity of the forcing term $f.$
Assuming regularity beyond $f \in \LL^{p^{\prime}}$ gives further differentiability when $p>2,$ up until $\beta = p-1$ where $\frac{p^\prime(\beta-2)} 2=1$.
Moreover, we remark that in general the regularity in this case is also sharp for $\alpha = 1$, as shown in \cite[Section 5]{Brasco2018}.
\end{remark}

Before embarking on the proof, which consists of five steps, we will briefly comment on the overall strategy.
We will employ the difference quotient technique at the level of the $V$-functional, and obtain a general estimate for $\norm{\Delta_hV_{p,\mu}(\D u)}_{L^2}$ in the Step 1, which forms the basis for the subsequent analysis.
Then in Step 2, we show that when $p \geq 2$ this leads to an iterative improvement in differentiability, which we iterate over a suitable set of parameters.
Moreover we show that we can pass to the limit in this iteration, which is the content of Step 3.
Finally Steps 4 and 5 treat the cases $p \leq 2$ and more regular $f$ respectively, where the same techniques are applied to a slightly different set of estimates and choices of parameters.

\begin{proof}
  The restriction of exponents allows us to apply the higher integrability result from \cite{Esposito2004}, and so we may assume that $u \in \WW^{1,q}_{\loc}(\Omega)$ with a corresponding estimate.
Fix $x_0\in \Omega$. Write $B_t\equiv B_t(x_0)$ for $t>0$. Further denote $a(x,z) = \p_z F(x,z)$.
Suppose $R>0$ such that $B_R \Subset \Omega$, and let $\frac{R}2 < r < s < R$ with $s-r<1$.
Set $r_1 = r + \frac{s-r}3,$ $r_2 = r+\frac{2(s-r)}3.$
Let $\phi$ be a radial cut-off with $\phi =1 $ in $B_r$, supported on $B_{r_1}$, and such that $\lvert \nabla^k \phi\rvert \lesssim_k (s-r)^{-k}$ for each $k \in \mathbb N$. Further set
\begin{equation}\label{eq:second_diff}
\Delta_h^2 u = \Delta_h (-\Delta_{-h}) u = (-\Delta_{-h}) \Delta_h =  u(x+h)+u(x-h)- 2u(x).
\end{equation}

\textbf{Step 1: Basic estimate:} 
Consider $h\in \R^n$ with $\lvert h\rvert \leq (s-r)/3$. Testing the Euler-Lagrange equation for $u$ with the test function $-\phi^2 \Delta_h^2 u \in \WW^{1,q}_0(B_R)$ we find
\begin{align*}
A=&-\int_{B_{s}} \phi^2 a(x,\D u)\cdot \Delta_h^2 \D u \d x\\
=& \int_{B_{s}} a(x,\D u)\cdot \D \phi^2\otimes \Delta_h^2 u\d x-\int_{B_{s}} \phi^2 f\cdot \Delta_h^2 u \d x\\
 =& B_1 + B_2.
\end{align*}
Concerning $A$, we note using discrete integration by parts
\begin{align*}
A=&\int_{B_{s}} \Delta_h(\phi^2 a(x,\D u))\cdot\Delta_h \D u\d x\\
 =& \int_{B_{s}} \phi^2 \Delta_h a(x,\D u)\cdot \Delta_h \D u\d x+\int_{B_{s}} \Delta_h \phi^2 a(x,\D u)_h \Delta_h \D u\d x \\
 =& \int_{B_{s}} \phi^2 \left( a(x,\D u_h) - a(x,\D u) \right) \cdot \Delta_h \D u\d x+\\
  &+ \int_{B_{s}} \phi^2 \left( a(x+h,\D u_h) - a(x,\D u_h) \right) \cdot \Delta_h \D u\d x\\
  &+ \int_{B_{s}} \Delta_h \phi^2 a(x,\D u)_h \Delta_h \D u\d x  \\
  &= A_{1}+A_{2} + A_{3}.
\end{align*}
Using \eqref{def:bounds1} via \eqref{eq:hbound1_quantiative} and \eqref{def:bounds31} we have
\begin{align*}
  A_1\gtrsim \int_{B_{s}} \phi^2 \lvert V_{p,\mu}(\D u_h) - V_{p,\mu}(\D u)\rvert^2 \,\d x,
\end{align*} 
whereas for $A_2, A_3$ by \eqref{eq:h2bound_derivative} and \eqref{def:bounds31} we have
\begin{align*}
  \lvert A_2\rvert + \lvert A_3 \rvert \lesssim \left(\lvert h\rvert^{\alpha} + \lvert h\rvert\right) \frac1{s-r}\int_{B_{r_2}}(1+\lvert \D u_h\rvert)^{q-1}\lvert \Delta_h \D u\rvert \d x.
\end{align*}
Similarly, we estimate $B_1$ as
\begin{align*}
  B_1 \lesssim \frac1{s-r}\int_{B_{r_1}} (1+\lvert \D u\rvert)^{q-1}\lvert \Delta_h^2 u\rvert \d x.
\end{align*}
To estimate $B_2,$ as in the proof of Theorem \ref{thm:regularityRelaxed} we consider the endpoint estimates in $\alpha.$
The case $\alpha = 0$ follows by the duality estimate
\begin{equation}
  \begin{split}
    \lvert B_2\rvert &\leq \norm{f}_{\WW^{-1,p^\prime}(B_s)} \|\phi^2\Delta_h^2 u\|_{\WW^{1,p}(B_{r_1})} \\
                     &\lesssim \norm{f}_{\WW^{-1,p^\prime}(B_s)} \left(\norm{\Delta_h^2\D u}_{\LL^p(B_{r_1})} + \frac1{s-r}\norm{\Delta_h^2u}_{\LL^p(B_{r_1})}\right).
    \end{split}
\end{equation} 
Moreover using the fundamental theorem of calculus we can estimate the second term as
\begin{equation}\label{eq:seconddiff_ftc}
\norm{\Delta_h^2u}_{\LL^p(B_{r_2})} \lesssim \lvert h\rvert\left( \int_{B_{r_1}}\int_0^1 \lvert\Delta_{2th}\D u_{-th}\rvert^p \,\d t \,\d x \right)^{\frac1p}.
\end{equation} 
We will also use this in the case $\alpha=1$ which gives
\begin{equation}
  \lvert B_2\rvert \leq \norm{f}_{\LL^{p^\prime}(B_s)} \|\Delta_h^2u\|_{\LL^p(B_{r_1})} \lesssim \lvert h\rvert\norm{f}_{\LL^{p^\prime}(B_s)} \sup_{0\leq t \leq 1}\norm{\Delta_{2th} \D u_{-th}}_{\LL^p(B_{r_1})}.
\end{equation} 
Now interpolating $f \mapsto \lvert B_2\rvert$ using $[\WW^{-1,p^\prime}(B_s),\LL^{p^\prime}(B_s)]_{\alpha,\infty} = \BB^{\alpha-1,p^\prime}_{\infty}(B_s)$ for $\alpha \in (0,1),$ we deduce that
\begin{equation}\label{eq:alpha_f_estimate}
  \begin{split}
    \lvert B_2\rvert&\leq \int_{B_s} \lvert f\cdot(\phi^2\Delta^2_h u)\rvert\\
    &\lesssim  \frac{\lvert h\rvert^{\alpha}}{s-r}\|f\|_{\BB_{\infty}^{\alpha-1,p^\prime}(B_s)}\left(\|\Delta_h^2 \D u\|_{\LL^{p}(B_{r_2})}+\sup_{0\leq t \leq 1}\norm{\Delta_{2th} \D u_{-th}}_{\LL^p(B_{r_1})}\right)
  \end{split}
\end{equation}
for all $\alpha \in (0,1],$ noting that $\BB_{\infty}^{0,p^\prime}(B_s) = \LL^{p^\prime}(B_s).$

Collecting the above estimates, we have shown that
\begin{equation}\label{eq:tangential_mainestimate}
  \begin{split}
    \|\Delta_h V_{p,\mu}(\D u)\|_{\LL^2(B_{r})}^2 
    &\lesssim
     \frac{\lvert h\rvert^{\alpha}}{s-r} \|f\|_{\BB^{1-\alpha,p^\prime}_{\infty}(B_{s})}\|\Delta_h \D u\|_{\LL^{p}(B_{r_1})}\\
    &\quad+ \frac{\lvert h\rvert^{\alpha}}{s-r} \|f\|_{\BB^{1-\alpha,p^\prime}_{\infty}(B_{s})}\sup_{0\leq t \leq 1}\norm{\Delta_{2th} \D u_{-th}}_{\LL^p(B_{r_1})}\\
            &\quad+\frac{\lvert h\rvert^{\alpha}}{s-r} \int_{B_{r_2}}(1+\lvert\D u_h\rvert)^{q-1}\lvert\Delta_h \D u\rvert\d x\\
            &\quad+\frac1{s-r}\int_{B_{{r_1}}} (1+\lvert \D u\rvert)^{q-1}\lvert \Delta_h^2 u\rvert \d x\\
            &= C_{1,1} + C_{1,2} + C_2 + C_3.
  \end{split}
\end{equation}

\textbf{Step 2: Iteration, case $p \geq 2$.}

We will repeatedly use \eqref{eq:tangential_mainestimate} to iteratively infer an improvement in differentiability, which will involve carefully estimating $C_2$ and $C_3$ at each step. 
Assume that
\begin{align}\label{eq:inducBase2}
  \|V_{p,\mu}(\D u)\|_{\BB^{\delta,2}_\infty(B_s)}\lesssim A^{\gamma}
\end{align}
for some $\delta>0$, $\gamma\geq 0$, where
\begin{equation}\label{eq:shorthand_a2}
  A = \left(1+\overline\F(u)+\|f\|_{\BB_{\infty}^{\alpha-1,p^\prime}(\Omega)}^{p^\prime}\right).
\end{equation} 
We know by (a local version of) Theorem \ref{thm:regularityRelaxed} that we can take $\delta = \frac \alpha 2.$
By Lemma \ref{lem:Vfunc_diff}, we can estimate
\begin{equation}
  \begin{split}
    C_{1,1} + C_{1,2} &\leq \frac{\lvert h\rvert^{\alpha}}{s-r} \norm{f}_{\BB_{\infty}^{\alpha-1,p^\prime}(B_s)}\left(\lvert h\rvert^{\delta}\norm{V_{p,\mu}(\D u)}_{\BB_{\infty}^{\delta,2}(B_{s})}\right)^{\frac2{p}}\\
                      &\leq \lvert h\rvert^{\alpha + \frac{2\delta}{p}} \frac{A}{s-r} \norm{V_{p,\mu}(\D u)}_{\BB_{\infty}^{\delta,2}(B_{s})}^{\frac2{p}}\\
                      &\lesssim \lvert h\rvert^{\alpha+\frac{2\delta}p} \frac{A^{\frac{2\gamma}p + 1}}{s-r}.
  \end{split} 
\end{equation} 
For the remaining two terms, suppose further that there is $\tau_1,\tau_2 \geq 1$ and $\eps>0$, $\sigma>0$ such that $\BB^{\delta,2}_{\infty}(B_s)$ embeds into $\LL^{\tau_1-\eps}(B_s) \cap \BB_{\infty}^{\sigma,\tau_2}(B_s),$ where
\begin{equation}\label{eq:tau_conjugate2}
  \frac 2{p\tau_2} + \frac{2(q-1)}{p(\tau_1-\eps)} \leq 1.
\end{equation} 
Then by Lemma \ref{lem:Vfunc_diff}, since $V_{p,\mu}(\D u) \in B^{\delta,2}_{\infty}(B_s)$ this implies that 
\begin{equation}
  \D u \in \LL^{\frac{p}2(\tau_1-\eps)}(B_r) \cap \BB^{\frac{2\sigma}{p},\frac{p\tau_2}2}(B_r),
\end{equation} 
and so by H\"older we can estimate
\begin{align}\label{eq:B42}
  C_2 \lesssim& \frac{\lvert h\rvert^{\alpha}}{s-r}\|\Delta_h \D u\|_{\LL^{\frac{p\tau_2}2}(B_{r_2})}\left(1+\|\D u\|_{\LL^{\frac p2(\tau_1-\eps)}(B_s)}\right)^{q-1}\\
  \lesssim&\frac{\lvert h\rvert^{\alpha}}{s-r}\|\Delta_h V_{p,\mu}(\D u)\|_{\LL^{\tau_2}(B_{r_2})}^{\frac 2p} \left(1+\|V_{p,\mu}(\D u)\|_{\LL^{\tau_1-\eps}(B_s)}^{\frac2p}\right)^{q-1}\\
  \lesssim& \lvert h\rvert^{\alpha+\frac{2\sigma}{p}}\norm{V_{p,\mu}(\D u)}_{\BB^{\delta,2}_{\infty}(B_s)}^{\frac2p}\frac{A^{\frac{2(q-1)\gamma}p+1}}{s-r}.
\end{align} 
We can similarly bound $C_3,$ except we use \eqref{eq:seconddiff_ftc} to bound
\begin{align*}
  \|\Delta_h^2 u\|_{\LL^{\frac{p\tau_2}2}(B_{r_1})}\lesssim& \lvert h\rvert \sup_{0<t<1}\|\Delta_{2th}\D  u_{-th}\|_{\LL^{\frac{p\tau_2}2}(B_{r_2})}\lesssim \lvert h\rvert^{1+\frac{2\sigma}p} \|V_{p,\mu}(\D u)\|_{\BB^{\delta,2}(B_{s})}^{\frac 2p},
\end{align*}
so we arrive at
\begin{equation}
C_2 + C_3 \lesssim \lvert h\rvert^{\alpha + \frac{2\sigma}p} \frac{A^{\frac{2q}p\gamma+1}}{s-r}.
\end{equation} 

In particular, we can estimate \eqref{eq:tangential_mainestimate} as
\begin{equation}\label{eq:improved_h_decay2}
  \begin{split}
  \norm{\Delta_hV_{p,\mu}(\D u)}_{\LL^2(B_r)} 
  &\lesssim  \frac{\lvert h\rvert^{\frac{\alpha}2 + \frac{\delta}p}}{s-r} \norm{f}_{\BB^{1-\alpha,p'}_{\infty}(B_s)} \norm{V_{p,\mu}(\D u)}_{\BB^{\delta,2}_{\infty}}^{\frac1p} \\
  &\quad + \frac{\lvert h\rvert^{\frac{\alpha}2  + \frac{\sigma}p}}{s-r}\norm{V_{p,\mu}(\D u)}_{\BB_{\infty}^{\delta,2}(B_s)}^{\frac1p} (1 + \norm{V_{p,\mu}(\D u)}_{\LL^{\tau_1-\eps}(B_s)})^{\frac{q-1}p} \\
  &\lesssim \left( \lvert h\rvert^{\frac{\alpha}2 + \frac{\delta}p} + \lvert h\rvert^{\frac{\alpha}2 + \frac{\sigma}p} \right) \frac{A^{\frac{q \gamma}p + \frac12}}{s-r},
  \end{split}
\end{equation} 
where we have kept the precise dependence of constants which we will need in the next step.
Hence this shows that $V_{p,\mu}(\D u) \in \BB^{\delta^\prime,2}_{\infty}(B_r)$ with
\begin{equation}
  \delta^\prime = \frac{\alpha}2 +  \frac1p \min\left\{ \delta,\sigma \right\},
\end{equation} 
with the associated estimate
\begin{equation}
  \|V_{p,\mu}(\D u)\|_{\BB^{\delta^\prime,2}_\infty(B_r)}\lesssim \frac{A^{\frac{q}p\gamma + \frac12}}{s-r}.
\end{equation} 
This idea will be to iterate this procedure for a carefully chosen set of parameters $\delta, \sigma, \tau_1, \tau_2, \eps.$
To do this, consider a sequence $\rho_k \in (R/2,R)$ such that $\rho_{k+1} < \rho_k$ for all $k$ to be determined.
We will apply the claim with $\rho_{k+1}, \rho_k$ in place of $r,s$ for each $k.$
We set $\delta_0 = \frac{\alpha} 2$ and show by induction that $V_{p,\mu}(\D u) \in \BB^{\delta_k,2}_{\infty}(B_{\rho_{tk}})$, for some $t=t(n,p,q,\alpha) \geq 1$ where
\begin{equation}
  \delta_k = \frac{\alpha}{2}\sum_{i=0}^k \frac1{ p^i},
\end{equation} 
which tends to $\frac{\alpha p}{2(p - 1)}$ as $k \to \infty.$
By Theorem \ref{thm:embedding}, we can take
\begin{equation}\label{eq:tau12}
  \frac1{\tau_1} \geq \frac 1 2 -\frac{\delta_k}n.
\end{equation}
Choosing $\tau_1$ so equality holds in \eqref{eq:tau12}, choosing $\tau_2$ via the equality case of \eqref{eq:tau12} we set
\begin{equation}
  \tau_2 = \tau_{2,0} := \left(\frac p2 - \frac{\frac{2n}{n-2\delta_k}-\eps}{q-1}\right)^{-1} \geq 1.
\end{equation} 
Note that since $q < \frac{np}{n-\alpha}$ and $\delta_k \geq \frac {\alpha} 2,$ by shrinking $\eps>0$ we can assume that $\frac{p}2(\tau_1-\eps) > q$, so we have $\tau_{2,0} < \frac{2q}p$.
We will consider two cases depending on possible values of $\tau_{2,0}$; if we have $\tau_{2,0} \geq 2$ for suitably small $\eps>0$, then we make take $\sigma = \delta$, whereas if $\tau_{2,0}<2$ we will require a two-step iteration process.
This distinction gives rises to two cases which we will consider separately.

\textbf{Case (a)}: Suppose that for some $\delta_k$ we have $\tau_{2,0} < 2$ for all $\eps>0$ sufficiently small, which is equivalent to the condition
\begin{equation}\label{eq:deltak_casea_condition}
\delta_k < \frac n {\alpha(q-1)}\left(\frac{1-p} 2+\frac{q-1} p\right).
\end{equation} 
Given $\eps>0$ to be determined, defining $\tau_1, \tau_2$ to satisfy \eqref{eq:tau_conjugate2}, \eqref{eq:tau12} with equality (so $\tau_2 = \tau_{2,0}$), we will choose $\sigma$ to satisfy
\begin{equation}
  \sigma = \delta_k + n\left(\frac1{\tau_2} - \frac1 2\right).
\end{equation} 
This ensures that $\BB^{\delta_k,2}_{\infty}(B_{\rho_{tk}}) \hookrightarrow \BB^{\sigma,\tau_2}_{\infty}(B_{\rho_{tk}})$ by Theorem \ref{thm:embedding}.
Note that
\begin{equation}
  \lim_{\eps \to 0} \sigma = q\left(\delta_k-\frac {\alpha} 2\right)+ \frac{np-nq+\alpha q}2 = q\left(\delta_k - \frac\alpha 2\right)+\kappa_0,
\end{equation} 
for some $\kappa_0>0$, noting that $\frac qp < \frac{n+\alpha}n < \frac{n}{n-\alpha}.$
Thus we can choose $\eps>0$ so that
\begin{equation}
  \sigma \geq q\left(\delta_k-\frac{\alpha}2\right)+\frac{\kappa_0} 2.
\end{equation} 
Then by \eqref{eq:improved_h_decay2} we deduce that $V_{p,\mu}(\D u) \in \BB^{\delta_{k,1},2}_{\infty}(B_{\rho_{tk+1}})$, with
\begin{equation}
  \delta_{k,1} = \min\left\{ \delta_{k+1}, \frac{\alpha}2+\frac q p\left(\delta_k-\frac {\alpha} 2\right)+\frac{\kappa_0}{2p} \right\}.
\end{equation} 
for some small $c_0>0$.
We can now replace $\delta_k = \delta_{k,0}$ by $\delta_{k,1}$ and apply this iteratively, which gives $V_{p,\mu}(\D u) \in \BB^{\delta_{k,j},2}_{\infty}(B_{\rho_{tk}+j})$ where
\begin{equation}
  \delta_{k,j} \geq \min\left\{ \delta_{k+1}, \frac{\alpha} 2 + \frac{\kappa_0}{2p} \sum_{i=0}^j \left( \frac qp \right)^i \right\}.
\end{equation} 
Since the second term is divergent, there is some $j_0= j_0(\alpha,n,p,q) \in \bb N$ for which either $\delta_{k,j_0} = \delta_{k+1}$ or $\delta_{k,j_0} \geq \frac n {\alpha(q-1)}\left(\frac{1-p} 2+\frac{q-1} p\right).$
In the former case we are done, otherwise we can turn to case (b).

\textbf{Case (b)}: Suppose for some $\delta^\prime \geq\delta_k,$ we have $V_{p,\mu}(\D u) \in \BB^{\delta^\prime,2}_{\infty}(B_{\rho_{tk+\ell}})$ for some $\ell\in[0,j_0]$, and $\delta^\prime \geq \frac n {\alpha(q-1)}\left(\frac{1-p} 2+\frac{q-1} p\right)$; that is $\tau_{2,0} \geq 2$ for small enough $\eps>0$, with $\delta^{\prime}$ in place of $\delta_k$. Then we will take $\tau_2 = 2,$ and for $\eps>0$ sufficiently small we can choose $\tau_1 \leq \frac{np}{n-2\delta^\prime }$ using \eqref{eq:tau_conjugate2}. Hence we obtain \eqref{eq:improved_h_decay2} with $\sigma = \delta_k,$ giving $V_{p,\mu}(\D u) \in \BB^{\delta_{k+1},2}_{\infty}(B_{\rho_{t(k+1)}})$ with $t\leq j_0$. This completes our induction argument.

\textbf{Step 3: Limiting case.} We now wish to pass to the limit as $k \to \infty$, to infer differentiability of order $\delta = \frac{\alpha p}{2(p-1)}$.
From the previous step we know there is some $t \geq 1$, $C_k \geq 0$, $\tilde\gamma_k>0$ such that
\begin{equation}\label{eq:mainestimate_step2}
  \norm{V_{p,\mu}(\D u)}_{\BB^{\delta_k,2}_{\infty}(B_{\rho_{tk}})} \lesssim_k A^{\widetilde\gamma_k},
\end{equation} 
where the implicit constant depends on the choice of radii $\{\rho_j\}_{j\geq1}$.
Moreover since we have
\begin{equation}\label{eq:caseb_inequality}
  \frac n {\alpha(q-1)}\left(\frac{1-p} 2+\frac{q-1} p\right)<\frac{\alpha p}{2(p-1)} = \lim_{k \to \infty} \delta_k,
\end{equation} 
we can choose $k_0$ sufficiently large such that the above holds for $\delta_k$ with $k \geq k_0$, which ensures we are in case (b) in the previous step.
Then for $k \geq k_0$, we can take $\tau_2=2$ and $\tau_1-\eps = \frac{2(q-1)}{p-1}$ in \eqref{eq:improved_h_decay2} to estimate
\begin{equation}\label{eq:step3_iterate}
  \begin{split}
    &\norm{V_{p,\mu}(\D u)}_{\BB^{\delta_{k+1},2}_{\infty}(B_{\rho_{tk+1})}} 
  \leq \frac{C_1}{\rho_{tk+1}-\rho_{tk}} \norm{V_{p,\mu}(\D u)}_{\BB^{\delta_k,2}_{\infty}}^{\frac1p} \\
  &\qquad + \frac{C_1}{\rho_{tk+1}-\rho_{tk}} \norm{V_{p,\mu}(\D u)}_{\BB_{\infty}^{\delta_k,2}(B_{\rho_{tk}})}^{\frac1p} \left(1 + \norm{V_{p,\mu}(\D u)}_{\LL^{\frac{2(q-1)}{p-1}}(B_{\rho_{tk}})}\right)^{\frac{q-1}p}
  \end{split}
\end{equation} 
for all $k \geq k_0$, where $C_1$ is independent of $k$.
To pass to the limit in the iteration, we seek an estimate of the form
\begin{equation}\label{eq:step3_interpolation}
  \norm{V_{p,\mu}(\D u)}_{\BB^{\delta_k,2}_{\infty}(B_{\rho_{tk}})}^{\frac1p} \norm{V_{p,\mu}(\D u)}_{\LL^{\frac{2(q-1)}{p-1}}(B_{\rho_{tk}})}^{\frac{q-1}p} \leq C_2 A^{\widetilde\gamma_0}\norm{V_{p,\mu}(\D u)}_{\BB^{\delta_k,2}_{\infty}(B_{\rho_{tk}})}^{\kappa},
\end{equation} 
with $\kappa < 1$, which we claim holds for sufficiently large $k$.

Indeed by Theorem \ref{thm:embedding} we know that for $k \geq 1$ we have
\begin{equation}
  \BB^{\delta_k,2}_{\infty}(B_{\rho_{tk}}) \hookrightarrow \LL^{2\xi}(B_{\rho_{tk}}) \quad \text{ for all }\quad \xi < \frac{n}{n-2\delta_k},
\end{equation} 
so in particular this holds for $\xi_k := \frac{n}{n-2\delta_{k-1}}$.
We then seek $\theta_k \in (0,1)$ such that
\begin{equation}
  \begin{split}
    \norm{V_{p,\mu}(\D u)}_{\LL^{\frac{2(q-1)}{p-1}}(B_{\rho_{tk}})} 
    &\lesssim \left(1 + \norm{V_p(\D u)}_{\LL^{\frac{2q}p}(B_R)}\right)^{1-\theta_k} \norm{V_{p,\mu}(\D u)}_{\LL^{\xi_k}(B_{\rho_{tk}})}^{\theta_k} \\
    &\lesssim A^{\tilde\gamma_0} \norm{V_{p,\mu}(\D u)}_{\BB^{\delta_k,2}_{\infty}(B_{\rho_{tk}})}^{\theta_k},
  \end{split}
\end{equation} 
where we use the above embedding in the second line.
This holds provided we can take
\begin{equation}
  \frac{1-\theta_k}{2q/p} + \frac{\theta_k}{2\xi_k} = \frac{p-1}{2(q-1)},
\end{equation} 
that is
\begin{equation}\label{eq:theta_rearrange}
  \theta_k = \frac{q-p}{q(q-1)} \left( \frac pq - \frac1{\xi_k} \right)^{-1}.
\end{equation} 
Note that $\delta_{k-1} \geq \delta_0 = \frac{\alpha}2$, and since $q < \frac{np}{n-\alpha}$ it follows that $2\xi_k > \frac{2q}p$ and hence that $\theta_k > 0$.
Now we claim that for $k \geq 1$ sufficiently large we have
\begin{equation}
  \kappa_k := \frac1p + \frac{q-1}p \theta_k < 1,
\end{equation} 
which would also imply that $\theta_k < \frac{p-1}{q-1} \leq 1$.
Indeed using \eqref{eq:theta_rearrange} and noting that $\lim_{k \to \infty} \xi_k = \frac{n}{n-\alpha p^\prime}$ we can compute
\begin{equation}
  \kappa_{\infty} := \lim_{k \to \infty} \kappa_k = \frac{\alpha p^{\prime}q}{p(np - nq + \alpha p^{\prime} q)}.
\end{equation} 
By rearranging we see that
\begin{equation}
  q < \frac{np}{n-\alpha} \implies \kappa_{\infty} < 1.
\end{equation}
Hence there exists $k_1 \geq k_0$ such that $\kappa := \kappa_{k_1} < 1$ verifies \eqref{eq:step3_interpolation}.

We now set $\tilde\delta_j = \delta_{k_1+j}$ and $\tilde\rho_j = \rho_{tk_1+j}$, then combining \eqref{eq:mainestimate_step2} with $k=k_1$, \eqref{eq:step3_iterate}, and \eqref{eq:step3_interpolation} we deduce that
\begin{equation}\label{eq:step3_iteration_estimate}
  \norm{V_{p,\mu}(\D u)}_{\BB^{\tilde\delta_{j+1},2}_{\infty}(B_{\tilde\rho_{j+1}})} \leq \frac{\tilde C}{\tilde\rho_{j+1}-\tilde\rho_j} \left( \norm{V_{p,\mu}(\D u)}_{\BB^{\tilde\delta_j,2}_{\infty}(B_{\tilde\rho_j})}^{\frac1p} + \norm{V_{p,\mu}(\D u)}_{\BB^{\tilde\delta_j,2}_{\infty}(B_{\tilde\rho_j})}^{\kappa}\right) A^{\tilde\gamma}
\end{equation} 
for all $j\geq 1$, where $\tilde C>0$ and $\tilde\gamma>0$ are constants.
By Young's inequality we can estimate this as
\begin{equation}
  \norm{V_{p,\mu}(\D u)}_{\BB^{\tilde\delta_{j+1},2}_{\infty}(B_{\tilde\rho_{j+1}})} \leq \frac12\norm{V_{p,\mu}(\D u)}_{\BB^{\tilde\delta_j,2}_{\infty}(B_{\tilde\rho_j})}+ \left(\frac{2\tilde CA^{\tilde\gamma}}{\tilde\rho_{j+1}-\tilde\rho_j}\right)^{p' + \frac1{1-\kappa}},
\end{equation} 
which we can iterate over suitable $\widetilde\rho_{j},$ which are determined by our choice of $\rho_k$.
Following the proof of \cite[Lemma 6.1]{Giusti2003}, we set $\rho_0 = R$ and set
\begin{equation}
  \rho_{k+1} = \rho_k - (1-\lambda) \lambda^k \frac{R}2,
\end{equation} 
where $\lambda \in (0,1)$ is chosen to satisfy $\frac12 \lambda^{p'+\frac1{1-\kappa}} < 1$.
Then noting that $\tilde\rho_{j} - \tilde\rho_{j+1} = (1-\lambda)\lambda^{tk_1+j} \frac{R}2$ we see that
\begin{equation}
  \begin{split}
  \norm{V_{p,\mu}(\D u)}_{\BB^{\tilde\delta_j,2}_{\infty}(B_{R/2})} 
  &\leq \frac1{2^{j}} \norm{V_{p,\mu}(\D u)}_{\BB^{\tilde\delta_0,2}_{\infty}(B_R)} \\
  &\quad+ \left(\frac{2A^{\tilde\gamma}}{(1-\lambda)\lambda^{tk_1} R}\right)^{p' + \frac1{1-\kappa}} \sum_{i=0}^j 2^{-j}\lambda^{-j\left( p^{\prime} + \frac1{1-\kappa}\right)}, 
  \end{split}
\end{equation} 
so passing to the limit we have
\begin{equation}
  \limsup_{j \to \infty} \norm{V_{p,\mu}(\D u)}_{\BB^{\tilde\delta_j,2}_{\infty}(B_{R/2})} \lesssim \left(R^{-1}A^{\tilde\gamma}\right)^{p' + \frac1{1-\kappa}} < \infty.
\end{equation} 
We can now pass to the limit and deduce differentiability in $\BB^{\frac{\alpha p}{2(p-1)},2}_{\infty}$ noting that the above gives the uniform bound
\begin{equation}
  \norm{\Delta_hV_{p,\mu}(\D u)}_{\LL^2(B_{R/2-\lvert h\rvert})} \leq C \lvert h\rvert^{\tilde\delta_j}
\end{equation} 
for all $\lvert h\rvert \leq \frac R2$. Sending $j \to \infty$ establishes the claimed differentiability.

\textbf{Step 4: Iteration, case $p < 2$.} We will employ a similar argument, and we will only highlight the differences.
As before assume there is $\delta>0,$ $k\geq 0$ such that ${V_{p,\mu}(\D u) \in \BB^{1+\delta,2}_{\infty}(B_s)}$ with the estimate \eqref{eq:tangential_mainestimate}.
We will estimate \eqref{eq:tangential_mainestimate} using this, where similarly as in the $p\geq 2$ we use Lemma \ref{lem:Vfunc_diff} to bound
\begin{equation}
  C_{1,1} + C_{1,2} \lesssim \lvert h\rvert^{\alpha+\delta} \frac{A^{\frac1{p^\prime}}}{s-r} \left( 1 + \norm{\D u}_{L^p(B_{r_2})} \right)^{1-\frac p2} \seminorm{V_{p,\mu}(\D u)}_{\BB^{\delta,2}_{\infty}(B_{r_2})}.
\end{equation} 
For $C_2$ we first estimate
\begin{equation}
  \begin{split}
  C_2 &\leq \frac{\lvert h\rvert^{\alpha}}{s-r} \int_{B_{r_2}} (1 + \lvert \D u\rvert + \lvert \D u_h\rvert)^{q - 1}\lvert \Delta_h\D u\rvert \,\d x \\
      &\lesssim \frac{\lvert h\rvert^{\alpha}}{s-r}\int_{B_{r_2}} (1 +\lvert \D u\rvert + \lvert \D u_h\rvert)^{q - \frac p2} \lvert \Delta_hV_{p,\mu}(\D u)\rvert \,\d x,
  \end{split}
\end{equation} 
where we have used \eqref{eq:vfunction_difference}.
We seek to estimate $C_3$ similarly, but we will use the fundamental theorem of calculus to write
\begin{equation}
  \begin{split}
    C_3 &\leq \frac1{s-r}\int_{B_{r_1}} (1 + \lvert \D u\rvert)^{q-1} \left\lvert \int_0^1 \Delta_h \D u_{-th} \cdot h \,\d t\right\rvert \,\d x \\
        &\leq \frac{\lvert h\rvert}{s-r} \int_0^1 \int_{B_{r_2}} (1 + \lvert \D u_{(1-t)h}\rvert + \lvert \D u_{-th}\rvert + \lvert \D u\rvert)^{q-1} \lvert \Delta_h \D u_{-th}\rvert \,\d x \,\d t \\
        &\leq \frac{\lvert h\rvert}{s-r} \int_0^1 \int_{B_{r_2}} (1 + \lvert \D u_{(1-t)h}\rvert + \lvert \D u_{-th}\rvert + \lvert \D u\rvert)^{q-\frac p2} \lvert \Delta_h V_{p,\mu}(\D u)_{-th}\rvert \,\d x \,\d t,
  \end{split}
\end{equation} 
where we have used the fact that $p-2 \leq 0$ in the last line.

Now suppose there is $\tau_1, \tau_2 \geq 1$ and $\eps, \sigma > 0$ such that we have the embedding $\BB^{\delta,2}_{\infty}(B_s) \hookrightarrow \LL^{\tau_1-\eps}(B_s) \cap B^{\sigma,\tau_2}_{\infty}(B_s)$ as before, but we now require
\begin{equation}\label{eq:subquad_tau}
  \frac1{\tau_2} + \frac{\frac {2q}p - 1}{\tau_1 - \eps} \leq 1.
\end{equation} 
For the embeddings to hold, by Theorem \ref{thm:embedding} we require
\begin{equation}\label{eq:subquad_embedding}
  \frac1{\tau_1} \geq \frac12 - \frac{\delta}2, \quad \frac1{\tau_2} \geq \frac12 - \frac{\delta-\sigma}n.
\end{equation} 
Now by H\"older we can estimate
\begin{equation}
  C_2 + C_3 \leq \frac{\lvert h\rvert^{\alpha + \sigma}}{s-r} \left( 1 + \norm{\D u}_{\LL^{\frac{p(\tau_1-\eps)}2}(B_s)} \right)^{q- \frac p2} \seminorm{V_{p,\mu}(\D u)}_{\BB^{\sigma,\tau_2}_{\infty}(B_s)},
\end{equation} 
and so collecting estimate we have
\begin{equation}\label{eq:step4_preiterate}
  \begin{split}
  \norm{\Delta_hV_{p,\mu}(\D u)}_{\LL^2(B_r)} 
  &\lesssim \frac{\lvert h\rvert^{\alpha+\delta}A^2}{s-r} \norm{V_{p,\mu}(\D u)}_{\BB^{\delta,2}_{\infty}(B_s)} \\
  &\quad + \frac{\lvert h\rvert^{\alpha+\sigma}}{s-r} \left( 1 + \norm{V_{p,\mu}(\D u)}_{\LL^{\tau_1-\eps}(B_s)} \right)^{\frac{2q-p}{2p}} \norm{V_{p,\mu}(\D u)}_{\BB^{\delta,2}_{\infty}(B_s)}.
  \end{split}
\end{equation} 
Therefore we infer that $V_{p,\mu}(\D u) \in \BB^{\delta^\prime,2}_{\infty}(B_r)$ with
\begin{equation}
  \delta^\prime = \frac{\alpha}2 + \frac12\min\left\{ \delta,\sigma \right\},
\end{equation}
and there is $\widetilde\gamma > 0$ such that
\begin{equation}
  \norm{u}_{B^{\delta^\prime,2}_{\infty}(B_{r})} \lesssim \frac{A^{\widetilde \gamma}}{s-r}.
\end{equation} 
To iterate this, we set $\delta_0 = \frac{\alpha}2$ and show by induction that $V_{p,\mu}(\D u) \in \BB^{\delta_k,2}_{\infty}(B_{\rho_{tk}})$ for all $k$ where
\begin{equation}
  \delta_k = \alpha \sum_{i=1}^k \frac1{2^i},
\end{equation} 
which tends to $\alpha$ as $k \to \infty.$

Choose $\tau_1$ so that the first inequality in $\eqref{eq:subquad_embedding}$ holds with equality, then since $\delta \geq \frac{\alpha}2$ and $q < \frac{np}{n-\alpha}$ we have $\tau_1 \geq \frac{2n}{n-\alpha}.$ Therefore for $\eps>0$ sufficiently small we get
\begin{equation}
  \frac{\frac{2q}p-1}{\tau_1 - \eps} < 1-\frac{p}{2q} < 1,
\end{equation} 
so defining $\tau_{2,0} = \tau_2$ in the equality cases of \eqref{eq:subquad_tau}, $\eqref{eq:subquad_embedding}_1$, we get $1 < \tau_{2,0} < \frac{2q}p$.
As before, if $\tau_{2,0} \leq 2$ for sufficiently small $\eps>0$ then we can take $\tau_2 = 2$ and $\sigma = \delta$.
Otherwise we have
\begin{equation}\label{eq:deltak_casea_condition2}
  \delta_k < \frac{n(q-p)}{2q - p}
\end{equation} 
for all $\eps>0$ small.
Determining $\tau_2 = \tau_{2,0}$ and $\tau_1 = \frac{2n}{n-2\delta_k}$, by $\eqref{eq:subquad_embedding}_2$ we choose $\sigma = \delta_k + n \left( \frac1{\tau_2} - \frac12 \right).$ Observe that
\begin{equation}
  \lim_{\eps \to 0} \sigma = \frac{2q}p \left( \delta - \frac{\alpha}2 \right) + \left( n - (n-\alpha) \frac qp \right) =: \frac{2q}p \left( \delta - \frac{\alpha}2 \right)  + \kappa_0.
\end{equation} 
Since $q < \frac{np}{n-\alpha}$ we have $\kappa_0>0,$ and so we can choose $\eps>0$ sufficiently small, so that
\begin{equation}
  \sigma \geq \frac{2q}p \left( \delta - \frac{\alpha}2 \right) + \frac{\kappa_0}2 \geq \frac{\kappa_0}2.
\end{equation} 
Then analogously as in the case $p\geq 2,$ we can use this iteratively to show that we have ${V_{p,\mu}(\D u)\in \BB^{\delta_{k,j},2}(B_{\rho_{tk+j}})}$ with
\begin{equation}
  \delta_{k,j} \geq \min\left\{ \delta_{k+1}, \frac{\alpha}2 + \frac{\kappa_0}{4} \sum_{i=1}^j \left( \frac qp \right)^i \right\}.
\end{equation} 
We can now iterate this exactly as in the case $p\geq 2$.
To pass to the limit, since $\alpha > \frac{n(q-p)}{2q-p}$ there is $k_0$ such that for $k \geq k_0$ we can take $\tau_2 =2$ and $\sigma = \delta_k$ in  \eqref{eq:step4_preiterate}.
This gives the estimate
\begin{equation}
  \begin{split}
   &\norm{V_{p,\mu}(\D u)}_{\BB^{\delta_{k+1},2}(B_{\rho_{tk+1}})} 
  \leq \frac{C_1}{\rho_{tk+1}-\rho_{tk}} \norm{V_{p,\mu}(\D u)}_{\BB^{\delta_k,2}_{\infty}(B_{\rho_{tk}})}^{\frac12} \\
  &\qquad + \frac{C_1}{\rho_{tk+1}-\rho_{tk}} \left( 1 + \norm{V_{p,\mu}(\D u)}_{\LL^{\frac{2(2q-1)}p}(B_{\rho_{tk}})} \right)^{\frac{2q-p}{2p}} \norm{V_{p,\mu}(\D u)}_{\BB^{\delta_k,2}_{\infty}(B_{\rho_{tk}})}^{\frac12}.
  \end{split}
\end{equation} 
Now set $\xi_k = \frac{n}{n-2\delta_{k-1}}$ so that the embedding $\BB^{\delta_k,2}_{\infty} \hookrightarrow \LL^{2\xi_k}$ holds, and choose $\theta_k \in (0,1)$ so that
\begin{equation}
  \norm{V_p(\D u)}_{\LL^{\frac2p (2q-p)}(B_{\tilde\rho_j})} \leq \norm{V_p(\D u)}_{\LL^{\frac{2q}p}(B_{\tilde\rho_j})}^{1-\theta_k}\norm{V_p(\D u)}_{\LL^{2\xi}(B_{\tilde\rho_j})}^{\theta_k}.
\end{equation} 
This holds provided we take
\begin{equation}\label{eq:theta_subquadratic}
  \frac{1-\theta_k}{2q/p} + \frac{\theta_k}{2\xi_k} = \frac p{2(2q-p)}.
\end{equation} 
We then require
\begin{equation}
  \kappa_k = \frac12 + \frac{2q-p}{2p} \theta_k < 1,
\end{equation}
which we claim holds for sufficiently large $k$ since
\begin{equation}
  \lim_{k \to \infty} \kappa_k = \frac12 + \frac{q-p}{2q} \left( \frac pq - \frac{n-2\alpha}n \right)^{-1} < 1,
\end{equation} 
holds as $q < \frac{np}{n-\alpha}$.
Hence there is $k_1 \geq k_0$ such that $\kappa := \kappa_{k_1} < 1$.
For these parameters we have
\begin{equation}\label{eq:step4_iteration_estimate}
  \begin{split}
    &\norm{V_{p,\mu}(\D u)}_{\BB^{\delta_{k_1+j+1},2}_{\infty}(B_{\tilde\rho_{tk_1+j+1}})} \\
    &\leq \frac{\tilde C}{\rho_{tk_1+j+1}-\rho_{tk_1+j}} \left( \norm{V_{p,\mu}(\D u)}_{\BB^{\delta_{k_1+j},2}_{\infty}(B_{\rho_{tk_1+j}})}^{\frac12} + \norm{V_{p,\mu}(\D u)}_{\BB^{\delta_{k_1+j},2}_{\infty}(B_{\rho_{tk_1+j}})}^{\kappa}\right) A^{\tilde\gamma},
  \end{split}
\end{equation} 
for some $\tilde C>0$ and $\tilde \gamma>0$, which depend on $k_1$.
We now split this using Young's inequality and iterate as in Step 3 to conclude.

\textbf{Step 5: Case of regular $f$.}
  Now suppose $f \in \BB^{\beta-1,p^\prime}_1(\Omega)$ with $\beta \geq \alpha,$ and suppose that $V_{p,\mu}(\D u) \in \BB^{\delta,2}_{\infty,\loc}(\Omega).$
  We will assume $p > 2$ as we otherwise see no improvement, and we will also assume $\beta \leq \frac{2\alpha}{p^{\prime}}$ for the same reason.
  We argue analogously as in Step 1, however we will use a different estimate for $B_2$; if $\beta \leq 1$ we can argue as in \eqref{eq:alpha_f_estimate} to estimate
  \begin{equation}
    B_2 \leq \lvert h\rvert^{\beta} \frac{\|f\|_{\BB^{\beta-1,p^\prime}_{\infty}(B_{s})}}{s-r}\|\Delta_h \D u\|_{\LL^p(B_{r_2})} \leq \lvert h\rvert^{\beta + \frac {2\delta}p}\frac{\|f\|_{\BB^{\beta-1,p^\prime}_{\infty}(B_{s})}}{s-r}\|V_p(\D u)\|_{\BB^{\delta,2}_{\infty}(B_{r_2})}^{\frac2p}
  \end{equation} 
  using Lemma \ref{lem:Vfunc_diff} for the second inequality.
  Now assuming $\beta > 1,$ we will estimate
  \begin{align*}
    B_2  &= \int_{B_s} \phi^2 f\cdot \Delta_h\int_0^1 \D u(x-th) h\,\d t\,\d x \\
         &\lesssim \lvert h \rvert\norm{\phi^2f}_{\BB^{\beta-1,p'}_1(B_s)} \norm{\Delta_h\D u}_{\BB_{\infty}^{1-\beta,p}(B_{r_2})}.
  \end{align*}
  Set $\tau = \frac{2\delta}p + \beta - 1$, which satisfies $\tau \leq 1$ since $\beta \leq \frac{2\alpha}{p^{\prime}}$.
By interpolation along $\sigma\in (0,1)$ we will show that
\begin{equation}\label{eq:negative_besov_estimate}
  \norm{\Delta_h\D u}_{\BB^{-\sigma,p}_{\infty}(B_{r_2})} \lesssim \lvert h\rvert^{\tau} \norm{\D u}_{\BB^{\tau-\sigma}_{\infty}(B_s)}.
\end{equation}  
If $\sigma = 0,$ by definition of the Besov seminorm we have
\begin{equation}
  \norm{\Delta_h\D u}_{\LL^p(B_{r_2})} \leq \lvert h\rvert^{\tau} \norm{\D u}_{\BB^{\tau}_{\infty}(B_s)}.
\end{equation} 
If $\sigma = 1$ note that
\begin{align}
  \norm{\Delta_h\D u}_{\WW^{-1,p}(B_{r_2})} &\leq \norm{\D u}_{\WW^{-1,p}(B_s)}, \\
  \norm{\Delta_h\D u}_{\WW^{-1,p}(B_{r_2})} &\lesssim \lvert h\rvert\norm{\D u}_{\LL^p(B_s)},
\end{align}
which gives
\begin{equation}
  \norm{\Delta_h\D u}_{\WW^{-1,p}(B_{r_2})} \lesssim \lvert h\rvert^{\tau}\norm{\D u}_{\BB^{\tau-1,p}_{\infty}(B_s)}
\end{equation} 
for $\tau\in (0,1)$ by interpolation.
Now interpolating in $\sigma$ we deduce \eqref{eq:negative_besov_estimate} for $\sigma, \tau \in [0,1]$.

  Hence estimating the other terms in the same way as before, we arrive at the estimate
  \begin{equation}\label{eq:higher_mainestimate}
    \begin{split}
      \|\Delta_h V_{p,\mu}(\D u)\|_{\LL^2(B_{r})}^2 
      \lesssim& \frac1{s-r}\lvert h\rvert^{\beta+\frac{2\delta}p} \norm{f}_{\BB^{\beta-1,p^\prime}_{1}(\Omega)} \norm{V_{p,\mu}(\D u)}_{\BB^{\delta,2}_{\infty}(B_s)}^{\frac2p}\\
      &\quad+\frac1{s-r}\lvert h\rvert^{\alpha} \int_{B_{r_2}}(1+\lvert\D u_h\rvert)^{q-1}\lvert\Delta_h \D u\rvert\d x\\
      &\quad+\frac1{s-r}\int_{B_{{r_1}}} (1+\lvert \D u\rvert)^{q-1}\lvert \Delta_h^2 u\rvert \d x\\
      =& C_1 + C_2 + C_3
    \end{split}
  \end{equation}
  as in the end of step 1.

We also estimate the latter two terms slightly differently.
For $C_2$ we use \eqref{eq:vfunction_difference} to estimate
\begin{equation}
  \begin{split}
    C_2 &\leq \frac{\lvert h\rvert^{\alpha}}{s-r} \int_{B_{r_2}}(1+ \lvert \D u\rvert + \lvert\D u_h\rvert)^{q-1}\lvert\Delta_h \D u\rvert\d x \\
        &\lesssim \frac{\lvert h\rvert^{\alpha}}{s-r} \int_{B_{r_2}}(1+ \lvert \D u\rvert + \lvert\D u_h\rvert)^{q-\frac p2}\lvert\Delta_h V_{p,\mu}(\D u)\rvert\d x \\
  \end{split}
\end{equation} 
For $C_3,$ we will write
\begin{equation}
  \Delta_h^2u = u_h + u_{-h} - 2u = h \cdot \int_0^1 \left(\D u_{th} - \D u_{-th}\right) \,\d t,
\end{equation} 
so we can estimate
\begin{equation}
  \begin{split}
  C_3 &\leq \frac{\lvert h\rvert}{s-r}\int_{B_{r_1}} \int_0^1 (1 + \lvert \D u\rvert)^{q-1} \left\vert \D u_{th} - \D u_{-th} \right\rvert \,\d t \,\d x \\
        &\leq \frac{\lvert h\rvert}{s-r}\int_{B_{r_1}} \int_0^1 (1 + \lvert \D u\rvert)^{q-1} \left\vert \D u_{th} - \D u\right\rvert \,\d t \,\d x \\
        &\quad + \frac{\lvert h\rvert}{s-r}\int_{B_{r_1}} \int_0^1 (1 + \lvert \D u\rvert)^{q-1} \left\vert \D u_{-th} - \D u\right\rvert  \,\d t\,\d x \\
        &\leq \frac{\lvert h\rvert}{s-r} \int_0^1 \int_{B_{r_1}} (1 + \lvert \D u\rvert)^{q-\frac p2} \left( \lvert \Delta_{th}V_{p,\mu}(\D u)\rvert + \lvert \Delta_{-th}V_{p,\mu}(\D u)\rvert \right) \,\d x\,\d t.
  \end{split}
\end{equation} 
Now given $V_{p,\mu}(\D u) \in \BB^{\delta,2}_{\infty}(B_s),$ suppose $\tau_1, \tau_2 \geq 1$ and $\sigma,\eps>0$ are such that the embedding ${\BB^{\delta,2}_{\infty}(B_s) \hookrightarrow \LL^{\tau_1-\eps}(B_s) \cap \BB^{\sigma,\tau_2}_{\infty}(B_s)}$ and \eqref{eq:subquad_tau} holds, exactly as in Step 4 (except that now $p\geq 2$).
Then we can estimate
\begin{equation}
  C_2 + C_3 \lesssim \left( \lvert h\rvert^{\alpha + \sigma} + \lvert h\rvert^{1+\sigma} \right) \frac1{s-r} \left( 1 + \norm{V_{p,\mu}(\D u)}_{\LL^{\tau_1-\eps}(B_s)} \right)^{\frac{2q-p}p} \norm{V_{p,\mu}(\D u)}_{\BB^{\delta,2}_{\infty}(B_s)} ,
\end{equation} 
so it follows that $V_{p,\mu}(\D u) \in B^{\delta^{\prime},2}_{\infty}(B_s)$ with
\begin{equation}
\delta^{\prime} = \min\left\{1, \frac{\beta}2+\frac \delta p , \frac{\alpha}2+\frac{\sigma}2\right\},
\end{equation} 
and we have the associated estimate
\begin{equation}
  \begin{split}
  \norm{V_{p,\mu}(\D u)}_{\BB^{\delta^{\prime},2}_{\infty}(B_r)} 
  &\lesssim \frac{\norm{f}_{\BB^{\beta-1,p^\prime}_{1}(\Omega)}^{\frac12}}{s-r} \norm{V_{p,\mu}(\D u)}_{\BB^{\delta,2}_{\infty}(B_s)}^{\frac1p} \\
  &\quad+ \frac1{s-r} \left( 1 + \norm{V_{p,\mu}(\D u)}_{\LL^{\tau_1-\eps}(B_s)} \right)^{\frac{2q-p}{p}} \norm{V_{p,\mu}(\D u)}_{\BB^{\delta,2}_{\infty}(B_s)}^{\frac12}.
  \end{split}
\end{equation} 
From here the iteration is similar as in Step 4, by noting that our analysis did not require $p<2$. 
Indeed starting with $\delta_0 = \frac{\alpha}2$ we inductively show that $V_{p,\mu}(\D u) \in \BB_{\infty}^{\delta_k,2}(B_{\rho_tk})$ where $\rho_k \to R/2$ is to be chosen as before and
\begin{align*}
  \delta_{k+1}=\min\left\{\frac{\beta}2+\frac {\delta_k} p, \frac{\alpha}2 + \frac{\delta_k}2\right\}.
\end{align*}
If $\delta_k < \frac{n(q-p)}{2q-p}$ then we employ a two-step iteration which yields
\begin{equation}
  \delta_{k,j} \geq \min\left\{ \delta_{k+1}, \frac{\alpha}2 + \frac{\kappa_0}4 \sum_{i=1}^j \left( \frac qp \right)^i \right\} 
\end{equation} 
as in the $p < 2$ case.
Iterating this in $j$ and $k$ we can show there is $t \geq 1$ such that $V_p(\D u) \in \BB^{\delta_{k},2}_{\infty}(B_{\rho_{tk}})$ for all $k \leq k_0$, where $\delta_{k_0} \geq \frac{n(q-p)}{2q-p}$. 
Note this is always reached in finitely many steps since $q < \frac{np}{n-\alpha}$.
Then since $\delta_k > \frac{n(q-p)}{2q-p}$ for all $k > k_0$, we may take $\tau_2 = 2$ and $\sigma = \delta_k$ to infer the estimate
\begin{equation}
  \begin{split}
   &\norm{V_{p,\mu}(\D u)}_{\BB^{\delta_{k+1},2}(B_{\rho_{tk+1}})} 
  \leq \frac{C_1}{\rho_{tk+1}-\rho_{tk}} \norm{V_{p,\mu}(\D u)}_{\BB^{\delta_k,2}_{\infty}(B_{\rho_{tk}})}^{\frac1p} \\
  &\qquad + \frac{C_1}{\rho_{tk+1}-\rho_{tk}} \left( 1 + \norm{V_{p,\mu}(\D u)}_{\LL^{\frac{2(2q-1)}p}(B_{\rho_{tk}})} \right)^{\frac{2q-p}{2p}} \norm{V_{p,\mu}(\D u)}_{\BB^{\delta_k,2}_{\infty}(B_{\rho_{tk}})}^{\frac12}.
  \end{split}
\end{equation} 
Now we can iterate this with the same choices of $\xi_k$, $\theta_k$ and $k_1 \geq k_0$ as in Step 4 to pass to the limit.
\end{proof}

\subsection{Improved boundary regularity for radial integrands}\label{sec:boundaryImproved}
Our next goal is to obtain a version of Theorem \ref{thm:interiorImproved} applicable up to the boundary, by adapting the approach in \cite{EbmeyerLiuSteinhauer2005,Ebmeyer2005}.
Our results will be restricted to radial integrands $F\equiv F_0(\lvert z\rvert),$ and to homogeneous boundary conditions in the Dirichlet case.
More precisely, we will prove the following.

\neumannImproved*

Following the procedure in Section \ref{sec:flattening} we can reduce to the case of a flat boundary.
While this procedure only requires $\p\Omega$ to be of class $C^{1,\alpha}$, the $C^{1,1}$ regularity is assumed to ensure the fractional differentiability is preserved under this transformation.
In particular, we will write $\Omega_t = \Omega \cap B_t(0)$, $\Gamma_t = B_t\cap\p\Omega$, and suppose $\Omega_s = B_{s}^+$, $\Gamma_t = B_t \cap \{x_n =0 \}.$
In the case of the Laplacian it is known that solutions in $C^1$ domains are in general no better than $\BB^{\frac 3 2,2}_2(\Omega)$ \cite{Costabel2019}, so additional regularity of $\p\Omega$ is necessary, however we have not attempted to optimise this.

\subsubsection{The boundary difference quotient argument}

 The key ingredient is that we require a version of \eqref{eq:tangential_mainestimate} valid up to the boundary.
 The results in Section \ref{sec:flattening} have already shown that we can locally reduce to the case of the flat boundary, where we will establish our modified estimate.

\begin{lemma}\label{lem:mainEstimateImprovedDiff}
  Suppose $1<p\leq q < \infty,$ $\alpha \in (0,1],$ and suppose $F\equiv F_0(x,\lvert z\rvert)$ satisfies \eqref{def:bounds1}--\eqref{def:bounds31}.
  Let $q < \frac{n+\alpha}n p$ and $f \in \BB^{\alpha-1,p^\prime}_{\infty}(\Omega).$
  In the Dirichlet case assume $g_D = 0$, and in the Neumann case assume $g_N \in \WW^{\alpha-\frac1{p^\prime},p^{\prime}}(\Gamma_s).$
  Assume $u$ is a relaxed minimiser for $\overline \F_D$ with $u=0$ on $\Gamma_s$ or $\overline\F_M$ with $\gamma= \p B_s^+\setminus \Gamma_s$.
  Let $\tilde u$ be an $\WW^{1,p}$-extension of $u$ to $B_s$ given by an odd reflection in the Dirichlet case, and an even reflection in the Neumann case.
  Now let $r<s$ and suppose $ h = h e_i$ with $i\in \{1,\ldots,n\}$ with $\lvert h\rvert \leq \frac{s-r} 3$. Then in the Neumann case we have
\begin{equation}
  \begin{split}
    &\|\Delta_h V_{p,\mu}(\D \tilde u)\|_{\LL^2(\Omega_{r})}^2 + \|\Delta_{-h} V_{p,\mu}(\D \tilde u)\|_{\LL^2(\Omega_{r})}^2\\
    &\quad\lesssim \frac{\lvert h\rvert^{\alpha}}{s-r} \left(\|f\|_{\BB^{\alpha-1,p^\prime}_{\infty}(\Omega_{s})} +  \norm{g_N}_{\WW^{\alpha-\frac1{p^\prime},p^\prime}(\Gamma_s)} \right)\norm{\Delta_h \D \tilde u}_{L^p(\Omega_{r_2})} \\
    &\qquad +\frac1{s-r}\int_{B_{{r_1}}} (1+\lvert \D u\rvert)^{q-1}\lvert \Delta_h^2 \tilde u\rvert \d x\\
    &\qquad+\frac{\lvert h\rvert^{\alpha}}{s-r} \int_{B_{r_2}}(1+\lvert\D \tilde u_{h}\rvert+\lvert\D \tilde u_{-h}\rvert)^{q-1}\left(\lvert\Delta_h \D \tilde u\rvert+\lvert\Delta_{-h}\D \tilde u\rvert\right)\d x
  \end{split}
\end{equation}
Here $r_1 = r + \frac{s-r}3$ and $r_2 = r + \frac{2(s-r)}3.$ In the Dirichlet case, the term involving $g_N$ may be dropped.
  Moreover in both cases, if a-priori $u \in W^{1,q}(\Omega),$ it suffices to assume $q < \frac{np}{n-\alpha}.$

\end{lemma}

A crucial part in the proof will be played by the following elementary lemma:
\begin{lemma}\label{lem:cancellation}
  Suppose $s>0$ and let $\sigma,\tau : (-s,s) \to \bb R^N$ be integrable with both $\sigma$ and $\tau$ either both odd or both even. Then for $\lvert h\rvert \leq \frac{s}2,$ we have
  \begin{equation}
    \int_0^h \Delta_h\left( \sigma(x) \cdot \Delta_{-h}\tau(x) \right) \,\d t =  - \int_0^h \Delta_{-h}\left( \sigma(x) \cdot \Delta_h\tau(x) \right) \,\d t
  \end{equation} 
\end{lemma}

\begin{proof}[Proof of Lemma \ref{lem:mainEstimateImprovedDiff}]
We denote $a(x,t) = \p_z F(x,t)$ and $a_0(x,t) = \p_t F_0(x,t)/t$.
As in Theorem \ref{thm:interiorImproved}, let $\phi$ be a radial cut-off with $\phi =1$ in $B_r$, supported in $B_{r_1}$.
We will also take $\BB^{\alpha-1,p^{\prime}}_{\infty}$-extension of $f$ to $B_s$, which we will also denote by $f$, such that $\norm{f}_{\BB^{\alpha-1,p^{\prime}}_{\infty}(\Omega_s)}\lesssim \norm{f}_{\BB^{\alpha-1,p^{\prime}}_{\infty}(B_s)}$.

\textbf{Step 1: Tangential directions.} 
We will consider difference quotients in the direction $h = h e_i$ with $i\in \{1,\ldots,n-1\}$ and $\lvert h\rvert \leq \frac{s-r} 3$.
Suppose $h = h e_i$ with $i\in \{1,\ldots,n-1\}$ and $\lvert h\rvert \leq \frac{s-r} 3$. Noting that $\phi^2 \Delta_h^2 u$ is an admissible test function in both the Dirichlet and mixed case, there is no problem in applying the proof of Theorem \ref{thm:interiorImproved} in order to obtain \eqref{eq:tangential_mainestimate}.

The only difference is that in the Neumann case we obtain an extra term
\begin{equation}
  B_3 = \int_{\Gamma_s} \phi^2 g_N \cdot \Delta_h^2\tilde u \,\d \H^{n-1},
\end{equation} 
which, analogously as in \eqref{eq:A4Interpolated} from the proof of Theorem \ref{thm:regularityRelaxed}, we estimate as
\begin{equation}
  \lvert B_3\rvert \lesssim \frac{\lvert h\rvert^{\alpha}}{s-r} \norm{g}_{\WW^{\alpha-\frac1{p^\prime},p^\prime}(\Gamma_s)} \norm{\Delta_h \D u}_{L^p(\Omega_{r_2})}.
\end{equation} 
Indeed the case $\alpha = 0$ follows from the same duality pairing and the trace embedding $W^{1,p}(\Omega_s) \hookrightarrow W^{1-\frac1p,p}(\partial\Omega_s)$, whereas for $\alpha = 1$ we write
\begin{equation}\label{eq:B3_FTCextension}
  \begin{split}
    B_3 &= \int_0^s \int_{\Gamma_s} \partial_n\left( \phi^2 g_N \Delta_h^2\tilde u \right)(x',t) \,\d\H^{n-1}(x')\,\d t \\
        &= \int_{\Omega_s} \partial_n\left( \phi^2 \tilde g \Delta_h^2\tilde u \right) \,\d x \\
        &= \int_{\Omega_s} \partial_n\left( \phi^2\tilde g \right) \Delta_h^2\tilde u \,\d x + \int_{\Omega_s} \Delta_h\left( \phi^2\tilde g \right) \Delta_h\partial_nu \,\d x,
  \end{split}
\end{equation} 
where $\tilde g$ is an extension of $g$ to $B_s$ such that the estimate $\norm{\tilde g}_{\WW^{1,p^\prime}(B_s)} \lesssim \norm{g}_{\WW^{1-\frac1{p^{\prime}},p^{\prime}}(\Gamma_s)}$.
Now we can estimate using H\"older as before.
One can also observe that we did not require our difference quotients to be tangential for this step, so the same argument works for $h = he_n$.

\textbf{Step 2: Normal direction.}
Let now $h = h e_n$ with $\lvert h\rvert \leq \frac{s-r} 3$. Similarly as in the first step, we will test the system against $\phi \Delta_h^2\tilde u.$ 
In the Dirichlet case, we observe this is admissible owing to the fact that $\tilde u$ is odd and that $g_D = 0.$
In both cases this gives
\begin{equation}
0=  \int_{\Omega_s} a(x,\D u) \cdot \D\left( \phi^2 \Delta_h^2\tilde u \right)  - f \cdot \phi^2 \Delta_h^2\tilde u \,\d x + \int_{\Gamma_s} g \cdot \phi^2 \Delta_h^2\tilde u \,\d\H^{n-1},
\end{equation} 
which we can write as
\begin{equation}
  \begin{split}
    A^\prime &= - \int_{\Omega_{s}} \phi^2 a(x,\D \tilde u)\cdot \Delta_{h}^2\D \tilde u \d x\\
       &= \int_{\Omega_{s}} a(x,\D u)\cdot \D \phi^2\otimes \Delta_h^2 \tilde u \d x- \int_{\Omega_{s}}\phi^2 f \cdot \Delta_h^2 \tilde u \,\d x + \int_{\Gamma_s} \phi^2 g_N \cdot \Delta_h^2\tilde u \,\d\H^{n-1}\\
       &= B_1^\prime + B_2^\prime + B_3^\prime,
  \end{split}
\end{equation}
understanding that $g_N \equiv 0$ in the Dirichlet case.
For this we observe the identity
\begin{equation}
  \begin{split}
    &\Delta_h\left( \phi^2 a(x,\D \tilde u) \cdot \Delta_{-h} \D \tilde u \right) + \Delta_{-h}\left( \phi^2 a(x,\D \tilde u) \cdot \Delta_{h} \D \tilde u \right)\\
    &= \Delta_h\left( \phi^2 a(x,\D \tilde u) \right) \cdot \left( \Delta_h \D \tilde u \right)_h + \phi^2 a(x,\D \tilde u) (-\Delta_h^2\D \tilde u) \\
    &\quad + \Delta_{-h}\left( \phi^2 a(x,\D \tilde u) \right) \cdot \left( \Delta_h \D \tilde u \right)_{-h} + \phi^2 a(x,\D \tilde u) (-\Delta_h^2\D \tilde u)
  \end{split}
\end{equation} 
holds, recalling that $\Delta_h \Delta_{-h} = - \Delta_h^2.$
Then, since $(\Delta_h\D \tilde u)_{-h} = - \Delta_{-h}\D \tilde u$ and further ${(\Delta_{-h}\D \tilde u)_h = - \Delta_h\D \tilde u,}$ we can write
\begin{equation}
  \begin{split}
    2A^\prime &= \int_{\Omega_s} \Delta_h\left( \phi^2 a(x,\D \tilde u) \right) \cdot \Delta_h \D \tilde u \,\d x
        + \int_{\Omega_s} \Delta_{-h}\left( \phi^2 a(x,\D \tilde u) \right)  \cdot \Delta_{-h}\D \tilde u \,\d x \\
       &\quad + \int_{\Omega_s} \Delta_h\left( \phi^2 a(x,\D \tilde u) \cdot \Delta_{-h} \D \tilde u\right) \,\d x 
              + \int_{\Omega_s} \Delta_{-h}\left( \phi^2 a(x,\D \tilde u) \cdot \Delta_{h} \D \tilde u\right) \,\d x\\
       &= A_1^\prime + A_2^\prime + A_3^\prime + A_4^\prime.
  \end{split}
\end{equation} 
Unlike in the tangential case the last two terms do not necessarily vanish, however by exploiting the symmetries of the system we will show that they cancel; the regions of integration of the respective terms are shown in Figure \ref{fig:diff_quotients}.
For $1 \leq i \leq n$, let $a_i$ denote the $i^{\mathrm{th}}$ column of $a,$ so that $a_i(x,\D u) = a_0(x,\lvert \D u\rvert) \p_i u,$ as then we can write $A_3^\prime = \sum_{i=1}^n A_{3,i}^\prime$ where
\begin{equation}\label{eq:boundary_diff_A3i}
  \begin{split}
    A_{3,i}^\prime &:= \int_{\Omega_s} \Delta_h \left( \phi^2 a_0(x,\D \tilde u) \cdot \Delta_{-h}\partial_i\tilde u \right)  \,\d x \\
             &= \int_{\Omega_s \setminus (\Omega_s+h)} \Delta_h \left( \phi^2 a_i(x,\lvert\D \tilde u\rvert)\,\p_i\tilde u \cdot \Delta_{-h}\partial_i\tilde u \right)  \,\d x \\
             &= \int_0^h \int_{B^{n-1}_s} \Delta_h \left( \phi^2 a_i(x,\lvert \D \tilde u\rvert) \p \tilde u_i \cdot \left( \p_i\tilde u(x^\prime,x_n-h) - \p_i \tilde u(x^\prime,x_n) \right) \right)  \,\d x^\prime \,\d x_n.
  \end{split}
\end{equation} 

\begin{figure}[ht]
    \centering
    \begin{subfigure}{0.45\textwidth}
\begin{tikzpicture}[scale=.4]
	\begin{pgfonlayer}{nodelayer}
		\node [style=none] (0) at (-5, 0) {};
		\node [circle,fill,minimum size = 3pt,inner sep = 0pt, outer sep = 0pt] (1) at (0, 0) {};
		\node [style=none] (2) at (5, 0) {};
		\node [style=none] (3) at (0, 5) {};
		\node [style=none] (4) at (-5, -1) {};
		\node [style=none] (5) at (0, 4) {};
		\node [style=none] (6) at (5, -1) {};
		\node [circle,fill,minimum size = 3pt,inner sep = 0pt, outer sep = 0pt] (7) at (0, -1) {};
		\node [style=none] (8) at (0, 7) {};
		\node [style=none] (9) at (-7, 0) {};
		\node [style=none] (10) at (7, 0) {};
		\node [style=none] (11) at (0, -3) {};
		\node [style=none] (12) at (.9, 7) {\scriptsize $x_n$};
		\node [style=none] (13) at (0.5, 0.5) {\scriptsize $0$};
		\node [style=none] (14) at (0.6, -0.5) {\scriptsize $-h$};
		\node [style=none] (15) at (7, .8) {\scriptsize $x'$};
		\node [style=none] (16) at (-4.9, 0) {};
		\node [style=none] (17) at (4.9, 0) {};
	\end{pgfonlayer}
	\begin{pgfonlayer}{edgelayer}
    \fill[fill=gray!20] (4.center) to[bend left = 5] (16.center) to (17.center) to[bend left =5] (6.center) to cycle;
    \draw [->,thick](1.center) to (8.center);
		\draw [->,thick](1.center) to (10.center);
		\draw [->,thick](1.center) to (9.center);
		\draw [->,thick](1.center) to (11.center);
		\draw [bend left=45] (0.center) to (3.center);
		\draw [bend left=45] (3.center) to (2.center);
		\draw [dashed,bend left=45] (4.center) to (5.center);
		\draw [dashed,bend left=45] (5.center) to (6.center);
    \draw [dashed](4.center) to (7.center);
		\draw [dashed](7.center) to (6.center);
	\end{pgfonlayer}
\end{tikzpicture}
      \subcaption{$A_3'$}
    \end{subfigure}
    \begin{subfigure}{0.45\textwidth}
      \begin{tikzpicture}[scale=.4]
	\begin{pgfonlayer}{nodelayer}
		\node [style=none] (0) at (-5, 0) {};
		\node [circle,fill,minimum size = 3pt,inner sep = 0pt, outer sep = 0pt] (1) at (0, 0) {};
		\node [style=none] (2) at (5, 0) {};
		\node [style=none] (3) at (0, 5) {};
		\node [style=none] (4) at (-5, 1) {};
		\node [style=none] (5) at (0, 6) {};
		\node [style=none] (6) at (5, 1) {};
		\node [circle,fill,minimum size = 3pt,inner sep = 0pt, outer sep = 0pt] (7) at (0, 1) {};
		\node [style=none] (8) at (0, 8) {};
		\node [style=none] (9) at (-7, 0) {};
		\node [style=none] (10) at (7, 0) {};
		\node [style=none] (11) at (0, -2) {};
		\node [style=none] (12) at (.9, 8) {\scriptsize $x_n$};
		\node [style=none] (13) at (0.5, 0.5) {\scriptsize $0$};
		\node [style=none] (14) at (0.5, 1.5) {\scriptsize $h$};
		\node [style=none] (15) at (7, 0.8) {\scriptsize $x'$};
	\end{pgfonlayer}
	\begin{pgfonlayer}{edgelayer}
    \fill[fill=gray!20] (4.center) to[bend left = 5] (16.center) to (17.center) to[bend left =5] (6.center) to cycle;
		\draw [->,thick](1.center) to (8.center);
		\draw [->,thick](1.center) to (10.center);
		\draw [->,thick](1.center) to (9.center);
		\draw [->,thick](1.center) to (11.center);
		\draw [bend left=45] (0.center) to (3.center);
		\draw [bend left=45] (3.center) to (2.center);
		\draw [dashed,bend left=45] (4.center) to (5.center);
		\draw [dashed,bend left=45] (5.center) to (6.center);
    \draw [dashed](4.center) to (7.center);
    \draw [dashed](7.center) to (6.center);
	\end{pgfonlayer}
\end{tikzpicture}
      \subcaption{$A_4'$}
    \end{subfigure}
    \caption{Regions of integration in $A_3'$ and $A_4'$}
    \label{fig:diff_quotients}
\end{figure}

We apply Lemma \ref{lem:cancellation} with
\begin{align}
  \sigma(t) = \phi(x^\prime,t) a_0(x^\prime,t,\lvert \D \tilde u\rvert(x,t)) \p_i\tilde u(x^\prime,t), \quad \tau(t) = \p_i(x^\prime,x_n)
\end{align}
for each $1 \leq i \leq n$ and $x^\prime \in B_s^{n-1}.$
Across the $x_n$-axis we have $\partial_iu$ is odd or even depending on whether $i=n$ and the boundary condition, and hence $a_0(x,\lvert \D u\rvert)$ is even.
Also $\phi$ is even along the $x_n$-axis since it is a radial cutoff, so it follows that $\sigma$ and $\tau$ have the same parity.
Hence, using this lemma in \eqref{eq:boundary_diff_A3i} and summing over $n$, we deduce that $A_3^\prime + A_4^\prime = 0.$

Now we can proceed as in the tangential case, where we write
\begin{equation}
  \begin{split}
    A_1^\prime &= \int_{\Omega_s} \phi^2 \left(a(x,\D \tilde u_h) - a(x,\D \tilde u)\right) \cdot \Delta_h \D \tilde u \,\d x \\
         &\quad+ \int_{\Omega_s} \phi^2 \left(a(x+h,\D \tilde u_h) - a(x,\D \tilde u_h)\right) \cdot \Delta_h \D \tilde u \,\d x \\
         &\quad+ \int_{\Omega_s} \Delta_h\phi^2 \, a(x,\D \tilde u)_h \Delta_h \tilde u \,\d x,
\end{split}
\end{equation} 
so by \eqref{eq:h2bound_derivative}, \eqref{eq:hbound1_quantiative}, \eqref{def:bounds31} we can estimate
\begin{equation}
  A_1^\prime \gtrsim \int_{\Omega_s} \phi^2 \lvert \Delta_hV_{p,\mu}(\D \tilde u)\rvert^2 \,\d x - \frac{C}{s-r} \left( \lvert h\rvert + \lvert h\rvert^{\alpha} \right) \int_{B_{r_2}} (1 + \lvert \D \tilde u_h\rvert)^{q-1} \lvert \Delta_h \D \tilde u\rvert\,\d x,
\end{equation} 
and similarly
\begin{equation}
  A_2^\prime \gtrsim \int_{\Omega_s} \phi^2 \lvert \Delta_{-h}V_{p,\mu}(\D \tilde u)\rvert^2 \,\d x - \frac{C}{s-r} \left( \lvert h\rvert + \lvert h\rvert^{\alpha} \right) \int_{B_{r_2}} (1 + \lvert \D \tilde u_{-h}\rvert)^{q-1} \lvert \Delta_{-h} \D \tilde u\rvert\,\d x,
\end{equation} 
which gives us our estimates for $A^\prime = A_1^\prime + A_2^\prime.$
For the remaining terms we can as before estimate
\begin{align}
  B_1^\prime &\lesssim \frac1{s-r}\int_{\Omega_{r_1}} (1 + \lvert \D \tilde u\rvert)^{q-1} \lvert\Delta_h^2\tilde u\rvert\,\d x,\\
  B_2^\prime &\lesssim \frac{\lvert h\rvert^{\alpha}}{s-r} \|f\|_{\BB^{1-\alpha,p^\prime}_{\infty}(\Omega_{s})}\|\Delta_h \D \tilde u\|_{\LL^p(B_{r_2})},\\
  B_3^\prime &\lesssim \frac{\lvert h\rvert^{\alpha}}{s-r} \norm{g}_{\WW^{\alpha-\frac1{p^\prime},p^\prime}(\Gamma_s)} \norm{\Delta_h \D \tilde u}_{\LL^p(\Omega_{r_2})}.
\end{align}
Indeed for $B_2^\prime$, the direction of the difference quotient plays no role provided we estimate the right-hand side in the full ball $B_s$, which is also the case for $B_3^\prime$ as noted in Step 1.
Therefore collecting our estimates we have shown the desired estimate.
\end{proof}

To extend Step 5 from the proof of Theorem \ref{thm:interiorImproved}, we will need an analogous interpolation estimate for $f$, along with an analogous estimate for $g_N$.

\begin{lemma}\label{lem:boundary_regularf}
  Suppose $p \geq 2$, $\alpha \in (0,1)$, $\mu>0$ and $\beta \in [\alpha,\frac{2\alpha}{p^{\prime}}]$. Then for $0 < \delta < \alpha$ and $0<s<r$, let $v \in \WW^{1,p}(B_s)$ such that $V_{p,\mu}(\D v) \in \BB^{\delta,2}_{\infty}(B_s)$, and $0 \leq \phi \leq 1$ be a suitable cut-off supported on $B_{r+\frac{s-r}3}$.
  Also let $f \in \BB^{\beta-1,p^{\prime}}_1(\Omega_s)$ and $g \in \BB^{\beta-\frac1{p^{\prime}},p^{\prime}}_1(\Gamma_s)$.
  Then for $h = he_i$ with $1 \leq i \leq n$ and $\lvert h\rvert \leq \frac{s-r}3$, we have
  \begin{align}
    \int_{\Omega_s} \phi^2 f \Delta_h^2v \,\d x &\lesssim \lvert h\rvert^{\beta + \frac{2\delta}p}\norm{\phi}_{\WW^{1,\infty}(B_s)} \norm{f}_{\BB^{\beta-1,p^{\prime}}_1(B_s)} \norm{V_p(\D v)}_{\BB^{\delta,2}_{\infty}(B_s)}^{\frac2p},\\
    \int_{\Gamma_s} \phi^2 g \Delta_h^2v \,\d \H^{n-1} &\lesssim \lvert h\rvert^{\beta + \frac{2\delta}p}\norm{\phi}_{\WW^{1,\infty}(B_s)} \norm{g}_{\BB^{\beta-\frac1{p^{\prime}},p^{\prime}}_1(\Gamma_s)} \norm{V_p(\D v)}_{\BB^{\delta,2}_{\infty}(B_s)}^{\frac2p}.
  \end{align}
\end{lemma}
\begin{proof}
  The interpolation argument in Lemma \ref{lem:mainEstimateImprovedDiff} already establishes the case $\beta \leq 1$, so we will assume $\beta > 1$.
  Additionally we will assume $f$ admits a norm-preserving $\BB^{\beta-1,p^{\prime}}_{\infty}$-extension to $B_s$ which we continue to denote by $f$.

  The estimate for $f$ is analogous as in the interior case, as taking difference quotients in the normal direction does not changes our analysis. 
  Indeed we use the duality pairing
  \begin{equation}
    \int_{\Omega_s} \phi^2 f \cdot \Delta_h^2 v \,\d x \leq \lvert h\rvert \norm{\phi^2f}_{\BB^{\beta-1,p^{\prime}}_1(B_s)} \norm{\Delta_hv}_{\BB^{1-\beta,p}_{\infty}(B_{r_2})},
  \end{equation} 
  and use the same interpolation estimate \eqref{eq:negative_besov_estimate} from Step 5 of the interior case.

  For the $g$ term, we observe that since $\beta > 1$ by \cite[Theorem 2.7.2]{Triebel1983}, we can extend $g$ to a function $\tilde g \in \BB^{\beta,p^{\prime}}_1(B_s)$ satisfying
  \begin{equation}
    \norm{\tilde g}_{\BB^{\beta,p^{\prime}}_1(B_s)} \lesssim \norm{g}_{\BB^{\beta-\frac1{p^{\prime}},p^{\prime}}_1(\Gamma_s)}.
  \end{equation} 
  Then using \eqref{eq:B3_FTCextension} from the proof of Lemma \ref{lem:mainEstimateImprovedDiff} we can estimate
  \begin{equation}
    \begin{split}
    \int_{\Gamma_s} \phi^2 g \Delta_h^2v \,\d\H^{n-1} 
    &= \int_{\Omega_s} \partial_n\left( \phi^2\tilde g \right) \Delta_h^2v \,\d x + \int_{\Omega_s} \Delta_h\left( \phi^2\tilde g \right) \Delta_h\partial_nv \,\d x \\
    &\lesssim \lvert h\rvert \norm{\phi}_{\WW^{1,\infty}(B_s)}\norm{\D\tilde g}_{\BB^{\beta-1,p^{\prime}}_1(B_s)} \norm{\Delta_hv}_{\BB^{1-\beta,p}_{\infty}(B_s)},
    \end{split}
  \end{equation} 
  and the latter term we once again use \eqref{eq:negative_besov_estimate} similarly as with the estimate for $f$, which establishes the result.
\end{proof}

\subsubsection{Odd and even extensions}\label{sec:oddeven_extension}

Lemma \ref{lem:mainEstimateImprovedDiff} indicates that we want to construct $\WW^{1,q}$-extensions of solutions that preserve fractional differentiability.
Under zero Dirichlet boundary conditions, we will take an odd extension given by
\begin{align}\label{eq:extensionDirichlet}
  \tilde u(x^\prime,x_n) = \begin{cases}
		u(x^\prime,x_n) \quad&\text{ if } x_n\geq 0\\
		-u(x^\prime,-x_n) &\text{ if } x_n<0
		\end{cases}
\end{align}

Observe that since $u$ vanishes on $\Gamma_s,$ by Lemma \ref{lem:extension} it is weakly differentiable in $B_s.$
Moreover by \cite[Theorem 4.5.2]{Triebel1992} (see also \cite[Theorem 2.9.2]{Triebel1983}) for all $\delta \in (0,1)$ we have the $\BB^{1+\delta,p}_{\infty}$-regularity is preserved with the corresponding estimate
\begin{equation}\label{eq:dirichlet_extension}
  \seminorm{\D \tilde u}_{\BB^{\delta,p}_{\infty}(B_s)} \lesssim \seminorm{\D u}_{\BB^{\delta,p}_{\infty}(\Omega_s)},
\end{equation} 
and we have an analogous estimate at the level of $V$-functionals, namely that
\begin{equation}\label{eq:dirichlet_extension2}
  \seminorm{V_p(\D \tilde u)}_{\BB^{\delta,p}_{\infty}(B_s)} \lesssim \seminorm{\D u}_{\BB^{\delta,p}_{\infty}(\Omega_s)},
\end{equation} 
In the case of zero Neumann boundary, we can instead take an even extension given by
\begin{equation}\label{eq:evenextension}
  \tilde u(x^\prime,x_n) = \begin{cases}
		u(x^\prime,x_n) \quad&\text{ if } x_n\geq 0,\\
    u(x^\prime,-x_n) &\text{ if } x_n<0.\end{cases}
\end{equation} 
Note this preserves fractional differentiability up to order $\delta < \frac1p,$ as illustrated by the following lemma.

\begin{lemma}\label{eq:zeroextension}
  Let $1 < p < \infty$ and $0 < \delta < \frac1p.$
  Then given $f \in \BB^{\delta,p}_{\infty}(B_1^+),$ let $\tilde f$ be the extension of $f$ to $B_1$ by zero.
  Then $\tilde f \in \BB^{\delta,p}_{\infty}(B_1)$ and
  \begin{equation}\label{eq:even_extension_bdd}
    \| \tilde f\|_{\BB^{\delta,p}(B_1)}\lesssim \| f\|_{\BB^{\delta,p}(B_1^+)}.
  \end{equation} 
\end{lemma}

This can be deduced from \cite[Section 2.8.7]{Triebel1983}, which asserts that the characteristic function $\chi_{\{x_n > 0\}}$ is a multiplier for $\BB^{\delta,p}_{\infty}(\bb R^n)$ provided $\delta < \frac1p$.
Here we instead present an elementary proof, based on a reduction to the one-dimensional case.

\begin{proof}
  It suffices to verify fractional differentiability in the direction $e_n.$
  For $h >0,$ we have
  \begin{equation}
    \begin{split}
      \int_{B_{1-h}} \lvert\Delta_{he_n}( \chi_{\{x_n > 0\})} f )\rvert^p \,\d x 
      &\lesssim \int_{B_{1-h}} \chi_{\{x_n > h\}} \lvert \Delta_{h e_n} f\rvert^p \,\d x \\
      &\quad + \int_{B_{1-h}} \chi_{\{ 0 < x_n < h\}} \lvert f\rvert^p \,\d x\\
      &= I_1 + I_2.
    \end{split}
  \end{equation} 
  Evidently $I_1 \leq h^{\delta} \seminorm{f}_{\BB^{\delta,p}_{\infty}(B_1^+)}.$
  For the second term note, that by Fubini we have
  \begin{equation}
    \int_{B_{1-h}^{n-1}} \seminorm{f(x^\prime,\cdot)}_{\BB^{\delta,p}_{\infty}((0,h))}^p \,\d x \lesssim \seminorm{f}_{\BB^{\delta,p}_{\infty}(B^+)}^p,
  \end{equation} 
  since, by considering a countable dense set of $\eps \in (0,h)$, we can show that
  \begin{equation}
    \seminorm{f(x^\prime,\cdot)}_{\BB^{\delta,p}_{\infty}((0,h))}^p = \sup_{0<\eps<h} \frac1{\eps^{\delta p}} \int_0^h \lvert f(x^\prime,h+\eps) - f(x^\prime,h)\rvert^p \,\d x < \infty
  \end{equation} 
  for $\H^{n-1}$-almost all $x^\prime \in B_{1-h}^{n-1}.$
  Then for any such $x^\prime,$ by Sobolev embedding we have $f(x^\prime,\cdot) \in \BB^{\delta,p}_{\infty}((0,h)) \hookrightarrow \LL^{\frac{p}{1-\delta p}}((0,h))$ provided $\delta p < 1.$
  In this case, by H\"older we can split
  \begin{equation}
    \int_0^h \lvert f(x^\prime,x_n)\rvert^p \,\d x_n \lesssim \lvert h\rvert^{\delta p} \norm{f(x^\prime,\cdot)}_{\LL^{\frac{p}{1-\delta p}}(0,h)}^p \lesssim  \lvert h \rvert^{\delta p}\seminorm{f(x^\prime,\cdot)}_{\BB^{\delta,p}_{\infty}((0,h))},
  \end{equation} 
  so integrating over $x^\prime \in B_{1-r}^{n-1}$ we deduce that
  \begin{equation}
    I_2 \leq \lvert h\rvert^{\delta p}\seminorm{f}_{\BB^{\delta,p}_{\infty}(B^+)}^p,
  \end{equation} 
  from which the result follows.
\end{proof}

In the Neumann case, we can use this to perform our iteration to deduce differentiability given by some $\beta > \frac12$ for $V_{p,\mu}(\D u)$ as we will see in Theorem \ref{thm:radial_neumann} below.
However beyond this point our extension will no longer preserve the regularity of $u,$ so our iteration scheme will be forced to terminate.
To go beyond this range we will need to assume $g_N=0$ and apply the following lemma, which will be used in the proof of Theorem \ref{thm:relaxedImproved}.

\begin{lemma}\label{lem:zeroNeumann}
  Suppose $v \in \WW^{1,q}(B_1^+)$ satisfies
  \begin{equation}
    \int_{B_1^+} a_0(x,\lvert \D v\rvert) \D v \cdot \D \varphi - f \cdot \varphi \,\d x = 0
  \end{equation} 
  for all $\varphi \in \WW^{1,q}(\Omega),$
  where $a_0 \colon \overline{\Omega} \times [0,\infty) \to \bb R$ is continuous satisfying the growth bounds
  \begin{equation}\label{eq:a0_growthbound}
    t^{q-2} \lesssim a_0(x,t) \lesssim 1+t^{q-2}.
  \end{equation} 
  Then, if we additionally have $V_{q,\mu}(\D v) \in \BB^{\delta,2}_{\infty}(B_s^+)$ for some $\delta>\frac12$ and all $s<1,$ we have 
  \begin{equation}
    \p_n v = 0 \text{ on }\ \Gamma_1.
  \end{equation} 
\end{lemma}
\begin{proof}
  By our assumption on $V_{q,\mu}(\D v),$ we have $\D v$ admits a trace on $\Gamma_s$ for each $s<1$ by \cite[Section 4.4.2]{Triebel1992}, so by taking the precise representative this will be understood as $\D v.$

Let $\varphi \in C^{\infty}_c(B_1^{n-1})$ where $B_1^{n-1}$ denotes the unit ball in $\bb R^{n-1}.$
In particular, there is $s < 1$ such that $\varphi$ is supported in $B^{n-1}_s.$
Then for $\e$ sufficiently small we have ${\phi(x^\prime)\,(\e-x_n) 1_{\{x_n\leq \e\}}}$ is a valid test function, and using this gives
\begin{align*}
  0 &= \int_{B_1^+\cap \{x_n\leq \e\}} a_0(x,\lvert \D v\rvert)\D v\cdot e_n \otimes \phi(x^\prime) \,\d x\\
    &\quad+ \int_{B_1^+\cap \{x_n\leq \e\}} \left(a_0(x,\lvert\D v\rvert) \D v\cdot \D\phi + f \cdot \phi \right)(\e - x_n) \d x 
\end{align*}
The second integral, we can bound using \eqref{eq:h2bound_derivative} by
\begin{align*}
  C\e \int_{B_1^+ \cap \{ x_n \leq \e\}} \|\D\phi\|_{\LL^\infty(B)}(1+\lvert \D v\rvert)^{q-1} + \lvert f\rvert \,\d x.
\end{align*}
Since the integral vanishes as $\e \to 0,$ it follows that
\begin{equation}
  \lim_{\e \to 0} \frac1{\e} \int_{B_1^+\cap \{x_n\leq \e\}} a_0(x,\lvert \D v\rvert)\D v\cdot e_n \otimes \phi(x^\prime) \d x = 0.
\end{equation}
Since $\D v \in \LL^{q}(\Gamma_s)$ as noted above, we can pass to the limit to get
\begin{align*}
  0 = \int_{\Gamma_1} a_0(x,\lvert \D v\rvert) \D v\cdot e_n \otimes \phi(x^\prime) \d x^\prime
\end{align*}
In particular, we deduce that 
\begin{equation}
  a_0(x,\lvert \D v \rvert)\partial_n v = 0
\end{equation} 
$\H^{n-1}$-almost everywhere on $\Gamma_1$.
By \eqref{eq:a0_growthbound}, we deduce that $\lvert \D v\rvert^{q-2}\lvert\p_n v\rvert = 0$ $\H^{n-1}$-almost everywhere on $\Gamma_1$, and in particular this implies that $\p_n v=0$ on $\Gamma_1$ in the sense of traces.
\end{proof}

\subsubsection{Proof of boundary differentiability}

We now collect the results from the previous sections to prove our global differentiability results, starting with the Dirichlet case.
In this case the odd extension always preserves the fractional regularity of $\D u$, so the iteration steps follow by a routine modification of the iteration in the interior case.

\begin{proof}[Proof of {Theorem \ref{thm:relaxedImproved}} in the Dirichlet case]
  By a standard straightening of the boundary described in Section \ref{sec:flattening}, it suffices to consider the case where $\Omega = B_1^+$ and $u$ is a relaxed minimiser for $\F_D$.
  Note in particular, that it is possible to localise the relaxed functional due to Lemma \ref{lem:additivityNeumann}, and so we need to prove improved differentiability in $\Omega_r = B_r^+$ for some $r<1$. 
  Taking an odd extension $\tilde u$ of $u$ and making the obvious changes to the domains of integration, we can employ the iteration technique of Theorem \ref{thm:interiorImproved} without change using the bound in Lemma \ref{lem:mainEstimateImprovedDiff}, namely that
  \begin{equation}\label{eq:boundary_diff_estimate}
  \begin{split}
    &\|\Delta_h V_{p,\mu}(\D \tilde u)\|_{\LL^2(\Omega_{r})}^2 + \|\Delta_{-h} V_{p,\mu}(\D \tilde u)\|_{\LL^2(\Omega_{r})}^2\\
    &\quad\lesssim \frac{\lvert h\rvert^{\alpha}}{s-r} \|f\|_{\BB^{\alpha-1,p^\prime}_{\infty}(\Omega_{s})}\norm{\Delta_h \D \tilde u}_{L^p(\Omega_{r_2})} \\
    &\qquad +\frac1{s-r}\int_{B_{{r_1}}} (1+\lvert \D \tilde u\rvert)^{q-1}\lvert \Delta_h^2 \tilde u\rvert \d x\\
    &\qquad+\frac{\lvert h\rvert^{\alpha}}{s-r} \int_{B_{r_2}}(1+\lvert\D \tilde u_{h}\rvert+\lvert\D \tilde u_{-h}\rvert)^{q-1}\left(\lvert\Delta_h \D \tilde u\rvert+\lvert\Delta_{-h}\D \tilde u\rvert\right)\d x,
  \end{split}
\end{equation}
holds for $h = he_i$ with $\lvert h\rvert \leq \frac{s-r}3$.
Using this estimate as a basis, which is analogous to \eqref{eq:tangential_mainestimate}, we can repeat Steps 2 and 4 in the proof of Theorem \ref{thm:interiorImproved}, noting that the differentiability of the extension is preserved by \eqref{eq:dirichlet_extension2}.
That is, for $\delta < 1$ we have
\begin{equation}
  \begin{split}
    V_{p,\mu}(\D u) \in \BB^{\delta,2}_{\infty}(\Omega_{s}) 
    &\implies V_{p,\mu}(\D \tilde u) \in \BB^{\delta,2}_{\infty}(B_{s}) \\
    &\implies V_{p,\mu}(\D \tilde u) \in \BB^{\delta^{\prime},2}_{\infty}(B_{r})\\
    &\implies V_{p,\mu}(\D u) \in \BB^{\delta^{\prime},2}_{\infty}(\Omega_{r}),
  \end{split}
\end{equation} 
where
\begin{equation}\label{eq:deltaprime}
  \delta^{\prime} = \frac{\alpha}2 + \frac{\min\{\delta,\sigma\}}{\max\{2,p\}},
\end{equation} 
using the notation from the proof of Theorem \ref{thm:interiorImproved}. 
Which can now iterate this and pass to the limit as before, noting the estimates take an analogous form as \eqref{eq:step3_iteration_estimate}, \eqref{eq:step4_iteration_estimate} from Steps 3 and 4 of the interior case respectively, except that the norms are now taken over $\Omega_{\rho_{k}}$.
In the case that $f$ is more regular, we modify the estimate \eqref{eq:boundary_diff_estimate} using Lemma \ref{lem:boundary_regularf} and argue exactly as in Step 5 of the proof of Theorem \ref{thm:interiorImproved}.
\end{proof}

We now turn to the Neumann case, where we first allow for non-homogeneous boundary conditions.

\radialNeumann*

\begin{proof}
  As in the Dirichlet case we can reduce to the case of a flat boundary, using Lemma \ref{lem:additivityNeumann} to reduce to the case of a relaxed minimiser $u$ for $\F_M$ on $\Omega_1 = B_1^+$ with $\gamma = \p B_1^+\setminus\Gamma_1$.
  We will extend $u$ by $\tilde u$ to $B_1$ by means of an even extension.
  Then employing Lemma \ref{lem:mainEstimateImprovedDiff} we infer \eqref{eq:boundary_diff_estimate}, except that we have an additional term arising from $g_N$. 
  We can then infer the same improvement in differentiability
\begin{equation}
  \begin{split}
    V_{p,\mu}(\D u) \in \BB^{\delta,2}_{\infty}(\Omega_{s}) 
    &\implies V_{p,\mu}(\D u) \in \BB^{\delta^{\prime},2}_{\infty}(\Omega_{r}),
  \end{split}
\end{equation} 
with $\delta^{\prime}$ given by \eqref{eq:deltaprime} \emph{provided} that $\delta < \frac12$, as the odd reflection no longer preserves the $\BB^{\delta,2}_{\infty}$-regularity beyond this range.
If $\frac{\alpha}2 \min\{2,p^{\prime}\} \leq \frac12$, then $\delta_k < \frac12$ for every $k$, where $\delta_k$ is as defined in Steps 2--4, so we can carry out the iteration and pass to the limit as before.
Otherwise there is some $k \in \mathbb N$ such that $\delta_k \geq \frac12$, where we must terminate our iteration process. 
In this case we can still show that $V_{p,\mu}(\D u) \in \BB^{\delta,2}_{\infty}(\Omega)$ for all $\delta < \frac12$, and for $\delta$ sufficiently close to $\frac12$ we have $\delta^{\prime} > \frac12$,  where $\delta^{\prime}$ is defined as in \eqref{eq:deltaprime}.
Thus running the iteration step once more we find some $\delta_0 := \delta^{\prime} > \frac12$ such that $V_{p,\mu}(\D u) \in \BB^{\delta_0,2}_{\infty}(\Omega)$.

The case when $p \geq 2$ and $f \in \BB^{\beta-1,p^{\prime}}_1(\Omega)$, $g_N \in \BB^{\beta-\frac1{p^{\prime}},p^{\prime}}_1(\Omega)$ is analogous, except that we use Lemma \ref{lem:boundary_regularf} to estimate the $f$ and $g_N$ terms to obtain an improvement
\begin{equation}
  \delta \mapsto \delta^{\prime} = \min\left\{1, \frac{\beta}2+\frac \delta p , \frac{\alpha}2+\frac{\sigma}2\right\},
\end{equation} 
provided $\delta < \frac12$.
Here we have used the fact that the estimate for $g_N$ is identical to that for $f$.
Again if $\min\left\{\frac{p^{\prime}\beta}2,\alpha\right\} \leq \frac12$ then we can iterate this indefinitely and pass to the limit to conclude, and otherwise we obtain differentiability for some $\delta_0 > \frac12$ by terminating the iteration early.
\end{proof}

Using this and the results from the previous sections, we can now complete the proof of Theorem \ref{thm:relaxedImproved} in the Neumann case.

\begin{proof}[Proof of {Theorem \ref{thm:relaxedImproved}} in the Neumann case]
  As in Theorem \ref{thm:radial_neumann}, we can reduce to the case when $\Omega = B_1^+.$ We focus first on the case where $f\in B^{\alpha-1,p^\prime}_\infty(\Omega)$.
  In the Neumann case, the result follows by Theorem \ref{thm:radial_neumann} when $\alpha \leq \frac1{\min\{2,p^\prime\}}.$
  Otherwise there is $\delta > \frac12$ such that $V_{p,\mu}(\D u) \in \BB^{\delta,p}_{\infty}(\Omega_r)$ for all $r<1$, and we use the fact that $g_N = 0$ to iterate further.

  \textbf{Claim}: If $\alpha > \frac1{\min\{2,p^\prime\}},$ then we have $\p_n u = 0$ on $\Gamma.$

  We will establish this for minimisers $u_{\e}$ of the relaxation from Section \ref{sec:relaxedVsRegularised}, which in this localised form minimises
  \begin{equation}
    \F_{\e}(v) = \int_{B_1^+} F(x,\D v) - f \cdot v \,\d x + \e \int_{B_1^+} \lvert \D v\rvert^q \,\d x,
  \end{equation} 
  over $v \in \WW^{1,q}(B_1^+)$ such that $v = u$ on $\gamma = \p B_1^+ \setminus \Gamma_1.$
  Then each $\F_\e$ satisfies the hypotheses of Theorem \ref{thm:radial_neumann} with $p = q,$ so given our assumption on $\alpha$ we know that $V_{p,\mu}(\D u_{\e}) \in \BB^{\delta,2}_{\infty}(B_s^+)$ for some $\delta > \frac12$ and $s < 1.$
  Hence by Lemma \ref{lem:zeroNeumann} it follows that $\p_n u_{\e} = 0$ on $\Gamma_s$ for each $s<1.$
  Now passing to the limit, noting $u_{\e}$ converges to $u$ in $\WW^{1,p}(B_1^+),$ the claim follows.

  Hence, applying \cite[Section 4.5.2]{Triebel1992}, we have an even extension of $u$ preserves the $\BB^{s,2}_{\infty}$-norm of $V_{p,\mu}(\D u)$ for all $s<1$ since $\p_nu$ is odd and vanishes on $\Gamma_1.$
  Now continuing to iterate as in Theorem \ref{thm:interiorImproved}, the result follows. The argument in the case $f\in B^{\beta-1,p^\prime}_\infty(\Omega)$ with $p\geq 2$ and $\beta\in [\alpha,2]$ is analogous.
\end{proof}

\begin{remark}
  Due to Corollary \ref{cor:autonomous} and the corresponding result in the Dirichlet case in \cite{Koch2021a}, if $F$ is autonomous, Theorem \ref{thm:relaxedImproved} applies if $2 \leq p\leq q<\min\left\{p+1,\frac{np}{n-1}\right\}$.
\end{remark}

\subsection{Mixed boundary problems}\label{sec:mixedboundary}

We will conclude this section with a brief discussion of mixed boundary problems of the form
\begin{equation}
  \min_{u} \int_{\Omega} F(\D u) - f \cdot u \,\d x - \int_{\Gamma_N} g_N \cdot u \,\d\H^{n-1},
\end{equation} 
subject to $u = g_D$ on $\Gamma_D,$ where $\Gamma_N, \Gamma_D \subset \p\Omega$ are disjoint such that $\overline\Gamma_D \cup \overline\Gamma_N = \p\Omega.$
For convex functionals with $p$-growth, techniques using first-order difference quotients have been developed in \cite{Savare1998}, and we will employ similar arguments.

The relaxed functional to use in this case is
\begin{align*}
\overline \F_{M}(u)=\inf \left\{\,\liminf_{j\to\infty} \F(v_j): (v_j)\subset Y, v_j\rightharpoonup v \text{ weakly in } X\,\right\}
\end{align*}
where $v\in X= \{u\in \WW^{1,p}(\Omega)\colon \tp{Tr}\, u = g_D \text{ on } \Gamma_D\}$ and $Y=\WW^{1,q}(\Omega)\cap X$. 
Using this definition it is straightforward to adapt the results of Section \ref{sec:relaxedVsRegularised} to this setting.

We wish to employ similar arguments as in the pure Dirichlet and Neumann cases, however we will need to ensure our difference quotients preserve the respective boundary conditions.
This will involve considering domains with piecewise regular boundary, such that the respective boundary conditions are prescribed on the regular portions.
More precisely we will consider \emph{$C^{k,\alpha}$-domains with corners}, in that every $x_0 \in \overline\Omega$ admits a $C^{k,\alpha}$-diffeomorphism to a neighbourhood of the closure of the model space
\begin{equation}
  \R^n_{++} = \{ x \in \bb R^n : x_i > 0 \text{ for all } 1 \leq i \leq n \}.
\end{equation} 
We write $B^{++} = \R^n_{++}\cap B_1(0)$.

Moreover we will impose further restrictions on the boundary components, and assume one of the following two cases are satisfied at each $x_0 \in \overline\Gamma_D \cap \overline \Gamma_N.$
We say that we are in the \emph{flat case}, if after localising, we may use a $C^{1,\alpha}$-flattening map such that $\Gamma_D$ and $\Gamma_N$ are mapped to $\{x_{n-1}>0,\, x_n = 0\}$ and $\{x_{n-1}<0,\, x_n = 0\}$, respectively. We say we are in the \emph{corner case}, if after localising and flattening, we may work in $B_{k,++} := B_1(0)\cap \{x_i>0,\, \forall k\leq i\leq n\}$ and moreover, $\Gamma_D$ is mapped to $\bigcup_{\ell\leq i\leq n}\left(\p B_{++}\cap \{x_i = 0\}\right)$ for some $k, \ell$ with $1 \leq k \leq n$ and $k \leq \ell \leq n+1$ (understanding that $\Gamma_D = \emptyset$ when $\ell = n+1$).
The two cases are illustrated in Figure \ref{fig:mixed_extension} below.
 Note that in the corner case we may take $k = \ell$ or $\ell = n+1$, which corresponds to prescribing purely Dirichlet and Neumann boundary conditions in piecewise regular domains respectively.


Under these assumptions, we obtain the following analogue to Theorem \ref{thm:regularityRelaxed}.

\begin{theorem}\label{thm:relaxed_mixed}
  Let $\Omega$ be a $C^{1,1}$-domain with corners, and let $\Gamma_D, \Gamma_N \subset \p\Omega$ be such that locally we are either in the flat or corner case. Let $F$ satisfy \eqref{def:bounds1}--\eqref{def:bounds3}, with ${1 < p \leq q < \frac{(n+\alpha)p}n}
  .$ Suppose further that  $f \in \BB^{\alpha-1,q^\prime}_{\infty}(\Omega),$ $g_D \in \WW^{1+\alpha-\frac1q,q}(\Gamma_D)$ and ${g_N \in \WW^{\alpha-\frac1{q^\prime},q^\prime}(\Gamma_N)}.$ Then if $u$ is a relaxed minimiser of $\F(\cdot)$ in $\WW^{1,p}(\Omega),$ we have $u \in \BB^{1+\frac{\alpha}{\max\{2,p\}},p}_{\infty}(\Omega)$ together with the estimate
\begin{equation}
  \begin{split}
    &\|u\|_{\BB^{1+\frac \alpha {\max(2,p)},p}_\infty(\Omega)}\\
    &\lesssim \left(1+\overline\F_M(u)+\|f\|_{\BB^{\alpha-1,q^\prime}_{\infty}(\Omega)}^{q^\prime}+ \|g_D\|_{\WW^{1+\alpha-\frac1{q^\prime},q^\prime}(\Gamma_D)}^{q^\prime}+\|g_N\|_{\WW^{\alpha-\frac1{q^\prime},q^\prime}(\Gamma_N)}^{q^\prime} \right)^\gamma.
  \end{split}
\end{equation}
In particular, $u \in \WW^{1,q}(\Omega).$
\end{theorem}

Since our estimates are local in nature, it suffices to consider the case $x_0 \in \overline\Gamma_D \cap \overline\Gamma_N$ where the two boundary conditions meet.
The only aspect in which the proof differs from \cite{Koch2020} and Theorem \ref{thm:regularityRelaxed} is in the construction of the difference quotients; as such we will only outline the necessary changes to treat the mixed case.
Informally we need to ensure the difference quotients do not move the unconstrained regions (the Neumann part) into the constrained regions (the Dirichlet part).
This is achieved by flattening the boundary to reduce to the model case, however it is unclear whether this is really necessary; in particular we must assume that the boundary is $C^{1,1}$-regular to preserve both the H\"older regularity of the coefficients and Besov-regularity of functions under the flattening map, but it will be interesting to understand whether this can be relaxed.

\begin{proof}[Proof of Theorem {\ref{thm:relaxed_mixed}}]
  We will derive estimates to the minimia $u_{\e} \in \WW^{1,q}(\Omega)$ associated to $\F_{\e}(\cdot)$ for each $\e>0,$ and pass to the limit at the end. Upon flattening the boundary, we will consider the flat and corner cases separately. In both cases, we will take a $B_{\infty}^{1+\alpha,p}$ extension $\tilde g$ of $g$ to $B_1(0),$ and consider $v_{\e} = u_{\e} - \tilde g.$ We will need to extend this to the full ball $B_1(0),$ and the respective procedures are illustrated in Figure \ref{fig:mixed_extension}.

\begin{figure}[ht]
    \centering
    \begin{subfigure}{0.45\textwidth}
      \begin{tikzpicture}[scale=.3]
	\begin{pgfonlayer}{nodelayer}
		\node [style=none] (0) at (0, 10) {};
		\node [style=none] (1) at (0, -10) {};
		\node [style=none] (2) at (-10, 0) {};
		\node [style=none] (3) at (10, 0) {};
		\node [style=none] (4) at (0, 0) {};
		\node [style=none] (5) at (-2, 0) {};
		\node [style=none] (8) at (0, 2) {};
		\node [style=none] (9) at (-3, 7) {};
		\node [style=none] (10) at (-7, 3) {};
		\node [style=none] (11) at (5, 5) {\footnotesize $v_{\e}$};
		\node [style=none] (15) at (2, -1.25) {};
		\node [style=none] (17) at (-5, 5) {\footnotesize $v_{\e}$};
		\node [style=none] (18) at (-5, -5) {\footnotesize $\WW^{1,p}$-extension};
		\node [style=none] (19) at (5, -5) {\footnotesize $0$};
		\node [style=none] (20) at (1.7, 10) {\footnotesize $x_{n-1}$};
		\node [style=none] (21) at (-10, 0.5) {};
		\node [style=none] (22) at (10, 0.7) {\footnotesize $x_n$};
		\node [style=none] (23) at (6, 0.7) {\footnotesize $\Gamma_D$};
		\node [style=none] (24) at (-6, 0.7) {\footnotesize $\Gamma_N$};
		\node [style=none] (25) at (0.5, -5) {};
		\node [style=none] (27) at (0.6, -0.6) {\footnotesize $0$};
	\end{pgfonlayer}
	\begin{pgfonlayer}{edgelayer}
		\draw [->,thick](4.center) to (0.center);
		\draw [->,thick](4.center) to (1.center);
		\draw [->,thick](4.center) to (3.center);
		\draw [->,thick](4.center) to (2.center);
		\draw [bend left=45,dashed] (5.center) to (8.center);
    \draw [->](4.center) to (9.center);
    \draw [->](4.center) to (10.center);
	\end{pgfonlayer}
\end{tikzpicture}
    \subcaption{Flat case}
    \end{subfigure}
    \begin{subfigure}{0.45\textwidth}
\begin{tikzpicture}[scale=.3]
	\begin{pgfonlayer}{nodelayer}
		\node [style=none] (0) at (0, 10) {};
		\node [style=none] (1) at (0, -10) {};
		\node [style=none] (2) at (-10, 0) {};
		\node [style=none] (3) at (10, 0) {};
		\node [style=none] (4) at (0, 0) {};
		\node [style=none] (5) at (0, 2) {};
		\node [style=none] (8) at (2, 0) {};
		\node [style=none] (9) at (3, 7) {};
		\node [style=none] (10) at (7, 3) {};
		\node [style=none] (11) at (5, 5) {\footnotesize $v_{\e}$};
		\node [style=none] (15) at (2, -1.25) {};
		\node [style=none] (17) at (-5, 5) {\footnotesize First order};
		\node [style=none] (28) at (-5, 4) {\footnotesize reflection};
		\node [style=none] (18) at (-5, -5) {\footnotesize $0$};
		\node [style=none] (19) at (5, -5) {\footnotesize $0$};
		\node [style=none] (20) at (1.7, 10) {\footnotesize $x_{n-1}$};
		\node [style=none] (21) at (-10, 0.5) {};
		\node [style=none] (22) at (10, 0.7) {\footnotesize $x_n$};
		\node [style=none] (23) at (6, .7) {\footnotesize $\Gamma_D$};
		\node [style=none] (24) at (-6, 0.5) {};
		\node [style=none] (25) at (0.5, -5) {};
		\node [style=none] (26) at (1, 5) {\footnotesize $\Gamma_N$};
		\node [style=none] (27) at (0.6, -0.6) {\footnotesize $0$};
	\end{pgfonlayer}
	\begin{pgfonlayer}{edgelayer}
    \draw [->,thick](4.center) to (0.center);
    \draw [->,thick](4.center) to (1.center);
    \draw [->,thick](4.center) to (3.center);
		\draw [->,thick](4.center) to (2.center);
		\draw [bend left=45,dashed] (5.center) to (8.center);
    \draw [->](4.center) to (9.center);
    \draw [->](4.center) to (10.center);
	\end{pgfonlayer}
\end{tikzpicture}
    \subcaption{Corner case}
    \end{subfigure}
    \caption{Extending $v_{\e}$}
   \label{fig:mixed_extension}
\end{figure}

In the flat case, we will then take a zero extension of $v_{\e}$ to $\{x_{n-1} > 0, x_n < 0\},$ before taking a $\WW^{1,q}$ extension $\tilde v_{\e}$ to the full ball $B_1(0).$
By construction $\tilde v$ agrees with $v$ on $B^+,$ and vanishes on $\Gamma_D$ (and also on $\{x_{n-1} = 0, x_n < 0\}$).
From here we observe that we can take difference quotients in directions $h \in \bb R^n$ such that $h_{n-1} < 0$ and $h_n > 0$ (see Figure \ref{fig:mixed_extension}(a)).
Since $v_h$ vanishes on $\Gamma_D$ we have $T_hv$ is a valid competitor, and we can argue in the same way as before.

In the corner case, we will take a first-order reflection along each $x_i$ for $k \leq i < \ell$ (as in Theorem 4.5.2 in \cite{Triebel1992}, which is a $W^{1,q}$-extension).
This gives a $\WW^{1,q}$ function on $B_{\ell,++}$ vanishing on $\p B_{k,++} \cap B_1(0),$, so we can extend it to the full ball by means of a zero extension to obtain $\tilde v.$
Now we can take difference quotients in directions $h \in \bb R^n$ such that $h_i > 0$ for each $\ell \leq i \leq n$ (see Figure \ref{fig:mixed_extension}(b) for the case $k = n-1$, $\ell = n$), and argue as before to conclude.
\end{proof}

Assuming $F\equiv F_0(x,\lvert z\rvert)$ is a radial integrand and homogeneous boundary conditions, we may also improve on Theorem \ref{thm:relaxedImproved} as in Section \ref{sec:improved}. In addition to assuming \eqref{def:bounds31}, we also found it was necessary to work on piecewise regular domains.
Here the idea is that the reflection procedure in Section \ref{sec:boundaryImproved} can be carried out separately in each $x_i$ direction, taking odd and even reflections for the case of Dirichlet and Neumann boundaries respectively.
We point out that this result also shows our regularity results hold for piecewise $C^{1,1}$ domains, even in the purely Dirichlet and Neumann cases.

\begin{theorem}\label{thm:relaxed_mixedImproved}
  Suppose $1<p\leq q < \infty,$ $\alpha \in (0,1]$.  Let $\Omega$ be a bounded $C^{1,1}$-domain with corners such that locally we are always in the corner case. Suppose $F\equiv F_0(x,\lvert z\rvert)$ satisfies \eqref{def:bounds1}-\eqref{def:bounds31}.
  Let $q < \frac{n+\alpha}n p$ and $f \in \BB^{\alpha-1,p^\prime}_{\infty}(\Omega).$ Assume $g_N = g_D = 0$.
  Then, if $u$ is a relaxed minimiser for $\F$,  for all $\delta \leq \frac{\alpha}2 \min\{2,p^\prime\}$ we have $V_{p,\mu}(\D u) \in \BB^{\delta,2}_{\infty,\loc}(\Omega).$ Further, if $p\geq 2$ and $f\in \BB^{\beta-1,p^\prime}_\infty(\Omega)$ for some $\beta\in[\alpha,2]$ then $V_{p,\mu}(\D u)\in \BB^{\delta,2}_\infty(\Omega)$ for all $\delta\leq \min\left\{\frac{p^{\prime}\beta}{2},\alpha\right\}$.

  Moreover, if a-priori $u \in W^{1,q}(\Omega),$ it suffices to assume $q < \frac{np}{n-\alpha}.$
\end{theorem}

\begin{proof}
  By flattening we can reduce once again to the corner ball $B_{k,++},$ where $\Gamma_D$ is mapped to $\bigcup_{\ell i \leq n}\{x_i = 0\}$ and $\Gamma_N$ is mapped to $\bigcup_{k \leq i < \ell} \{x_i = 0\}$.
  Here $\Gamma_D$ and $\Gamma_D$ are understood to be empty if $\ell = n+1$ and $k = \ell$ respectively.
We know that $u \in \BB^{1+\frac{\alpha}{p},p}_{\infty}(B_{k,++})$ due to Theorem \ref{thm:relaxed_mixed}. 
Now we extend $u$ to the full ball by taking an even reflection along $x_i$ for $k \leq i < \ell,$ and an odd reflection along $x_i$ for $\ell \leq i \leq n.$
From Section \ref{sec:oddeven_extension} we see the Besov regularity is preserved by these extensions, adapting Lemmas \ref{eq:zeroextension} and \ref{lem:zeroNeumann} for the case of an even extension. 
Therefore we may iteratively apply the difference quotient argument from Lemma \ref{lem:mainEstimateImprovedDiff} in each direction, and argue as in the proof of Theorem \ref{thm:relaxedImproved} to conclude.
\end{proof}

\begin{remark}
  In both Theorem \ref{thm:relaxed_mixed} and Theorem \ref{thm:relaxed_mixedImproved}, if $F$ is independent of $x$, due to Corollary \ref{cor:autonomous}, it suffices to assume $2\leq p\leq q<\min\left\{\frac{np}{n-1},p+1\right\}$.
\end{remark}

\begin{remark}\label{rem:mixed_counterexamples}
 In order to obtain Theorem \ref{thm:relaxed_mixedImproved}, we genuinely must assume that $\Gamma_D$ and $\Gamma_N$ meet at a corner.
  This was already observed in \cite{Savare1997} for linear equations, where the function $u(x,y) = \Im(x+iy)^{\frac12}$ is seen to be harmonic on the upper-half plane $\{ \Im z > 0 \},$ and satisfies a mixed Dirichlet-Neumann condition on the boundary ($u$ vanishes on $\{x=0, y>0\},$ and $\p_yu$ vanishes on $\{x=0, y<0\}$).
  Note that $u$ does not lie in $\WW^{2,r}_{\loc}(\bb R^2_+)$ for any $r \geq \frac 43,$ whereas the above results, if they apply, would imply $u$ lies in $\WW^{2,r}_{\loc}(\bb R^2_+)$ for all $1 \leq r < 2.$
  Similarly, we have already seen that $\Gamma_D$ must be sufficiently regular to deduce higher differentiability, so it must map to one of the corner faces under the flattening map.
\end{remark}

\section{Characterisation of regular boundary points}\label{sec:excessEstimate}
In this section, we will establish the following $\eps$-regularity result. 

\epsRegularity*

Throughout this section we will assume the above assumptions hold.
Additionally for $z_0 \in \R^{m\times n},$ we introduce the shifted integrand
\begin{equation}\label{eq:shifted_integrand}
  F_{z_0}(x,z) = F(x,z+z_0) - F(x,z_0) - \partial_zF(x,z_0)z.
\end{equation} 
This satisfies the estimates
\begin{align}
\label{eq:shifted_Fgrowth}
    \lvert F_{z_0}(x,z)\vert &\leq \Lambda_M (1+\lvert z\rvert^2)^{\frac{q-2}2}\lvert z\rvert^2 \\ 
    \lvert \partial_zF_{z_0}(x,z)\vert &\leq \Lambda_M (1 +\lvert z\rvert^2)^{\frac{q-1}2}\lvert z\rvert  \label{eq:shifted_DFgrowth}\\
    \lvert F_{z_0}(x,z)\vert &\geq \lambda_M (1 +\lvert z\rvert^2)^{\frac{p-2}2}\lvert z\rvert^2  \label{eq:shifted_elliptic}\\
    \lvert \partial_z F_{z_0}(x_1,z) - \partial_z F_{z_0}(x_2,z)\rvert &\leq \Lambda_M \lvert x_1 - x_2\rvert^{\alpha} (1+\lvert z\rvert^2)^{\frac{q-1}{2}}. \label{eq:shifted_holder}
\end{align} 
for all $x_1,x_2 \in \overline{\Omega}$ and $z \in \R^{m\times n},$ which follows by distinguishing between when $\lvert z\rvert \leq 1$ and $\lvert z\rvert > 1.$ Here the constants depend on $M>0$ where $\lvert z_0\rvert\leq M.$

While the result as a global $\e$-regularity result is new, the proof is essentially routine.
Following \cite{DeMaria2010}, we will employ a blow-up argument to establish a suitable excess decay estimate. In the case of Dirichlet boundary, one can also argue directly by means of an $\mathcal{A}$-harmonic approximation similarly as in \cite{Kronz2005} and we intend to return to direct arguments, also in the case of Neumann boundary, in future work.

\begin{remark}
  While our $\eps$-regularity results are local in nature, for our arguments to hold it is necessary to impose only one type of boundary condition on $\p\Omega;$ in particular this does not apply to the problems considered in Section \ref{sec:mixedboundary}.
  This is because the proof will make use of the translation invariance $u \mapsto u+c$ of the Neumann problem, which is not available in the mixed setting.
  It is unclear whether this is merely a technical difficulty, but it highlights that we crucially use a global property in the Neumann case.
\end{remark}

We will locally flatten the boundary, following the procedure in Section \ref{sec:flattening}.
In what follows, we will denote $B^+ = B_1(0) \cap \bb R^n_+$ and  ${\Gamma = B_1(0) \cap \{x_n = 0\}}$. 
Finally, for $x_0 \in B^+$ and $R >0$ we will denote ${{B^+_R(x_0) = B_R(x_0)}\cap \bb R^n_+}.$

\subsection{Caccioppoli-type inequalities}\label{sec:caccioppoli}

We now prove the following boundary Caccioppoli estimate in the Neumann case.
\begin{lemma}\label{eq:caccioppoli_neumann}
  In the setting of Theorem \ref{thm:eps_regularity}, suppose that $u\in \WW^{1,p}(\Omega)$ is a relaxed minimiser to \eqref{eq:neumann_problem}. Then there is $R_0>0$ such that for all $M>0,$ if $a \colon \bb R^n \to \bb R^m$ is a affine map such that $\lvert\D a\rvert\leq M$ and $x_0\in \overline\Omega,$ $0<R<R_0$,
\begin{align*}
  \dashint_{\Omega_{R/2}(x_0)} \lvert V_p(\D (u-a)) \rvert^2 \d x  
  &\lesssim \dashint_{\Omega_R(x_0)} \left\lvert V_p\left(\frac{u-a}{R}\right)\right\rvert^2 \d x + \lvert V_{p^\prime}(R^{\alpha}L)\rvert^2\\
  &+ \left(\dashint_{\Omega_R(x_0)} \lvert V_p(\D (u-a))\rvert^2+\left\lvert V_p\left(\frac{u-a}{R}\right)\right\rvert^2 \d x \right)^\frac q p,
\end{align*}
where
\begin{equation}
  L = 1 + \seminorm{g_N}_{C^{0,\alpha}(\p\Omega \cap B_R(x_0))} + \norm{f}_{\LL^{\frac n{1-\alpha}}(\Omega_R(x_0))},
\end{equation} 
and the implicit constant depends on $M, F, \Omega$ and other structural constants indicated in Section \ref{sec:notation}.
\end{lemma}

\begin{proof}
 We will put $\tilde F = F_{\D a}$ for the shifted integrand \eqref{eq:shifted_integrand}, as well as $\tilde u = u-a.$ Note that $\tilde u$ is a relaxed minimiser of the problem
$$
v \mapsto \int_{\Omega} \tilde F(x,\D v) - f \cdot v \,\d x +\int_{\Omega} \partial_zF(x,\D a)\D v \,\d x+\int_{\p\Omega} g \cdot v\,\d \H^{n-1},
$$
where we have written $g = g_N.$
Indeed, if we let $\overline{\F}_{0,N}$ denote the relaxed functional associated to $v \mapsto \int_{\Omega} \tilde F(x,\D v)),$ by definition of the relaxation we see that
\begin{equation}
  \begin{split}
    \overline{\F}_{N}(v) &= \overline{\F}_{0,N}(v-a) - \int_{\Omega} F(x,\D a) \,\d x -\p_zF(x,\D a)\cdot \D (v-a)\,\d x \\
                         &\quad- \int_{\Omega} f \cdot v \,\d x + \int_{\p\Omega} g \cdot v \,\d \H^{n-1}.
  \end{split}
\end{equation} 
Let $R/2<r<s<R/2$, applying Lemma \ref{lem:regularised_extension} to $u$ gives $w$ and $r<r^\prime<s^\prime<s.$ Let $\phi$ be a smooth cut-off supported on $B_{s^\prime}$, with $\phi =1$ on $B_{r^\prime}$ and $\lvert \D\phi\rvert\leq \frac c {s-r}$.
We will set $\tilde w = w-a$ and $\psi = (1-\phi)\tilde w.$
Then by Remark \ref{rem:boundaryConditionExtension} and Lemma \ref{lem:additivityNeumann} applied to $\overline{\F}_{0,N},$ we have
\begin{equation}
  \overline{\F}_{0,N}(\tilde u,\Omega) = \overline{\F}_{0,N}(\tilde u,\Omega_{s^\prime}(x_0),\Omega \cap \p B_{s^\prime}(x_0)) + \overline{\F}_{0,N}(\tilde u,\Omega \setminus B_{s^\prime}(x_0),\Omega \cap \p B_{s^\prime}(x_0)),
\end{equation} 
and a similar statement holds for $\psi.$ 
Further, by lower-semicontinuity of the energy $\int_{\Omega} \lvert V_p(\D v)\rvert^2\,\d x$ and \eqref{eq:shifted_elliptic} we have
\begin{equation}
  \int_{\Omega_{s^\prime}(x_0)} \lvert V_p(\D \tilde u)\rvert^2 \,\d x \lesssim \overline{\F}_{0,N}(\tilde u,\Omega_{s^\prime}(x_0),\Omega \cap \p B_{s^\prime}(x_0)).
\end{equation} 
Hence using the above and minimality of $\tilde u$ we have
\begin{align*}
  \int_{\Omega_{r^\prime}(x_0)} \lvert V_p(\D \tilde u)\rvert^2\,\d x
    &\lesssim \overline{\F}_{0,N}(\tilde u,\Omega_{s^\prime}(x_0),\Omega \cap \p B_{s^\prime}(x_0)) \\
    &\leq \overline{\F}_{0,N}(\psi,\Omega_{s^\prime}(x_0),\Omega \cap \p B_{s^\prime}(x_0))  - \int_{\Omega} f \cdot (\tilde u - \psi) \,\d x \\
    &\quad + \int_{\p\Omega} g\cdot(\tilde u-\psi)\,\d\H^{n-1} +\int_{\Omega_{s^\prime}(x_0)} \partial_z \tilde F(x,\D a) \D(\tilde u - \psi)\, \d x \\
    &= I_1 + I_2 + I_3 + I_4.
\end{align*}
We find using \eqref{eq:shifted_Fgrowth} and \eqref{eq:Vfunction_pq},
\begin{align*}
  \lvert I_1\rvert\lesssim \int_{\Omega_{s^\prime}(x_0)}\lvert V_{q,\mu}(\D\psi)\rvert^2\d x \lesssim \int_{\Omega_{s^\prime}(x_0)} \lvert V_p(\D\psi)\rvert^2+ \lvert V_p(\D\psi)\rvert^\frac{2q}p \,\d x.
\end{align*}
Hence using \eqref{eq:testFunctionEstimates1}, \eqref{eq:testFunctionEstimates2} we can bound
\begin{align*}
  \lvert  I_1 \rvert \leq& \int_{\Omega_s(x_0)\setminus B_r(x_0)} \lvert V_p(\D\tilde u)\rvert ^2+\left\lvert V_p\left(\frac{\tilde u}{s-r}\right)\right\rvert^2\d x\\
                         &\quad + (s-r)^{n\left( 1-\frac qp \right)} \left(\int_{\Omega_s(x_0)\setminus B_r(x_0)} \lvert V_p(\D\tilde u)\rvert^2+ \left\lvert V_p\left(\frac{\tilde u}{s-r}\right)\right\rvert^2\d x\right)^\frac {q} p.
\end{align*}
For $I_2$, we note that $\tilde u - \psi$ vanishes on $\Omega \cap \p B_{s^\prime},$ allowing us to use a Poincar\'e-Sobolev inequality to estimate
\begin{equation}
  \begin{split}
    \lvert I_2\rvert &\leq C \left(\int_{\Omega_{s^\prime}(x_0)} \lvert f\rvert^n \,\d x\right)^{\frac1n} \left(\int_{\Omega_{s^\prime}(x_0)} \lvert \tilde u - \psi\rvert^{\frac{n}{n-1}} \,\d x\right)^{\frac{n-1}{n}} \\
                     &\leq CR^{\alpha} \left(\int_{\Omega_R(x_0)} \lvert f\rvert^{\frac{n}{1-\alpha}} \,\d x\right)^{\frac{\alpha}{n}} \int_{\Omega_{s^\prime}(x_0)} \lvert \D(\tilde u - \psi)\rvert \,\d x.
  \end{split}
\end{equation} 
This is justified by integrating along one of the tangential directions.
To estimate $I_3,$ we use the translation invariance of the problem (due to the compatibility condition \eqref{eq:compatibility_condition}) to replace $\tilde u$ by $\tilde u+b,$ which in turn replaces $\tilde w$ by $\tilde w+b.$ Then we can choose $b \in \bb R$ so that
\begin{equation}\label{eq:psi_cancellation}
  \begin{split}
    0 &= \int_{B_{s^\prime} \cap \p\Omega} (\tilde u+b) - (1-\phi)(\tilde w+b) \,\d \H^{n-1}  \\
    &= \int_{B_{s^\prime} \cap \p\Omega} \tilde u - (1-\phi)\tilde w \,\d x+  b \int_{B_{s^\prime} \cap \p\Omega} \phi \,\d \H^{n-1}.
  \end{split}
\end{equation} 
Using this we can estimate
\begin{equation}
  \begin{split}
    I_3 &= \int_{B_{s^\prime} \cap \p\Omega} \left( g(x) - g(x_0) \right)\cdot (\tilde u - \psi) \,\d\H^{n-1} \\
        &\leq R^{\alpha} \norm{g}_{C^{0,\alpha}(B_R)} \int_{B_{s^\prime} \cap \p\Omega} \lvert\tilde u - \psi\rvert \,\d\H^{n-1} \\
        &\leq R^{\alpha} \norm{g}_{C^{0,\alpha}(B_R)} \int_{\Omega_{s^\prime}(x_0)} \lvert \D(\tilde u -\psi)\rvert \,\d x.
  \end{split}
\end{equation} 
For $I_4$, we use \eqref{eq:shifted_holder}, the divergence theorem, \eqref{eq:psi_cancellation}, and the assumption that $\p\Omega$ is $C^{1,\alpha}$ to bound
\begin{equation*}
  \begin{split}
    I_4 &\leq \int_{\Omega_{s^\prime}(x_0)} \partial_z\tilde F(x_0,\D a) D(\tilde u -\psi) \,\d x + CR^{\alpha} \int_{\Omega_{s'}(x_0)} \lvert \D(\tilde u -\psi)\rvert \,\d x \\
        &= \int_{\partial\Omega \cap \Omega_{s^\prime}(x_0)} \partial_z F(x_0,\D a) \cdot  (\tilde u - \psi) \otimes \mathbf{n}(x) \,\d \H^{n-1} + CR^{\alpha} \int_{\Omega_{s^\prime}(x_0)} \lvert \D(\tilde u -\psi)\rvert \,\d x\\
        &\leq \int_{\partial\Omega \cap B_{s^\prime}} \partial_z F(x_0,\D a) \cdot  (\tilde u - \psi) \otimes \mathbf{n}(x_0) \,\d \H^{n-1} \\
        &\quad + CR^{\alpha} \int_{\Omega_{s^\prime}(x_0)} \lvert \D(\tilde u -\psi)\rvert \,\d x + CR^{\alpha}\int_{\partial\Omega \cap B_{s^\prime}}  \lvert \tilde u -\psi\rvert  \,\d \H^{n-1}\\
        &\lesssim R^{\alpha} \int_{\Omega_{s^\prime}(x_0)} \lvert \D(\tilde u -\psi)\rvert \,\d x.
  \end{split}
\end{equation*}

Now using Young's inequality \eqref{eq:vfunction_young} applied to $\lvert V_p(z)\rvert^2$ we have
\begin{equation}
  \begin{split}
    R^{\alpha} \int_{\Omega_{s^\prime}(x_0)} \lvert \D(\tilde u - \psi)\rvert \,\d x 
    &\lesssim \delta \int_{\Omega_{s^\prime}(x_0)} \lvert V_p(\D(\tilde u -\psi))\rvert^2 \,\d x + C_{\delta}R^{n} \lvert V_{p^\prime}(R^{\alpha})\rvert^2 \\
    &\lesssim \delta \int_{\Omega_{s^\prime}(x_0)}\lvert V_p(\D \tilde u)\rvert^2 + \left\lvert V_p\left(\frac{\tilde u}{s-t}\right)\right\rvert^2 \,\d x + C_{\delta}R^{n} \lvert V_{p^\prime}(R^{\alpha})\rvert^2.
  \end{split}
\end{equation} 
Collecting estimates we have shown that
\begin{equation*}
  \begin{split}
    \int_{\Omega_r(x_0)} \lvert V_p(\D\tilde u)\rvert^2 \,\d x
    &\lesssim \int_{\Omega_s(x_0)\setminus B_r(x_0)}\lvert V_p(\D\tilde u)\rvert^2+ \left\lvert V_p\left(\frac{\tilde u}{s-r}\right)\right\rvert \,\d x\\
    &\quad+ (s-r)^{n\left( 1-\frac qp \right)}\left(\int_{\Omega_s(x_0)\setminus B_r(x_0)} \lvert V_p(\D\tilde u)\rvert^2+\left\lvert V_p\left(\frac{\tilde u}{s-r}\right)\right\rvert \,\d x\right)^\frac q p\\
    &\quad+ \delta \int_{\Omega_{s}(x_0)} \lvert V_p(\D \tilde u)\rvert^2 + \left\lvert V_p\left(\frac{\tilde u}{s-t}\right)\right\rvert^2 \,\d x + C_{\delta}R^n\lvert V_{p^\prime}(R^{\alpha})\rvert^2.
  \end{split}
\end{equation*}
Filling the hole, applying a standard iteration argument (see for instance \cite[Lemma 6.1]{Giusti2003}) and choosing $\delta>0$ sufficiently small to absorb the final gradient term, we deduce that
\begin{equation*}
  \begin{split}
    \int_{\Omega_{R/2}(x_0)} \lvert V_p(\D\tilde u)\rvert^2 \,\d x
    &\lesssim \int_{\Omega_R(x_0)}\lvert V_p(\D\tilde u)\rvert^2+ \left\lvert V_p\left(\frac{\tilde u}{R}\right)\right\rvert \,\d x+ R^n\lvert V_{p^\prime}(R^{\alpha})\rvert^2\\
    &\quad+ R^{n\left( 1-\frac qp \right)}\left(\int_{\Omega_R(x_0)} \lvert V_p(\D\tilde u)\rvert^2+\left\lvert V_p\left(\frac{\tilde u}{R}\right)\right\rvert \,\d x\right)^\frac q p,
  \end{split}
\end{equation*}
which is precisely the claimed inequality when written as averaged integrals.
\end{proof}

The Dirichlet case is essentially the same, and we will only outline the necessary modifications.

\begin{lemma}\label{eq:dirichlet_caccioppoli}
  In the setting of Theorem \ref{thm:eps_regularity}, suppose that $u\in \WW^{1,p}_{g_D}(\Omega)$ is a relaxed minimiser to \eqref{eq:dirichlet_problem} and let $x_0\in \overline\Omega,$ $0<R<R_0$. Then for each $M>0$, if $a \colon \bb R^n \to \bb R^m$ is an affine map such that $\lvert\D a\rvert\leq M$,
\begin{align*}
  \dashint_{\Omega_{R/2}(x_0)} \lvert V_p(\D (u-a)) \rvert^2 \d x  
  &\lesssim \dashint_{\Omega_R(x_0)} \left\lvert V_p\left(\frac{u-a}{R}\right)\right\rvert^2 \d x  + \lvert V_p(R^{\alpha}L)\rvert^2+ \lvert V_{p^\prime}(R^{\alpha}L)\rvert^2\\
  &\quad  + \left(\dashint_{\Omega_R(x_0)} \lvert V_p(\D (u-a))\rvert^2+\left\lvert V_p\left(\frac{u-a}{R}\right)\right\rvert^2 \d x \right)^\frac q p,
\end{align*}
where
\begin{equation}
  L = 1 + \seminorm{\D g_D}_{C^{0,\alpha}(\p\Omega \cap B_R(x_0))} + \norm{f}_{\LL^{\frac n{1-\alpha}}(\Omega_R(x_0))},
\end{equation} 
and the implicit constant depends on $M,F, \norm{\D g_D}_{\LL^{\infty}(\Omega)},\Omega$ and other structural constants.
\end{lemma}

\begin{proof}
 We will choose $R_0>0$ sufficiently small so $B_{R_0}(x_0)$ can be flattened by a $C^{1,\alpha}$-diffeomorphism, which we will use to define the smooth operator from Lemma \ref{lem:regularised_extension}.
 Writing $g = g_D,$ set 
  \begin{equation}
    \tilde g(x) = g(x) - g(x_0) - \D g(x_0) \cdot (x-x_0),
  \end{equation} 
  so we have $\norm{g}_{\LL^{\infty}(\Omega_R(x_0))} \leq CR^{\alpha} \seminorm{\D g}_{C^{0,\alpha}(\Omega_R(x_0))}.$ We then put $\tilde F(x,z) = F_{\D a + \D \tilde g(x)}(x,z).$
  Note that the shifted estimates \eqref{eq:shifted_Fgrowth}--\eqref{eq:shifted_elliptic} continue to hold, however the constants depend on $M \geq \lvert z_0 \rvert + \norm{\D g}_{\LL^{\infty}(\Omega)}.$
  For the H\"older bound we can use \eqref{eq:shifted_DFgrowth} and also local uniform estimates for $\partial_z^2F$ to obtain 
\begin{equation}\label{eq:shifted_holder_boundary}
  \lvert \partial_z \widetilde F(x_1,z) - \partial_z \widetilde F(x_2,z)\rvert \leq \Lambda_M \left(1 + \seminorm{\D g}_{C^{0,\alpha}(\Omega)}\right) \lvert x_1 - x_2\rvert^{\alpha} (1+\lvert z\rvert^2)^{\frac{q}{2}}.
\end{equation} 
  Now for $\frac R2 \leq r < s < R,$ applying Lemma \ref{lem:regularised_extension} to $u-g$ gives $v$ and $r \leq r^\prime < s^\prime  \leq s$ for which the claimed estimates hold.
  We then set $\tilde u = u-a-\tilde g$ and $\tilde w = v - a- (g- \tilde g),$ which satisfies similar estimates as we have only shifted $v$ by an affine map. 
  As before let $\phi$ be a cutoff supported in $B_{s^\prime}$ such that $\chi \equiv 1$ on $B_{r^\prime}$ and $\lvert\D \phi\rvert \leq \frac{C}{s^\prime-r^\prime}.$ Setting $\psi = (1-\phi)\tilde w,$ using \eqref{eq:shifted_elliptic}, Lemma \ref{lem:additivityNeumann}, Remark \ref{rem:boundaryConditionExtension} and minimality of $u$ we have
  \begin{equation}
    \begin{split}
      \int_{\Omega_{t^\prime}(x_0)} \lvert V_p(\D\tilde u)\rvert^2 \,\d x 
      &\lesssim \int_{\Omega_{s^\prime}(x_0)} \widetilde F(x,\D\tilde u) \,\d x \\
      &\lesssim \int_{\Omega_{s^\prime}(x_0)} \widetilde F(x,\D \psi) \,\d x + \int_{\Omega_{s^\prime}(x_0)} f \cdot (\tilde u -\psi) \,\d x \\ 
      &\quad + \int_{\Omega_{s^\prime}(x_0)} \partial_z F(x,\D a + \D \tilde g) \D(\tilde u - \psi) \,\d x \\
      &= I_1 + I_2 + I_3.
    \end{split}
  \end{equation} 
  We can estimate $I_1$ and $I_2$ as before. For $I_3$ we can use the H\"older estimate \eqref{eq:shifted_holder_boundary} and the fact that $\tilde u - \psi$ vanishes on $\partial \Omega_{s^\prime}(x_0)$ to show that
  \begin{equation}
    I_3 \leq C R^{\alpha}\left(1 + \seminorm{\D g}_{C^{0,\alpha}(\Omega)}\right) \int_{\Omega_{s^\prime}(x_0)} \lvert \D(\tilde u -\psi)\rvert \,\d x,
  \end{equation} 
  from which the rest follows as in the Neumann case.
\end{proof}

\subsection{Excess decay estimates}\label{sec:excess_decay}

We will obtain, by means of a blow-up argument, estimates for the \emph{excess energy}
\begin{equation}
  E(x,R) = \dashint_{\Omega_R(x)} \lvert V_p(\D u - (\D u)_{\Omega_R(x)})\rvert^2 \,\d x + R^{2\beta},
\end{equation} 
where $\beta < \alpha$ is fixed.
Throughout this section we will assume we are in the case of a flat boundary where $\Omega_R(x_0) = B_R^+(x_0).$

\begin{lemma}\label{lem:excess_estimate}
  In the setting of Theorem \ref{thm:eps_regularity}, assume $u$ is a relaxed minimiser of \eqref{eq:neumann_problem} or \eqref{eq:dirichlet_problem} in $B_R^+.$
  Then for each $M>0$ and $\beta \in (0,\alpha),$ there is $C_M > 0$ such that for all $\tau \in (0,\frac14),$ there is $\eps > 0$ such that if
  \begin{equation}
    \lvert (\D u)_{B^+_R(x_0)}\rvert \leq M \text{ and } E(x_0,R) < \eps,
  \end{equation} 
  with $x_0 \in B_{1/2}^+$ and $R < \frac12,$ then
  \begin{equation}
    E(x_0,\tau R) \leq C_M \tau^{2\beta} E(x_0,R).
  \end{equation}  
\end{lemma}

\begin{proof}[Proof in the Neumann case]
  We will employ a blow-up argument, which involves several steps. As before write $g = g_N.$

  \textbf{Step 1} (Blow up): Suppose otherwise, so there exists $x_j \in \overline B^+_{1/2}$ and $R_j>0$ such that
  \begin{equation}
    \lvert (\D u)_{B^+_{R_j}(x_j)}\rvert \leq M \text{ and } \lambda_j^2 = E(x_j,R_j) \to 0
  \end{equation} 
  with each $\lambda_j \leq 1,$ and, for $\widetilde C_M > 0$ to be chosen appropriately, we have
  \begin{equation}
    E(x_j,\tau R_j) \geq \widetilde C_M \tau^{2\beta} \lambda_j^2.
  \end{equation} 
  To simplify notation, put $B_j^+ = B^+_{R_j}(x_j),$ $\Gamma_j = \Gamma \cap B_{R_j}(x_j)$, and introduce in addition  ${\widetilde B_j^+ = R_j^{-1}(B_j^+-x_j) = \{y\in B_1\colon x_j+R_j y\in B_j^+\}},$ $\widetilde \Gamma_j = \{y\in B_1\colon x_j + R_j y\in \Gamma\}$ for the rescaled versions. Setting $A_j = (\D u)_{B_j^+}$, $b_j= (u)_{B_j^+}$ and $a_j(y) = b_j + R_j A_j y$, for some $b\in \R^n$, we consider the rescaled sequence
  \begin{equation}
    v_j(y) = \frac{u(x_j+R_jy) - a_j(y)}{\lambda_j R_j},
  \end{equation} 
  defined on $\widetilde B_j^+.$
  By definition of $\lambda_j$ we have
  \begin{equation}
    \dashint_{\widetilde B_j^+} \frac{\lvert V_p(\lambda_j\D v_j)\rvert^2}{\lambda_j^2} \,\d y + \frac{R_j^{2\beta}}{\lambda_j^2} \leq 1,
  \end{equation} 
  so in particular we have
  \begin{equation}
    \begin{cases}
      \hfill \dashint_{\widetilde B_j^+} \lvert \D v_j\rvert^p  \,\d y \leq C &\text{ if } p \leq 2, \\
      \dashint_{\widetilde B_j^+} \lvert \D v_j\rvert^2 + \lambda_j^{p-2} \lvert \D v_j\rvert^p  \,\d y \leq C &\text{ if } p > 2.
    \end{cases}
  \end{equation} 
  Note that $B^+\subset \widetilde B^+_j\subset B$ for any $j$ and $(v_j)_{\widetilde B^+_j} = 0,$ so by Lemma \ref{lem:precise_poincaresobolev} we deduce that $\{v_j\}$ is bounded in $\WW^{1,\min\{p,2\}}(B^+).$
  Also, by passing to a subsequence we can assume that $x_j \to x_0,$ and moreover that $x_0 \in \Gamma;$ indeed since $R_j^{2\beta} \leq\lambda_j^2 \to 0,$ if $x_0 \in B^+$ we have $B_j^+ = B_{R_j}(x_j) \subset B^+$ for $j$ sufficiently large, and we can apply the interior argument from \cite{DeMaria2010} to derive a contradiction. Also by passing to a subsequence we have $A_j\to A_0$ for some $A_0$.

  \textbf{Step 2} (Extremality of $v_j$):
  We introduce the rescaled integrands
  \begin{equation}
    F_j(y,z) = \lambda_j^{-2} F_{A_j}(x_j+R_jy, \lambda_j z),
  \end{equation} 
  where $F_{A_j}$ is given by \eqref{eq:shifted_integrand}. It is convenient to introduce $f_j(y) = R_j\lambda_j^{-1}f(x_j+R_jy)$ and ${g_j(y) = \lambda_j^{-1}g(x_j+R_jy)}$. Then by Lemma \ref{lem:relaxedEuler}, $v_j$ satisfies
  \begin{equation}
    \begin{split}
    &\int_{\widetilde B^+_j} \partial_zF_j(y,\D v_j)\cdot\D\varphi_j\,\d x  \\
    & \quad = - \frac1{\lambda_j}\int_{\widetilde B_j^+} \p_zF(x_j+R_jy,A_j) \cdot \D \varphi_j \,\d x + \int_{\widetilde B_j^+} f_j \cdot \varphi_j - \int_{\widetilde\Gamma_j^\prime}  g_j \cdot \varphi_j \,\d\H^{n-1} \\
    & \quad = I_1 + I_2 + I_3
    \end{split}
  \end{equation} 
  for all $\varphi_j \in \WW^{1,\infty}(\widetilde B_j^+)$ such that $\varphi_j = 0$ on $\p\widetilde B_j^+ \setminus \widetilde\Gamma_j.$ We will additionally require $\int_{\widetilde\Gamma_j} \varphi_j \,\d\H^{n-1}=0,$ so we can estimate
  \begin{equation}
    I_1 = \frac1{\lambda_j}\int_{\widetilde B^+_j} \left( \p_zF(x_j+R_jy,A_j) - \p_zF(x_j,A_j) \right) \cdot \D\varphi \,\d y \lesssim \frac{R_j^{\alpha}}{\lambda_j} \int_{\widetilde B^+_j} \lvert \D \varphi\rvert \,\d y,
  \end{equation} 
  where we have used the $\alpha$-H\"older continuity of $F.$ For the second term we apply the Sobolev inequality to estimate
  \begin{equation}
    \begin{split}
      I_2 &\lesssim \norm{f_j}_{\LL^n(\widetilde B^+_j)} \left(\int_{\widetilde B^+_j} \lvert\varphi\rvert^{\frac{n}{n-1}}\,\d y\right)^{\frac{n-1}{n}}\\
          &\lesssim \frac{R_j^{\alpha}}{\lambda_j} \norm{f}_{\LL^{\frac{n}{1-\alpha}}(B^+)}\int_{\widetilde B^+_j} \lvert \D\varphi\rvert\,\d y.
    \end{split}
  \end{equation} 
  More precisely we used Lemma \ref{lem:precise_poincaresobolev}, noting this gives an extra term which we estimate by $\lvert (u)_{\widetilde B^+_j}\rvert \leq C \norm{\D_1 u}_{\LL^1(\widetilde B^+_j)},$ integrating along one of the tangential directions.
  Finally, for $I_3$, we use the cancellation condition and Gagliardo's trace theorem (noting $\widetilde\Gamma_j$ is flat) to estimate
  \begin{equation}
    I_3 = \frac1{\lambda_j} \int_{\widetilde\Gamma_j} \left( g(x_j+R_jy) - g(x_j) \right) \cdot \varphi \,\d\H^{n-1} \lesssim \frac{R_j^{\alpha}}{\lambda_j} \int_{\widetilde\Gamma_j} \lvert\varphi\rvert \,\d \H^{n-1} \lesssim \frac{R_j^{\alpha}}{\lambda_j} \int_{\widetilde B^+_j} \lvert \D \varphi\rvert \,\d y,
  \end{equation} 
  Combining the estimates we conclude that
  \begin{equation}\label{eq:vj_extremal}
    0 \leq \int_{\widetilde B^+_j} \partial_zF_j(y,\D v_j)\cdot\D\varphi_j\,\d x + \frac{CR_j^{\alpha}}{\lambda_j} \int_{\widetilde B^+_j} \lvert \D \varphi_j\rvert\,\d y.
  \end{equation} 

  \textbf{Step 3} (Limit is harmonic): Passing to a subsequence, we obtain a weak limit $v_j \rightharpoonup v$ in $\WW^{1,\min\{2,p\}}(B^+,\bb R^N).$ We will show this satisfies a constant-coefficient equation. Let $\varphi \in \WW^{1,\infty}(\overline B^+,\bb R^N)$ with $\varphi = 0$ on $B^+\cap \R^{N-1}_+$. We claim that we may extend $\varphi$ to $\varphi_j$ defined on $\widetilde B^+_j$ so that $\varphi_j = 0$ on $\widetilde B^+_j\setminus \widetilde\Gamma_j$ and moreover $\|\varphi_j\|_{\WW^{1,\infty}(\widetilde B^+_j)}\lesssim \|\varphi\|_{\WW^{1,\infty}(B^+)}$ as well as $\int_{\widetilde\Gamma_j}\varphi_j = 0$, where the implicit constant is independent of $j$. In addition, we may ensure that $\varphi_j\to\varphi$ in $\WW^{1,\infty}(B^+) $. We postpone the construction of $\varphi_j$ to Lemma \ref{lem:phij} after the proof.
  
  We now wish to send $j \to \infty$ in \eqref{eq:vj_extremal}, for which we split $\widetilde B^+_j = E_j^+ \cup E_j^-$ where
  \begin{equation}
    E_j^+ = \{ y \in \widetilde B^+_j : \lambda_j \lvert \D v_j\rvert > 1 \},\quad E_j^- = \{ y \in \widetilde B^+_j : \lambda_j \lvert \D v_j\rvert \leq 1 \}.
  \end{equation} 
  By Markov's inequality we have
  \begin{equation}
    \mathscr L^n(E_j^+) \leq \int_{E_j^+} \lvert V_p(\lambda_j \D v_j)\rvert^2 \,\d y \leq C \lambda_j^{\min\{2,p\}},
  \end{equation} 
  noting that $\lambda_j \leq 1.$ Hence since $q \leq p+1,$ putting $r = \min\{2,p\}$ and using H\"older we can bound
  \begin{equation}
    \begin{split}
      &\int_{E_j^+} \p_zF_j(y,\D v_j) \cdot \D\varphi_j  \,\d y \\
\lesssim& \lambda_j^{-1}\mathscr L^n(E_j^+) + \lambda_j^{\frac{qr-r-p}p} \left(\int_{E_j^+} \lambda_j^{p-r}\lvert\D v_j\rvert^p \,\d y\right)^{\frac{q-1}p} \mathscr{L}^n(E_j^+)^{\frac{p-q+1}p} 
      \lesssim  \lambda_j^{r-1}.
    \end{split}
  \end{equation} 
  Hence this vanishes as $j \to \infty.$

  On $E_j^-$ we have
  \begin{equation}
    \int_{E_j^-} \p_zF_j(y,\D v_j)\,\d y 
    = \int_{E_j^-} \int_0^1 \p_z^2F(x_j+R_jy,A_j+t\lambda_j\D v_j) \D v_j \cdot \D \varphi_j \,\d t \,\d y.
  \end{equation} 
  Note that $\chi_{E_j^-} \to \chi_{B^+},$ $\lambda_j \D v_j \to 0$ almost everywhere in $B^+,$ so by local uniform continuity of $\p_z^2 F$ it follows that
  \begin{equation}
    \int_0^1 \p_z^2F(x_j+R_jy,A_j+t\lambda_j\D v_j)\,\d t \longrightarrow \p_z^2 F(x_0,A_0)
  \end{equation} 
  strongly in $\LL^{\frac{r}{r-1}}(B^+)$ by the dominated convergence theorem. Hence as $\{\D\varphi_j \chi_{E_j}\}$ is bounded, converges to $\D\varphi$ almost everywhere in $B_1^+$, and $\D v_j \rightharpoonup \D v$ weakly in $\LL^r(B^+)$ we deduce that
  \begin{equation}
    \lim_{j \to \infty} \int_{E_j^+}  \p_zF_j(y,\D v_j)\cdot \D \varphi_j\,\d y  = \int_{B^+} \p_z^2F(x_0,A_0) \D v \cdot \D \varphi \,\d y.
  \end{equation} 
  Combining with \eqref{eq:vj_extremal} and noting the same holds for $-\varphi,$ we deduce that
  \begin{equation}\label{eq:harmonicV}
    \int_{B^+} \p_z^2F(x_0,A_0) \D v \cdot \D \varphi \,\d y = 0.
  \end{equation} 
  We now claim $v$ lies in $C^1(\overline B_{\tau}^+)$ for each $\tau \in (0,\frac12)$ with the associated estimate
   \begin{equation}\label{eq:harmonic_energyscaling}
     \left( \dashint_{B_{\tau}^+} \lvert V_p(\D v) - (V_p(\D v))_{B_{\tau}^+}\rvert^2 \,\d y\right)^{\frac12} \leq C \tau \left\lvert V_{p,\mu}\left(\dashint_{B_1^+} \lvert \D v - (\D v)_{B_{1}^+}\rvert \,\d y\right)\right\rvert,
  \end{equation} 
  To see this, let $\eps >0$, $\eta_{\eps}$ be a standard mollifier in $\mathbb R^{n-1}$, and set $v_{\eps} = \left(v(\cdot,x_n) \ast \eta_{\eps}\right)(x')$.
  Then $v_{\eps}$ satisfies \eqref{eq:harmonicV} in place of $v$ in $B_{1-\eps}^+$, so by standard interior regularity theory (see for instance \cite[Section 11.1.10]{Hormander1983}) we know that $v_{\eps} \in C^{\infty}(B^+_{1-\eps})$ and that
  \begin{equation}
    -\div \partial_z^2 F(x_0,A_0) D v_{\eps} = 0
  \end{equation} 
  holds in $B_{1-\eps}^+$.
  From the mollification we have $\D_{x'}^k\D v_{\eps} \in \LL^r(B_{1-\eps}^+)$ for each $k$, so by using the equation we inductively find that $\partial_{x_n}^2v_{\eps} \in \WW^{k,r}(B_{1-\eps}^+)$ for each $k \geq 0$.
  In particular $v_{\eps} \in W^{1,2}(B_{1-\eps}^+)$ by Sobolev embedding; note this step is only necessary in the case $p<2$.
  We now can deduce energy bounds by testing $v_{\eps}$ against $\varphi_h = \Delta_{-he_i}^k(\chi^2 \Delta_{he_i}^k v_{\eps}) \in W^{1,\infty}(B_{1-\eps}^+)$ for $k \geq 1,$ $1 \leq i \leq n-1$ and suitable cutoffs $\chi.$
  However this is only admissible provided $\int_{\Gamma} \varphi_h \,\d\H^{n-1} = 0$, for which we use the fact that $u$ is only defined up a constant.
  In particular, changing $u\to u+\lambda_j R_j b$, changes $v_j\to v_j + b$ and hence $v_{\eps}\to v_{\eps}+b$. Note that these changes do not impact any of our estimates.
  Thus, with the choice $b_h=-(\Delta_{-h e_i}^k(\eta^2 \Delta_{he_i}^k v_\e)_{\Gamma})$, $\varphi_h$ is a valid test function, so by standard arguments following \cite[Proposition 2.10]{Carozza1998} we deduce that
  \begin{equation}
    \begin{split}
      \sup_{B_{t(1-\eps)}^+} \lvert \D v_{\eps}\rvert 
      &\leq  \frac{C}{(s-t)^{n}} \left(\int_{B_{(1-\eps)s}^+} \lvert \D v_{\eps}\rvert^2 \,\d x \right)^{\frac12} \\
      &\leq \frac12 \sup_{B_{(1-\eps)s}^+} \lvert \D v_{\eps}\rvert + \frac{C}{(s-t)^{2n}} \int_{B_{(1-\eps)s}^+} \lvert \D v_{\eps}\rvert\,\d x
    \end{split}
  \end{equation} 
  for $0 < t < s <1$. 
  By iteration we can absorb the first term in the second line. 
  Combining this with a uniform estimate for $\D^2v_{\eps}$ we obtain
  \begin{equation}
     \sup_{B_{\frac{1-\eps}2}^+} \lvert \D v_{\eps}\rvert + \sup_{B_{\frac{1-\eps}2}^+} \lvert \D^2 v_{\eps}\rvert \lesssim \left(\int_{B_{1-\eps}^+} \lvert \D v_{\eps}\rvert^2\,\d x\right)^{\frac12} \lesssim \int_{B_{1-\eps}^+} \lvert \D v_{\eps}\rvert \,\d x.
  \end{equation} 
  Finally estimating $\lvert \D V_p(Dv_{\eps})\rvert \lesssim \lvert D^2v_{\eps}\rvert (1+ \lvert \D v_{\eps}\rvert)^{\frac{p-2}2}$, for each $\tau \in (0,\frac12)$ we can estimate
  \begin{equation}
    \begin{split}
      \left( \dashint_{B_{\tau(1-\eps)}^+} \lvert V_p(\D v_{\eps}) - (V_p(\D v_{\eps}))_{B_{\tau(1-\eps)}^+}\rvert^2 \,\d y\right)^{\frac12}  
      &\lesssim  \tau \sup_{B^+_{\frac{1-\eps}2}} \left( \lvert \D^2v_{\eps}\rvert ( 1 + \lvert \D v_{\eps}\rvert)^{\frac{p-2}2}\right) \\
      &\lesssim \tau \left\lvert V_{p,\mu}\left(\dashint_{B_{1-\eps}^+} \lvert \D v_{\eps}\rvert \,\d y\right)\right\rvert.
    \end{split} 
  \end{equation}
  By linearity we can replace $v_{\eps}$ by $v_{\eps} - (\D v_{\eps})_{B_{1-\eps}^+} x$ in the above estimate, after which sending $\eps \to 0$ gives \eqref{eq:harmonic_energyscaling}.


  \textbf{Step 4} (Conclusion): We now consider the affine maps
  \begin{align}
    a_{j,2\tau}(x) &= (u)_{B^+_{2\tau R_j}(x_j)} + (\D u)_{B^+_{2\tau R_j}(x_j)} \cdot (x-x_j). \\
    \tilde a_{j,2\tau}(y) &= (v_j)_{\widetilde B^+_{j,2\tau}} + (\D v_j)_{\widetilde B^+_{j,2\tau}} \cdot y
  \end{align} 
approximating $u$ and $v_j$ respectively, where $\widetilde B^+_{j,2\tau} = R_j^{-1}(B^+_{2\tau R_j}(x_j)-x_j).$ Then by the Caccioppoli-type inequality (Lemma \ref{eq:caccioppoli_neumann}) we have
  \begin{equation}
    \begin{split}
      \frac1{\lambda_j^2} E(x_j,\tau R_j)
      &\lesssim \frac1{\lambda_j^2} \dashint_{B^+_{2\tau R_j}(x_j)} \left\lvert V_p\left( \frac{u-a_{2\tau R_j}}{2\tau R_j} \right) \right\rvert^2 \,\d x + \frac{\tau^{2\beta}R_j^{2\beta}}{\lambda_j^2} + \frac{\lvert V_{p^\prime}(\tau^{\alpha}R_j^{\alpha}L)\rvert^2}{\lambda_j^2}  \\
      &\quad+ \lambda_j^{\frac2p(q-p)} \left( \frac1{\lambda_j^2} E(x_j,2\tau R_j) + \frac1{\lambda_j^2} \dashint_{B^+_{2\tau R_j}(x_j)} \left\lvert V_p\left( \frac{u-a_{2\tau R_j}}{2\tau R_j} \right) \right\rvert^2 \,\d x\right)^{\frac qp}.
    \end{split}
  \end{equation} 
  Note that $\frac{R_j^{2\beta}}{\lambda_j^2} \leq 1,$ and sending $j \to \infty$ we have $$\lambda_j^{-2}\lvert V_{p^\prime}(\tau^{\alpha}R_j^{\alpha}L)\rvert^2 \sim  \lambda_j^{-2} \tau^{2\alpha} R_j^{2\alpha} \to 0.$$ Also by the Poincar\'e inequality (Lemma \ref{lem:precise_poincaresobolev}) we can estimate
  \begin{equation}\label{eq:excess_lowerorder_bound}
    \frac1{\lambda_j^2} \left(\dashint_{B^+_{2\tau R_j}(x_j)} \left\lvert V_p\left( \frac{u-a_{2\tau R_j}}{2\tau R_j} \right) \right\rvert^{2r} \,\d x \right)^{\frac1r}\lesssim \frac1{\lambda_j^2} E(x_j,2\tau R_j)  \lesssim \frac{\tau^{-n}}{\lambda_j^2} E(x_j,R_j) \lesssim 1
  \end{equation} 
  for $1 \leq r \leq \frac{n}{n-1}.$
  Taking $r=1$ and noting that $\lambda_j^{\frac2p(q-p)} \to 0 $ as $j \to 0,$ it follows that
  \begin{equation}
    \limsup_{j \to \infty}\frac1{\lambda_j^2} E(x_j,\tau R_j) \leq \limsup_{j \to \infty} \frac1{\lambda_j^2} \dashint_{\widetilde B^+_{j,2\tau}} \left\lvert V_p\left( \frac{v_j-\tilde a_{j,2\tau}}{2\tau} \right) \right\rvert^2 \,\d x + \tau^{2\beta}.
  \end{equation} 
  To pass to the limit in the first term, observe by \eqref{eq:excess_lowerorder_bound} that the sequence
  \begin{equation}
    \frac1{\lambda_j^2} \left\lvert V_p\left( \frac{v_j-\tilde a_{j,2\tau}}{2\tau} \right) \right\rvert^2 \chi_{\widetilde B_{j,2\tau}}
  \end{equation} 
  is uniformly bounded in $L^{\frac{n}{n-1}}(B_{2\tau})$ and hence is uniformly integrable.
  Now by strong $\LL^p$ convergence we have $v_j \to v$ a.e.\ in $B_{2\tau}^+$ and ${\tilde a_{j,2\tau} \to \tilde a(y) = (v)_{B^+(\tau)} + (\D v_j)_{B^+_{\tau}} \cdot y}$ uniformly in $B_{2\tau}^+,$ so by Vitali's convergence theorem and also \eqref{eq:harmonic_energyscaling} we deduce that
  \begin{equation}
    \limsup_{j \to \infty}\frac1{\lambda_j^2} E(x_j,\tau R_j) \lesssim \dashint_{B_{2\tau}^+} \lvert V_p(\D v - (\D v)_{B^+_{2\tau}})\rvert^2 \,\d y + \tau^{2\beta} \lesssim \tau^{2\beta}.
  \end{equation} 
  Hence there is $C_M>0$ for which
  \begin{equation}
    \widetilde C_M \tau^{2\beta} \leq \frac1{\lambda_j^2} E(x_j,\tau R_j) \leq C_M\tau^{2\beta}
  \end{equation} 
  for $j$ sufficiently large, and so choosing $\widetilde C_M > C_M$ gives the desired contradiction.
\end{proof}

We now prove the following Lemma, which guarantees the existence of maps $\{\varphi_j\}$ with the properties claimed in Step 3 of the previous proof.
\begin{lemma}\label{lem:phij} We use the notation of Lemma \ref{lem:excess_estimate}.
  Let $\varphi \in \WW^{1,\infty}(\overline B^+,\bb R^N)$ with $\varphi = 0$ on $B^+\cap \R^{N-1}_+$. Then there are $\{\varphi_j\}$ so that $\varphi_j = 0$ on $\widetilde B^+_j\setminus \widetilde\Gamma_j$ and $\int_{\widetilde\Gamma_j}\varphi_j \,\d\H^{n-1}= 0$. Moreover, $\varphi_j\to\varphi$ in $\WW^{1,\infty}(B^+) $.
\end{lemma}
\begin{proof}
Consider the cone $C_j$, which is uniquely determined by having sections $\Gamma$ and $\widetilde\Gamma_j$. Denote by $C_j^\prime= C_j\cap \{x_n\leq 1\}$. We construct a family of invertible, uniformly bi-$C^2$-maps $f_j\colon C_j^\prime\cap \{x_n\geq 0\}\to (C_j^\prime\cap \{x_n>0\})\cup (C_j^\prime \cap \widetilde B^+_j)$, such that
  \begin{align*}
    f_j(\Gamma)=\widetilde\Gamma_j \quad f(B_1^+)\subset C_j\cap \widetilde B^+_j \subset f(C_j^\prime)
	\end{align*}
	and further
	\begin{align}\label{eq:convergenceF_j}
	f_j\to \tp{id} \text{ in } C^2(B_1^+) \quad f_j^{-1}\to \tp{id} \text{ in } C^2(f_j(B_1^+)).
	\end{align}
  Once, $\{f_j\}$ is constructed, we extend $\varphi$ by $0$ to $\R^{n-1}_+$ and set
  \begin{align*}
  \varphi_j = \begin{cases}
  	\lvert J f_j\rvert\varphi\circ f_j^{-1} \text{ in } C_j\\
  	0 \text{ else }.
  	\end{cases}
  \end{align*}
  Using the properties of $\{f_j\}$ (see Fig.\ref{fig:1}), we find that
  \begin{align*}
  \|\lvert J f_j\rvert\varphi\circ f_j^{-1}\|_{\WW^{1,\infty}(\widetilde B^+_j)}\lesssim \|\varphi\|_{\WW^{1,\infty}(B^+)}, \qquad \int_{\widetilde\Gamma_j} \lvert f_j\rvert\varphi \circ f_j^{-1} = \int_{\Gamma} \varphi = 0
  \end{align*}
  with implicit constants independent of $j$.
  
\begin{figure}
 \centering
 \resizebox{0.5\textwidth}{!}{
\definecolor{uququq}{rgb}{0.25,0.25,0.25}
\begin{tikzpicture}[line cap=round,line join=round,>=triangle 45,x=1.0cm,y=1.0cm]
\draw[->,color=black] (-7,0) -- (7,0);
\foreach \x in {-7,-6,-5,-4,-3,-2,-1,1,2,3,4,5,6}
\draw[shift={(\x,0)},color=black] (0pt,2pt) -- (0pt,-2pt) node[below] {\footnotesize $\x$};
\draw[->,color=black] (0,-8) -- (0,5);
\foreach \y in {-8,-7,-6,-5,-4,-3,-2,-1,1,2,3,4}
\draw[shift={(0,\y)},color=black] (2pt,0pt) -- (-2pt,0pt) node[left] {\footnotesize $\y$};
\draw[color=black] (0pt,-10pt) node[right] {\footnotesize $0$};
\clip(-7,-8) rectangle (7,5);
\draw [dash pattern=on 4pt off 4pt] (0,0) circle (4cm);
\draw [domain=-7:7] plot(\x,{(-0-0*\x)/4});
\draw (-2.24,-3.32)-- (2.24,-3.32);
\draw [dash pattern=on 4pt off 4pt] (-6.13,4)-- (-2.24,-3.32);
\draw [dash pattern=on 4pt off 4pt] (2.24,-3.32)-- (6.13,4);
\draw [dash pattern=on 4pt off 4pt] (-6.13,4)-- (6.13,4);
\draw [dash pattern=on 4pt off 4pt] (2.24,-3.32)-- (0,-7.53);
\draw [dash pattern=on 4pt off 4pt] (-2.24,-3.32)-- (0,-7.53);
\draw (-4,0)-- (4,0);
\draw [black, xshift=0cm, domain=0:180] plot(\x:4);
\draw (-4,0)-- (-2.24,-3.32);
\draw (2.28,-3.23)-- (4,0);
\draw (-1.48,3.44) node[anchor=north west] {$ \widetilde B^+_j $};
\begin{scriptsize}
\fill [color=uququq] (0,-7.53) circle (1.5pt);
\draw[color=uququq] (0.74,-7.52) node {$(0,-a))$};
\draw[color=black] (1.74,-3.) node {$\widetilde\Gamma_j$};
\fill [color=uququq] (0,-3.32) circle (1.5pt);
\draw[color=uququq] (0.66,-3.06) node {$(0,-h_j)$};
\draw[color=black] (2.22,0.3) node {$\Gamma$};
\end{scriptsize}
\end{tikzpicture}
}
\caption{$\widetilde B^+_j$ and $C_j^\prime$}\label{fig:1}
\end{figure}  
  
  It remains to construct $\{f_j\}$.
  If $\widetilde B^+_j=B^+$, we may set $f_j = \tp{id}$. Hence we may assume $h_j = x_{j,n}>0$. Note that then $\widetilde B^+_j = B_1\cap \{x_n>-h_j\}$.
 It is straightforward to check that $C_j$ is centred at $-a$ where $a=\frac{h}{1-\sqrt{1-h^2}}$ and has aperture $\alpha = \frac{1-\sqrt{1-h^2}}{h}$. We set, writing $x=(x^\prime,x_n)$,
  $$
  f_j(x^\prime,x_n)= \left(\frac{a-h+(1+h)x_n}{x_n+a}\,x^\prime,\left(-h+(1+h)x_n\right)\right).
  $$
  We find
  $$
  f_j^{-1}(y^\prime,y_n)=\left(\frac{\frac{y_n+h}{1+h}+a}{a+y_n}\,y^\prime,\frac{y_n+h}{1+h}\right)
  $$
  Note that as $h\to 0$, $a\to\infty$. Thus, ensuring that $h_j<c_0$, we may assume that $a-h>1$. Then $x_n+a>1$ and $a+y_n>1$, so that $f_j$ and $f_j^{-1}$ are smooth. Moreover, using again that $a\to\infty$ as $h\to 0$, it is a straightforward calculation to check that the desired convergence \eqref{eq:convergenceF_j} holds.
\end{proof}
 
We will conclude by sketching the necessary modifications to the proof of Lemma \ref{lem:excess_estimate} in the Dirichlet case.

\begin{proof}[Proof of Lemma \ref{lem:excess_estimate} in the Dirichlet case]
  We will establish a decay estimate for the modified excess
  \begin{equation}
    \widetilde E(x_0,R) = \dashint_{\Omega_R(x_0)} \lvert V_p(\D u - \D g - (\D u)_{\Omega_R(x_0)} + \D g(x_0))\rvert^2 \,\d x + R^{\beta}.
  \end{equation} 
  We suppose otherwise, then we can find $x_j \in \overline B_{1/2}^+$ and $R_j>0$ such that
  \begin{equation}
    \lvert (\D u)_{B_{R_j}^+(x_j)}\rvert \leq M , \text{ and } \lambda_j = \widetilde E(x_j, R_j) \to 0
  \end{equation} 
  with each $\lambda_j \leq 1,$ and for $\widetilde C_M>0$ to be chosen we have
  \begin{equation}
    \widetilde E(x_j,\tau R_j) \geq \widetilde C_M \tau^{2\beta} \lambda_j^2.
  \end{equation} 
  Let $\tilde g(x) = g(x) - g(x_0) - \D g(x_0) \cdot (x-x_0),$ we will consider the rescaled sequence
  \begin{equation}
    v_j(y)  = \frac{u(x_j+R_jy) - a_j(y) - \tilde g(x_j + R_jy)}{\lambda_j R_j}.
  \end{equation} 
  Arguing as in the Neumann case, we see that $v_j$ is bounded in $W^{1,\min\{2,p\}}(B^+),$ and we can assume that $x_j \to x_0 \in \Gamma$ and that $R_j \to 0.$
  We now use the fact that each $v_j$ is extremal with respect to the shifted integrand
  \begin{equation}
    F_j(y,z) = \lambda_j^{-2} F_{A_j + D\tilde g(x)}(x_j + R_jy,\lambda_jz),
  \end{equation} 
  to deduce the estimate
  \begin{equation}
    0 \leq \int_{\widetilde B_j^+} \p_z F_j(y,\D v_j) \cdot D\varphi_j \,\d x + \frac{CR^{\alpha}_j}{\lambda_j} \int_{\widetilde B_j^+} \lvert D \varphi_j\rvert \,\d x
  \end{equation} 
  for any $\varphi_j \in W^{1,\infty}_0(\widetilde B_j^+).$
  Here we are using Lemma \ref{lem:relaxedEuler}. This will involve estimating
  \begin{equation}
    \frac1{\lambda_j^2}\int_{\widetilde B_j^+} \p_z F(x_0+R_jy,\lambda_j\D a + \lambda_j\D \tilde g(x)) \cdot \D \varphi_j \,\d x \leq \frac{CR^{\alpha_j}}{\lambda_j} \int_{\widetilde B_j^+} \lvert D\varphi_j\rvert \,\d x,
  \end{equation} 
  using \eqref{eq:shifted_holder_boundary} and the fact that $\varphi_j$ vanishes on $\p\widetilde B_j^+.$
  The remainder of the proof is analogous to the Neumann case. We note that for $\varphi \in W^{1,\infty}_0(B^+),$ we can simply extend it by zero to $\widetilde B_j^+$ to obtain $\varphi_j,$ from which we obtain the limit map $v$ is harmonic.
  Moreover since $v$ is affine on $\Gamma,$ we deduce the same the decay estimate \eqref{eq:harmonic_energyscaling} as in the Neumann case using results in \cite[Chapter 10]{Giusti2003}.
  Finally applying Lemma \ref{eq:dirichlet_caccioppoli} we can argue analogously as in Step 4 to obtain a contradiction.
  Hence we infer a decay estimate of the form
  \begin{equation}
    \widetilde E(x_0,\tau R) \leq C_M \tau^{2\beta} \widetilde E(x_0,R).
  \end{equation} 
  The corresponding estimate for $E(x_0,R)$ follows by \eqref{eq:vfunction_additive} and noting that
  \begin{equation}
    \dashint_{B_R} \lvert V_p(\D \tilde g)\rvert \,\d x \leq C \lvert V_p(R^{\beta})\rvert^2 \leq C R^{2\beta}
  \end{equation} 
  for $R \leq 1.$
\end{proof}

\subsection{Iteration of excess and conclusion} 

We can now conclude by a standard iteration argument. Note that we can treat the Dirichlet and Neumann problems simultaneously, since we established the same excess decay estimate.

\begin{proof}[Proof of Theorem {\ref{thm:eps_regularity}}]
  Let $x_0$ and $R>0$ be such that
  \begin{equation}
    \lvert(\D u)_{B^+_R(x_0)}\rvert \leq M, \text{ and } E(x_0,R) < \eps,
  \end{equation} 
  where $\eps>0$ is to be specified. Then for any $x \in B_{R/2}^+(x_0)$ we have
  \begin{equation}
    \lvert(\D u)_{B^+_{R/2}(x)}\rvert \leq 2^{n+1}M, \text{ and } E(x,R/2) \leq C_*(n,p) \eps,
  \end{equation} 
  using standard properties of $V$-functions. Choosing $\eps_0>0$ to be as in Lemma \ref{lem:excess_estimate} associated to $2^{n+2}M$ and some $\tau \in (0,\frac14),$ we require $C_*(n,p)\eps < \eps_0,$ so it follows that
  \begin{equation}
    E(x,\tau R/2) \leq C_M \tau^{2\beta} E(x,R/2) \leq C_M C_*(n,p)\tau^{2\beta} \eps.
  \end{equation} 
  By shrinking $\eps>0$ further, if necessary, we claim that
  \begin{equation}
    \lvert(\D u)_{B^+_{\tau^{k}R/2}(x)}\rvert \leq 2^{n+2}M, \text{ and } E(x,\tau^{k}R/2) \leq C_M \tau^{2k\beta} E(x,R/2)
  \end{equation} 
  for all $k \geq 0.$ This can be checked inductively, by using Jensen's inequality to estimate
  \begin{equation}
    \lvert (\D u)_{B^+_{\tau^{k+1}R/2}(x)} - (\D u)_{B^+_{\tau^kR/2}(x)}\rvert \leq \tau^{-n} e_p^{-1}\left( E(x,\tau^kR/2) \right),
  \end{equation} 
  where $e_p(t) = (1+t^2)^{\frac{p-1}2}t^2.$ Noting that $e_p^{-1}(t) \sim \sqrt{t}$ for $t>0$ sufficiently small, we can estimate
  \begin{equation}
    \lvert (\D u)_{B^+_{\tau^kR/2}(x)}\rvert \leq 2^{n+1}M + C\sigma^{-n}  \sqrt{\eps}\sum_{i=0}^{\infty} \sigma^{i\beta},
  \end{equation} 
  which is less than $2^{n+2}M$ if $\eps>0$ is sufficiently small. This shows that
  \begin{equation}
    E(x,r) \leq C \left( \frac{r}{R} \right)^{\beta}
  \end{equation} 
  for all $x \in B^+_{R/2}(x_0)$ and $0<r<R,$ verifying the Campanato-Meyers characterisation of H\"older continuity (see for instance \cite[Theorem 2.9]{Giusti2003}).
\end{proof}

To deduce Corollary \ref{eq:dimension_estimates}, straightening the boundary and employing a covering argument, it suffices to consider the case where $\Omega=B_1^+$.
We note that if $v\in \WW^{\theta,p}(\Omega,\R^N)$, we have, c.f. \cite{Mingione2003},
 \begin{align}\label{eq:dimensionEstimate}
& \tp{dim}_{\mathscr{H}} \left\{x\in\Omega\colon \limsup_{r\to 0} \lvert(v)_{\Omega_r(x)}\rvert = \infty \right\}\leq n-\theta p\\
&\tp{dim}_{\mathscr{H}} \left\{x\in\Omega\colon \liminf_{r \to 0} \dashint_{\Omega_r(x)} \left\lvert V_p(\D u - (\D u)_{\Omega_r(x)}) \right\rvert^2 > 0\right\}\leq n-\theta p,
 \end{align}
 Combining the results of Section \ref{sec:relaxed} with Corollary \ref{cor:singularPoints} the assertion follows, as does Theorem \ref{thm:regularBoundaryPoints} and Theorem \ref{thm:nonHomogeneousNeumann}.

\vfill
\pagebreak

\subsection*{Declarations}

CI was supported by the Engineering and Physical Sciences Council [EP/L015811/1].

\subsection*{Acknowledgments}

The authors would like to thank Jan Kristensen for the many helpful discussions and comments, and Lars Diening for suggesting the papers \cite{EbmeyerLiuSteinhauer2005,Ebmeyer2005}.
In addition the authors are grateful to the referees, for their careful reading of the manuscript and their helpful suggestions.

\subsection*{Conflict of interest}

The authors declare that they have no conflict of interest.

\appendix

\bibliographystyle{siam}
\bibliography{../../bibtex/Neumann}
\end{document}